\documentclass{amsart}

\usepackage{graphicx}%
\usepackage{multirow}%
\usepackage{amsmath,amssymb,amsfonts}%
\usepackage{amsthm}%
\usepackage{mathrsfs}%
\usepackage[title]{appendix}%
\usepackage{textcomp}%
\usepackage{manyfoot}%
\usepackage{booktabs}%
\usepackage{algorithm}%
\usepackage{algorithmicx}%
\usepackage{algpseudocode}%
\usepackage{listings,bm}%
\usepackage{subfigure}
\usepackage{hyperref}

\newtheorem{theorem}{Theorem}[section]
\newtheorem{lemma}[theorem]{Lemma}

\theoremstyle{definition}
\newtheorem{definition}[theorem]{Definition}

\theoremstyle{plain}
\newtheorem{question}{Question}[section]
\newtheorem{hypothesis}{Hypothesis}[section]
\newtheorem{remark}{Remark}[section]
\newtheorem{problem}{Problem}[section]

\numberwithin{equation}{section}

\begin{document}

\title[existence of explicit symplectic
integrators]{The existence of explicit symplectic
integrators for general nonseparable Hamiltonian systems}

%    Remove any unused author tags.

%    author one information
\author{Lijie Mei}
\address{School of Mathematics \& Yunnan Key Laboratory of Modern Analytical
Mathematics and Applications, Yunnan Normal University, Kunming 650500, China}
%\curraddr{}
\email{bxhanm@126.com}
%\thanks{}

%    author two information
\author{Xinyuan Wu}
\address{Department of Mathematics, Nanjing University, Nanjing 210093, China}
%\curraddr{}
\email{xywu@nju.edu.cn}
%\thanks{}

\author{Yaolin Jiang}
\address{School of Mathematics and Statistics, Xi'an Jiaotong
University, Xi'an 710049, China}
%\curraddr{}
\email{yljiang@mail.xjtu.edu.cn}

\thanks{This work was supported by the National Natural Science
Foundation of China (grant No. 12163003 and 12271426) and the
Yunnan Fundamental Research Projects (grant No. 202401CF070033).}

%\thanks{The second author}

\subjclass[2020]{Primary 65P10; Secondary 37M15}
%Secondary 65L05

\keywords{explicit symplectic integrators, nonseparable Hamiltonian systems,
linear error growth}

\date{}

%\dedicatory{}

\begin{abstract}
The existence of explicit symplectic integrators for general
nonseparable Hamiltonian systems is an open and important problem in
both numerical analysis and computing in science and engineering, as
explicit integrators are usually more efficient than the implicit
integrators of the same order of accuracy. Up to now,
all responses to this problem are negative. That is, there exist
explicit symplectic integrators only for some special nonseparable
Hamiltonian systems, whereas the universal design involving explicit
symplectic integrators for general nonseparable Hamiltonian systems
has not yet been studied sufficiently. In this paper, we present a
constructive proof for the existence of explicit symplectic
integrators for general nonseparable Hamiltonian systems via finding
explicit symplectic mappings under which the special submanifold of
the extended phase space is invariant. It turns out that the
proposed explicit integrators are symplectic in both
the extended phase space and the original phase space. Moreover, on
the basis of the global modified Hamiltonians of the proposed
integrators, the backward error analysis is made via a parameter
relaxation and restriction technique to show the linear growth of
global errors and the near-preservation of first integrals. In
particular, the effective estimated time interval is nearly the same
as classical implicit symplectic integrators when applied to (near-)
integrable Hamiltonian systems. Numerical experiments with a
completely integrable nonseparable Hamiltonian and a nonintegrable
nonseparable Hamiltonian illustrate the good long-term behavior and
high efficiency of the explicit symplectic integrators proposed
and analyzed in this paper.
\end{abstract}

\maketitle

\section{Introduction}\label{sec:intro}
 The Hamiltonian system is one of the most important
dynamical systems and plays the major role in the development of the
theory of classical dynamical systems
\cite{MacKay1987,Giorgilli2022}. As a matter of fact, as
famously claimed by Schr\"{o}dinger \cite{Schrodinger1944}, ``The
Hamiltonian principle has become the cornerstone of modern physics
$\ldots$ If you wish to apply modern theory to any particular
problem, you must start putting the problem in the Hamiltonian
form''. Besides, great theoretical progress such as the
Kolmogorov--Arnold--Moser theory has been made for Hamiltonian
systems, and numerical algorithms specially designed to
solve  Hamiltonian problems  in ``the Hamiltonian way'' were created
\cite{Feng1991}, which constitutes another important progress for
Hamiltonian systems \cite{Sanz-Serna1994a,Sanz-Serna1994,Shang1999,
Hairer2006,Feng2010,Gauckler2018}.

It has become a common view in the field of numerical analysis and
scientific computing that preserving the symplectic geometric
structure probably brings some advantages such as high accuracy and
long-term reliability in the numerical simulation of Hamiltonian
systems. A systematic study of symplectic geometric algorithms for
Hamiltonian problems appeared in 1980s, and was pioneered by Feng
Kang. It was believed that geometric notion of symplecticness is
essential to design geometric integrators for the long-time behavor
of the solutions.

Since explicit methods can be implemented at low cost,
the design and analysis of explicit symplectic methods for
Hamiltonian systems have received a great deal of attention in the
last few decades. Consequently, this paper is concerned with
explicit symplectic geometric integrators for solving
 the $2d$-dimensional canonical Hamiltonian system
\begin{equation}\label{Hamilton}
\begin{aligned}
\frac{\mathrm{d}p}{\mathrm{d}t}=-\frac{\partial H}{\partial q},
\qquad
\frac{\mathrm{d}q}{\mathrm{d}t}=+\frac{\partial H}{\partial p},
\end{aligned}
\end{equation}
where $p=(p^1,\ldots,p^d)^\intercal\in \mathbb{R}^{d}$ and
$q=(q^1,\ldots,q^d)^\intercal\in\mathbb{R}^{d}$ are respectively
generalized momentums and positions conjugate to each other,
$\frac{\partial}{\partial p}$ and $\frac{\partial}{\partial q}$
denote the partial derivatives respectively with respect to $p$ and
$q$, and $H(p,q)$ is the Hamiltonian function. The Hamiltonian
system is of $d$ degrees of freedom.

Let $T^*\mathbb{R}^{d}$ denote the cotangent bundle of
$\mathbb{R}^{d}$. If $T^*\mathbb{R}^{d}$ is equipped with
the standard symplectic structure
\begin{equation}\label{symp-stru}
\omega_{2d}:=\mathrm{d}p\wedge\mathrm{d}q=
\sum_{i=1}^{d}\mathrm{d}p^i\wedge\mathrm{d}q^i,
\end{equation}
then $(T^*\mathbb{R}^{d},\omega_{2d})$ becomes a $2d$-dimensional
symplectic manifold \cite{Arnold1989}. Let
$$J_{2d}=\left(
          \begin{array}{cc}
            O & I_d \\
            -I_d & O \\
          \end{array}
        \right)\in \mathbb{R}^{2d\times 2d}
$$
be the standard skew-symmetric matrix, where $I_d$ is
the $d\times d$ identity matrix. A mapping $\psi:\
T^*\mathbb{R}^{d}\rightarrow T^*\mathbb{R}^{d}$ is called symplectic
if and only if it holds that
\begin{equation}\label{symplectic-cert1}
\mathrm{d}p\wedge\mathrm{d}q=\mathrm{d}\tilde{p}\wedge
\mathrm{d}\tilde{q}
\end{equation}
or equivalently
\begin{equation}\label{symplectic-cert2}
\Big(\dfrac{\partial(\tilde{p},\tilde{q})}{\partial(p,q)}\Big)
^\intercal J_{2d}\Big(\dfrac{\partial(\tilde{p},\tilde{q})}
{\partial(p,q)}\Big)=J_{2d}
\end{equation}
for arbitrary $q\in\mathbb{R}^{d}$ and $p\in T^*\mathbb{R}^{d}$,
where $(\tilde{p},\tilde{q})=\psi(p,q)$ and
$\dfrac{\partial(\tilde{p},\tilde{q})}{\partial(p,q)}$ is the
Jacobian matrix of $(\tilde{p},\tilde{q})$ with respect to $(p,q)$.
Thus, a matrix $M\in\mathbb{R}^{2d}$ is called symplectic if
the equality $M^\intercal J_{2d}M=J_{2d}$ holds.

Suppose that the Hamiltonian system \eqref{Hamilton} is solved
on the interval $[0,T_{end}]$ and prescribed with the initial
conditions
\begin{equation}\label{initial-cod}
p(0)=p_0, \quad q(0)=q_0.
\end{equation}
Let $\phi_t:~(p_0,q_0)\mapsto\big(p(t),q(t)\big)$ be the phase flow
of the Hamiltonian system \eqref{Hamilton}, then $\phi_t$ is
symplectic for each fixed $t$. A numerical method applied to
\eqref{Hamilton} with the constant stepsize $h$ can be viewed as
a discrete flow
\begin{equation*}\label{discrete-flow}
\Psi_h:~(p_n,q_n)\mapsto(p_{n+1},q_{n+1}),
\end{equation*}
where $(p_n,q_n)$ are numerical solutions obtained by the numerical
method at time $t_n=nh~$ for $n=1,\ldots,N_T$, and $N_T=T_{end}/h$.
The numerical method is called symplectic if its discrete flow
$\Psi_h$ is also symplectic.

For solving the Hamiltonian system \eqref{Hamilton}, many classical
numerical methods, such as the Euler method, the
Runge--Kutta (RK) method, and the multistep method,
could be applied. However, as realized by Feng Kang \cite{Feng2010},
``different representations for the same physical law can lead to
different computational techniques in solving the same problem,
which can produce different numerical results''. Since the
Hamiltonian system admits the symplectic structure
\eqref{symp-stru}, a natural idea is constructing numerical
integrators for Hamiltonian systems in the framework of symplectic
geometry, which builds the concept of symplectic (geometric)
algorithms and opens the research area of structure-preserving or
geometric integrators
\cite{Sanz-Serna1994,Hairer2006,Feng2010,Qin2011,Gauckler2018,Iserles2018}.
Up to now, symplectic methods have been developed and adopted in
many fields \cite{Antonana2022,
Aubry1998,David2023,Jin2020,Leimkuhler1994,
Li2020,Libert2011,Mei2017,Pan2021,Sharma2020,Zhu2020,Xiong2020,Wu2003,Wolski2018}.

As claimed in \cite{Hairer2006}, pioneering work on symplectic
integration is due to De Vogelaere \cite{Vogelaere1956} (see also
\cite{Skeel2020}), Ruth \cite{Ruth1983}, and Feng Kang
\cite{Feng1985}. The symplectic Euler method derived by De Vogelaere
\cite{Vogelaere1956} is considered to be the first symplectic
integrator, which is explicit for separable Hamiltonian systems but
implicit for nonseparable Hamiltonian systems. Up to 1983, besides
verifying the result by De Vogelaere that the leap-frog method is
symplectic,  Ruth \cite{Ruth1983} reported a third-order explicit
symplectic integrator for the separable Hamiltonian
$H(p,q)=\frac{1}{2}p^2 + V(q)$. Thereafter, much effort
spent on the explicit symplectic integrators has been made, based on
operator splitting for separable Hamiltonian systems.  Forest \&
Ruth \cite{Forest1990} proposed the fourth-order symplectic
integrator for the  separable Hamiltonian $H(p,q)=T(p) + V(q)$.
Yoshida \cite{Yoshida1990} proposed high-order symplectic
integrators based on the symmetric composition. According to the
particular formulation $H(p,q)=H_0(p,q)+\varepsilon H_1(q)$, the
integrability of $H_0(p,q)$, and the small magnitude of
$\varepsilon$, Wisdom \& Holman \cite{Wisdom1991} proposed another
type of symplectic integrators for the N-body problem of the Solar
System. Moreover, Farr\'{e}s et al. \cite{Farres2013} showed that
pseudo-high-order symplectic methods with fewer stages could be
obtained with the presence of the small parameter $\varepsilon$ for
the Solar System. Omelyan et al. \cite{Omelyan2002} studied the
construction of high-order force-gradient symplectic methods that
could only use positive stepsizes. Recently, Wu et al. \cite{Wu2022}
proposed explicit symplectic integrators based on the multiple
integrable decomposition for the Hamiltonian considered in
Schwarzschild and Kerr-type black hole spacetimes.

For general (separable or nonseparable) Hamiltonian systems, the
generating function method, which was proposed and systematically
developed by Feng Kang \cite{Feng2010}, is one of the main
approaches to the construction of symplectic integrators. It can be
proved that for separable Hamiltonian systems, the explicit
symplectic integrator derived from the generating
function method is identical to that constructed based on operator
splitting. However, for nonseparable Hamiltonian systems, the
symplectic integrators derived by the generating function method are
all implicit. In addition, although the B-series theory
\cite{Chartier2010} is used for the derivation, the higher-order
improved symplectic integrators based on modified differential
equations proposed by Chartier et al. \cite{Chartier2010} could also
be categorized into the generating function method.

In addition to the generating function method, Lasagni
\cite{Lasagni1998}, Sanz-Serna \cite{Sanz-Serna1988} and Suris
\cite{Suris1988} independently studied the symplectic conditions of
traditional RK methods and led to the profound discovery that
Gaussian collocation methods are symplectic for Hamiltonian systems.
Sun \cite{Sun1993,Sun2000} studied the symplectic partitioned RK
(PRK) method in detail and revealed the relationship between
symplectic RK methods and symplectic PRK methods. In spite of the
existence of explicit symplectic PRK methods for separable
Hamiltonian systems, they are not more general than composition or
splitting methods mentioned above \cite{Hairer2006}. More
importantly, the results of the literature
\cite{Feng2010,Hairer2006,Suris1988,Sun1993,Sun2000} illustrated
that there does not exist explicit symplectic RK method or explicit
symplectic PRK method for general nonseparable Hamiltonian systems.
Moreover, for the second-order separable Hamiltonian system
$q''=-\nabla U(q)$, the explicit symplectic
Runge--Kutta--Nystr\"{o}m \cite{Okunbor1992} method is exactly
identical to the operator-splitting explicit symplectic integrators
\cite{Sanz-Serna1994}.

Although the variational integrator \cite{Suris1990,Marsden2001}
using the discretized versions of Hamilton's principle provides an
alternative to derive symplectic integrators, Theorem~XI.6.4 of
\cite{Hairer2006} claims that the symplectic variational integrator
is identical to the symplectic PRK method. Hence, the variational
integrator cannot produce new symplectic integrators beyond the
symplectic RPK method, and explicit symplectic
variational integrators do not exist for nonseparable Hamiltonian
systems.

For some special nonseparable Hamiltonian systems, such as the
post-Newtonian (PN) Hamiltonian system of compact objects, the
Hamiltonian has a typical decomposition form of an integrable
principal part with a nonintegrable perturbation. In this case, the
symplectic integrator mixing an explicit integrator for the
integrable part with an implicit integrator for the nonintegrable
perturbation studied by Lubich et al. \cite{Lubich2010}, Zhong et
al. \cite{Zhong2010}, and Mei et al. \cite{Mei2013} has a high
computational efficiency compared with the completely implicit
symplectic integrator. Of course, the mixing symplectic integrator
is still implicit due to the iteration. Recently, Mei \& Huang
\cite{Mei2024} proposed using an explicit nonsymplectic method
instead of the implicit symplectic method to solve the nonintegrable
perturbation and adapted the mixing symplectic integrator to
explicit near-symplectic integrators. Numerical results showed that
the explicit near-symplectic scheme performs almost the same as the
implicit mixing symplectic integrator based on the presence of the
small PN parameter. However, the proposed near-symplectic method is
essentially nonsymplectic and thus does not have the advantages of
symplectic integrators in long-term numerical simulations. It is
true that there are also some researches on explicit symplectic
integrators for nonseparable Hamiltonian systems
\cite{Blanes2002,McLachlan2004,Chin2009,Tao2016a}. However, the
explicit symplectic integrators discussed in these researches are
only applicable to some specific types of nonseparable Hamiltonians,
and not to general nonseparable Hamiltonian systems.

All the previously mentioned symplectic integrators are constructed
in the original phase space $T^*\mathbb{R}^{d}$. An alternative
approach to deriving possible symplectic integrators was proposed by
Pihajoki \cite{Pihajoki2015} that constructs symplectic schemes for
the extended Hamiltonian in an extended phase space. It was shown
that the proposed methods could be symplectic in the extended phase
space and display a good numerical performance. Liu et al.
\cite{Liu2016} and Luo et al. \cite{Luo2017a,Luo2017b,Luo2021,
Luo2023} applied such methods to PN Hamiltonian systems and obtained
some good numerical results. To suppress the divergence of the two
copied numerical solutions, Tao \cite{Tao2016b} proposed adding a
mixed-up part to the extended Hamiltonian. Numerical results showed
that Tao's methods could provide a provable pleasant long-time
performance, while Pihajoki's methods are only accurate for a short
time. We note that both Pihajoki \cite{Pihajoki2015} and Tao
\cite{Tao2016b} did not claim the symplecticity of their methods in
the original phase space but only the symplecticity in the
extended phase
space. Recently, Jayawardana \& Ohsawa \cite{Jayawardana2023}
proposed a semiexplicit scheme by combining the symmetric projection
method \cite{Hairer2000}, and proved the symplecticity in both the
original phase space and the extended phase space. Ohsawa
\cite{Ohsawa2023} further proved the near preservation of the first
integrals of such methods. However, the method proposed by
Jayawardana \& Ohsawa \cite{Jayawardana2023} is essentially still
implicit because of the use of iterations, and the computational
efficiency significantly decreases with  {increases} of the
stepsize.

It follows from the  discussion stated above that there is no
effective approach to deriving explicit symplectic integrators for
general nonseparable Hamiltonian systems
\cite{McLachlan1992,McLachlan1993,Yoshida1993,Chin2009} and whether
there exist explicit symplectic integrators for general nonseparable
Hamiltonian systems remains an open problem in the field of
structure-preserving algorithms. In this paper, we aim to provide an
approach to designing explicit symplectic integrators for
nonseparable Hamiltonian systems with the help of extended phase
space. We note that once the original Hamiltonian is extended, a
completely integrable Hamiltonian may become nonintegrable in the
extended phase space. As the linear growth of global errors and the
near preservation of the first integrals of symplectic integrators
require the (near-) integrability \cite{Calvo1995,Hairer1997},
whether the two advantages of symplectic integrators hold for the
proposed explicit symplectic integrators or not will be essential,
which is another main concern of this paper. This paper thus marks
an introductory foray towards the development of explicit symplectic
integrators. To this end, the remainder of this paper is organized
as follows.

In Section~\ref{sec:exist-symplectic}, we summarize the existing
symplectic integrators to show the infeasibility of the existing
approaches to deriving explicit symplectic integrators for
nonseparable Hamiltonian systems. In Section~\ref{sec:existence},
using the linear symplectic transformation (symplectic matrix) from
the extended phase space to its special submanifold, we show the
existence of explicit symplectic integrators. In
section~\ref{sec:symp-projection}, taking into account the nonlinear
symplectic translation mapping, we propose the standard projection
symplectic integrator and the generalized projection symplectic
integrator that are both explicit and symplectic. Meanwhile, we show
that some of Tao's symplectic integrators and Pihajoki's original
phase space integrators are symplectic in the original phase space
as well. In Section~\ref{sec:error}, we first show the existence of
the global modified Hamiltonian for some special explicit symplectic
integrators proposed in this paper and then prove the linear growth
of global errors and the near preservation of the first integrals.
The short-term good performance of Tao's and Pihajoki's methods is
also explained. In Section~\ref{sec:num}, we numerically test the
explicit symplectic integrators with a completely integrable
nonseparable Hamiltonian and a nonintegrable nonseparable
Hamiltonian to illustrate the theoretical results presented in this
paper. Conclusions are drawn in the last section.

For the sake of convenience in  subsequent discussions,
the $4d$-dimensional symplectic manifold
$(T^*\mathbb{R}^{2d},\omega_{4d})$ is defined as follows:
\begin{equation*}\label{extend-manifold}
\begin{aligned}
&T^*\mathbb{R}^{2d}:=\Big\{(p,x,q,y)\in\mathbb{R}^{4d}
\big|(q,y)\in\mathbb{R}^{2d},~(p,x)\in T^{*}_{(q,y)}\mathbb{R}^{2d}\Big\},
\\
&\omega_{4d}:=\mathrm{d}p\wedge\mathrm{d}q + \mathrm{d}x\wedge\mathrm{d}y
=\sum_{i=1}^{d}\big(\mathrm{d}p^i\wedge\mathrm{d}q^i +
\mathrm{d}x^i\wedge\mathrm{d}y^i\big).
\end{aligned}
\end{equation*}
On noting the isomorphisms from the fibre $T^*_q\mathbb{R}^{d}$ of
the cotangent bundle $T^*\mathbb{R}^{d}$ at the point $q$ to
$\mathbb{R}^{d}$, we do not distinguish between
$T^*_q\mathbb{R}^{d}$ and $\mathbb{R}^{d}$ in the remaining
contents. The similar case also occurs for
$T^*_q\mathbb{R}^{2d}$ and $\mathbb{R}^{2d}$, $T^*\mathbb{R}^{d}$
and $\mathbb{R}^{2d}$,   and $T^*\mathbb{R}^{2d}$ and
$\mathbb{R}^{4d}$, because of
$T^*_q\mathbb{R}^{2d}\cong\mathbb{R}^{2d}$,
$T^*\mathbb{R}^{d}\cong\mathbb{R}^{2d}$, and
$T^*\mathbb{R}^{2d}\cong\mathbb{R}^{4d}$. In addition, we use
$\|\cdot\|$ to denote the two-norm throughout this
paper, if there is no ambiguity.

\section{Some symplectic integrators in the literature}\label{sec:exist-symplectic}
In this section, we summarize the main approaches to deriving
symplectic integrators in the literature, including the explicit
splitting methods for separable Hamiltonian systems, the symplectic
methods based on generating functions, and the symplectic RK methods
for general separable or nonseparable Hamiltonian systems. Because
of the equivalence between  the symplectic variational integrators
 and the symplectic PRK methods, we will not include the details on the
symplectic variational integrator, which can be found in \cite{Hairer2006}.
\subsection{Symplectic integrators derived in the original phase space}
\label{sec:symp-original}

\subsubsection{Splitting methods for separable Hamiltonian systems}
\label{sec:split}

Suppose that the Hamiltonian $H(p,q)$ is separable with the particular
separation \cite{Ruth1983,Forest1990}
\begin{equation}\label{separation-TV}
H(p,q)=T(p) + V(q),
\end{equation}
where $T(p)$ and $V(q)$ usually respectively represents the kinetic
energy and the potential energy of the separable system.

Let $X_H$ be the Hamiltonian vector field of $H(p,q)$ defined by the Poisson brackets
on the tangent bundle $T\mathbb{R}^d$ as follows:
\begin{equation}\label{vector-field}
X_H=\{\,\cdot\,,H\}=\frac{\partial H}{\partial p}\frac{\partial }{\partial q}
-\frac{\partial H}{\partial q}\frac{\partial }{\partial p}.
\end{equation}
We now split the Hamiltonian \eqref{separation-TV} into two
sub-Hamiltonians $T(p)$ and $V(q)$, and then the vector
field $X_H$ is also split into two sub-vector fields:
\begin{equation}\label{v-field-split}
X_H = X_T +X_V,
\end{equation}
where $X_T=\frac{\partial T}{\partial p}\frac{\partial }{\partial q}$ and
$X_V=-\frac{\partial V}{\partial q}\frac{\partial }{\partial p}$.
The canonical equations corresponding to $T(p)$ and $V(q)$ are respectively
\begin{equation}\label{sub-eq1}
\begin{aligned}
\frac{\mathrm{d}p}{\mathrm{d}t}=X_T(p)=0,
\qquad
\frac{\mathrm{d}q}{\mathrm{d}t}=X_T(q)=T_p(p),
\end{aligned}
\end{equation}
and
\begin{equation}\label{sub-eq2}
\begin{aligned}
\frac{\mathrm{d}p}{\mathrm{d}t}=X_V(p)=-V_q(q),
\qquad
\frac{\mathrm{d}q}{\mathrm{d}t}=X_V(q)=0,
\end{aligned}
\end{equation}
where $T_p=\frac{\partial T}{\partial p}$ and $V_p=\frac{\partial V}{\partial q}$.
The solutions of \eqref{sub-eq1} and \eqref{sub-eq2} can be explicitly
expressed by the phase flows $\exp(tX_T)$ and $\exp(tX_V)$ as follows:
\begin{equation}\label{flow1}
\begin{aligned}
\exp(tX_T):\ (p_0,q_0)\mapsto \big(p_0, q_0 + tT_p(p_0)\big),
\end{aligned}
\end{equation}
and
\begin{equation}\label{flow2}
\begin{aligned}
\exp(tX_V):\ (p_0,q_0)\mapsto \big(p_0 - tV_q(p_0),q_0\big).
\end{aligned}
\end{equation}

The general splitting procedure is to select suitable coefficients
$\alpha_1$, $\beta_1$, $\alpha_2$, $\beta_2,\ldots,\alpha_m$, $\beta_m$
such that the exponential mapping
\begin{equation}\label{splitting}
\Phi_h:=\exp(\beta_m hX_V)\circ\exp(\alpha_m hX_T)
\circ\cdots\circ\exp(\beta_1 hX_V)\circ\exp(\alpha_1 hX_T)
\end{equation}
is an $(r+1)$-th order approximation to the phase flow $\exp(hX_H)$,
i.e., $\Phi_h = \exp\big(hX_H + \mathcal{O}(h^{r+1})\big)$. Then,
$\Phi_h$ denotes an $r$-th order numerical method for
\eqref{separation-TV}, and the numerical solutions defined by
$(p_{n+1},q_{n+1})=\Phi_h(p_n,q_n)$ could be explicitly obtained
using \eqref{flow1} and \eqref{flow2}. Moreover, because the phase
flows $\exp(\alpha_k hX_T)$ and $\exp(\beta_k hX_V)$ are symplectic
mappings, $\Phi_h$ is of course symplectic as it is the composition
of symplectic mappings.

In fact, the above separation \eqref{separation-TV} could be extended to
the general case
\begin{equation}\label{general-separation}
H(p,q) = H_1(p,q) + H_2(p,q),
\end{equation}
where the solutions of both the two integrable
Hamiltonians $H_1(p,q)$ and $H_1(p,q)$ could be explicitly obtained.
For example, Wisdom \& Halman \cite{Wisdom1991} split the
Hamiltonian of the N-body problem in Jacobi coordinates as follows:
\begin{equation}\label{special-separation}
H(p,q) = H_0(p,q) + H_1(q),
\end{equation}
where $H_0(p,q)$ is Kepler motions of $N-1$ bodies across the center
body, and $H_1(q)$ is the intersection of the $N$ bodies' potential
energy that is in a much smaller magnitude than $H_0(p,q)$. Because
the Kepler motion could be solved by the Gauss function, the
separation \eqref{special-separation} corresponds to explicit
symplectic methods.

Another point worth of emphasizing is that we can also split the
Hamiltonian vector field $X_H$ into several integrable parts as:
\begin{equation}\label{several-sep}
X_H=X_1 + X_2 + \cdots +X_k,
\end{equation}
so that the composition of the phase flows $\exp(\alpha_i h X_1)$,
$\exp(\beta_i h X_2)$, $\exp(\gamma_i h X_3),\ldots$ with suitable
coefficients $\alpha_i,~\beta_i,~\gamma_i,~\ldots$ generates
high-order explicit symplectic methods
\cite{McLachlan2002,Hairer2006,Feng2010}. Recently, Wu et. al
\cite{Wu2022} proposed some explicit symplectic methods in
Schwarzschild- and Kerr-type black hole spacetimes just by using the
integrable multiple-part separation.

Although the construction of high-order splitting symplectic methods
can be conducted by solving the algebraic equations directly derived
by the Baker--Campbell--Hausdorff formula
\cite{McLachlan2002,Hairer2006}, the equivalence of splitting
methods and composition methods \cite{Hairer2006} yields another
effective way to derive high-order symplectic methods by composition
of low-order methods. The well-known approach is Yoshida's symmetric
composition strategy \cite{Yoshida1990}. Concerning the
relevant issues on  how to determine the coefficients,
we refer the reader to \cite{Hairer2006}.

The splitting approach becomes the most important way to derive
explicit symplectic methods for Hamiltonian systems. However, this
approach is only applicable to the separable Hamiltonian system.
If the Hamiltonian is nonseparable, then
the forced separation for the Hamiltonian results in either explicit
nonsymplectic methods or implicit symplectic methods.

\subsubsection{Symplectic methods based on generating functions}
\label{sec:generation}
Let
\begin{equation*}
J_{4d}=\left(
          \begin{array}{cc}
            O & I_{2d} \\
            -I_{2d} & O \\
          \end{array}
        \right)\in \mathbb{R}^{4d\times 4d},
\qquad
\widetilde{J}_{4n}=\left(
          \begin{array}{cc}
            J_{2d} & O \\
             O & -J_{2d} \\
          \end{array}
        \right)\in \mathbb{R}^{4d\times 4d}.
\end{equation*}
Define the space $CSp(\widetilde{J}_{4d},J_{4d})$ \cite{Feng2010} as follows:
\begin{equation}\label{CSp}
CSp(\widetilde{J}_{4d},J_{4d}):=\Big\{M\in\mathbb{R}^{4d\times4d}\big|
~\exists\,\mu\in\mathbb{R},~\mu\neq0,~\text{s.t.}~
M^\intercal J_{4d}M=\mu\widetilde{J}_{4d}\Big\}.
\end{equation}
Suppose that $M=\left(
                  \begin{array}{cc}
                    A_\alpha & B_\alpha  \\
                    C_\alpha & D_\alpha \\
                  \end{array}
                \right)\in CSp(\widetilde{J}_{4d},J_{4d})
$,
where $A_\alpha,~B_\alpha,~C_\alpha,~D_\alpha\in\mathbb{R}^{2d\times 2d}$.
Since $M$ is nonsingular, we denote the inverse of $M$ by
$
M^{-1}=\left(
              \begin{array}{cc}
                A^\alpha & B^\alpha  \\
                C^\alpha & D^\alpha \\
              \end{array}
              \right)
$.

Suppose that $(p,q)$ and $(P,Q)$ are the solutions of the
Hamiltonian system \eqref{Hamilton} respectively at time $t$ and
$t+\tau $, and then they are connected by the phase
flow $\exp(\tau X_H)$:
\begin{equation}\label{flow-express}
\exp(\tau X_H):~(p,q)\mapsto (P,Q).
\end{equation}
For convenience, we use the notations $Z=(p,q)^\intercal\in\mathbb{R}^{2d}$
and $\widehat{Z}=(P,Q)^\intercal\in\mathbb{R}^{2d}$.
 The matrix $M$ defines a linear fractional transformation from
 $W$ to $\widehat{W}$
satisfying
\begin{equation*}
\left(
  \begin{array}{c}
    \widehat{W} \\
    W \\
  \end{array}
\right)
=M
\left(
  \begin{array}{c}
    \widehat{Z} \\
    Z \\
  \end{array}
\right),
\end{equation*}
i.e.,
\begin{equation*}
\begin{aligned}
&\widehat{W}=A_\alpha \widehat{Z} + B_\alpha Z,
\quad\widehat{Z}=A^\alpha \widehat{W} + B^\alpha W,
\\
&W=C_\alpha \widehat{Z} + D_\alpha Z,
\quad Z=C^\alpha \widehat{W} + D^\alpha W,
\end{aligned}
\end{equation*}
under the condition $\det(C_\alpha+D_\alpha)\neq0$. According to
\cite[Theorem~5.3.3]{Feng2010}, for sufficiently small $\tau$ there
exists a time-dependent generating function $\phi(W,\tau)$ such that
\begin{equation}\label{generatingF}
\widehat{W} = \nabla\phi(W,\tau),
\end{equation}
and
\begin{equation}\label{HJ-equ}
\frac{\partial \phi}{\partial \tau} = -\mu H\big(A^\alpha\nabla\phi(W,\tau)
+B^\alpha W\big).
\end{equation}
Then \eqref{HJ-equ} is just the general Hamilton--Jacobi equation of
the Hamiltonian system \eqref{Hamilton} corresponding
to the linear transformation $M$.

Furthermore, if $H(Z)$ is analytical on $Z$, then
$\phi(W,\tau)$ can be expressed by a convergent power series in
$\tau$ for sufficiently small $\tau$
\cite[Theorem~5.3.4 in pp.~226]{Feng2010} as follows:
\begin{equation}\label{generatingSeries}
\phi(W,\tau)= \sum_{k=0}^{\infty}\phi^{(k)}(W)\tau^k,
\end{equation}
where the expressions of $\phi^{(k)}(W)$ for
$k=0,1,\ldots$ are determined by the recursive relation
presented in \cite[pp.~226]{Feng2010} or by comparing
powers of $\tau$ via the direct Taylor expansion such as
\cite[pp.~203]{Hairer2006}.

If we truncate the series \eqref{generatingSeries} at the $r$th
order as
\begin{equation}\label{genTrunSeries}
\psi^{(r)}(W,h)= \sum_{k=0}^{r}\phi^{(k)}(W)h^k,
\end{equation}
then the gradient mapping
\begin{equation}\label{gradient}
W\mapsto \widetilde{W}=\nabla\psi^{(r)}(W,h)
\end{equation}
defines an $r$th-order symplectic scheme $Z_n\mapsto Z_{n+1}$ by
\begin{equation}\label{gen-symplectic}
A_\alpha Z_{n+1} + B_\alpha Z_n= \nabla\psi^{(r)}(C_\alpha Z_{n+1}
+ D_\alpha Z_n,h),
\end{equation}
where $Z_n$ is just the numerical solution of \eqref{Hamilton}
at time $t_n=nh$, i.e., $Z_n\approx Z(t_n)$.

The matrix
$
M=\left(
    \begin{array}{cc}
      -J_{2d} & J_{2d} \\
      \frac{1}{2}I_{2d} & \frac{1}{2}I_{2d} \\
    \end{array}
  \right)
$ corresponds to the well known Poincar\'{e} type generating function
$\phi\big(\frac{P+p}{2},\frac{Q+q}{2},t\big)$, whose four leading-power
terms are $\phi^{(0)}(W)=\phi^{(2)}(W)=0$, $\phi^{(1)}(W)=-H(W)$, and
$\phi^{(3)}(W)=\frac{1}{24}(\nabla H)^{\intercal}J_{2d}H_{zz}
J_{2d}\nabla H$. In this case, it derives $W=\frac{1}{2}(Z+\widehat{Z})$
and $\widehat{W}=J_{2d}(Z-\widehat{Z})$.
Then, the second-order symplectic scheme just reads
\begin{equation*}
J_{2d}(Z_{n}-Z_{n+1})=-h\nabla H\Big(\frac{1}{2}(Z_{n}+Z_{n+1})\Big),
\end{equation*}
i.e.,
\begin{equation*}
Z_{n+1}=Z_{n}+hJ^{-1}_{2d}\nabla H\Big(\frac{1}{2}(Z_{n}+Z_{n+1})\Big),
\end{equation*}
which is just the midpoint rule \eqref{midpoint}. By considering
higher-order power terms of $h$, the fourth-order symplectic scheme
in this way reads
\begin{equation}\label{foruth-gen}
\begin{aligned}
Z_{n+1}=Z_{n}&+hJ^{-1}_{2d}\nabla
H\Big(\frac{1}{2}(Z_{n}+Z_{n+1})\Big)
\\
&-\frac{1}{24}h^3J^{-1}_{2d}\nabla\Big((\nabla
H)^{\intercal}J_{2d}H_{zz} J_{2d}\nabla
H\Big)\Big(\frac{1}{2}(Z_{n}+Z_{n+1})\Big).
\end{aligned}
\end{equation}

By letting
\begin{equation}\label{firstMatrix}
M=\left(
    \begin{array}{cccc}
      O & O &-I_{d} & O \\
      O & -I_{d} & O & O \\
      O  & O & O & I_{d} \\
      I_{d}  & O & O & O \\
    \end{array}
  \right),
\end{equation}
we get another widely used generating function in the form
$\phi(q,P,t)$, $W=(q,P)^\intercal$, and $\widehat{W}=-(p,Q)^\intercal$. The
leading-power terms are $\phi^{(0)}(W)=-\frac{1}{2}W^\intercal E_0W$
and $\phi^{(1)}(W)=-H(E_0W)$, where $E_0=\left(
       \begin{array}{cc}
         O & I_d \\
         I_d& O \\
       \end{array}
     \right)
$. Then, $\widehat{W}=\psi^{(1)}(W,h)$ defines the following first-order symplectic scheme
\begin{equation}\label{leftSymEuler}
\left\{
\begin{aligned}
&p_{n+1} = p_{n} - hH_q(p_{n+1},q_n),
\\
&q_{n+1} = q_{n} + hH_p(p_{n+1},q_n),
\end{aligned}
\right.
\end{equation}
which is just the left symplectic Euler method in \eqref{sym-Euler1}.

We note that the high-order symplectic methods based on
the modifying integrator theory of Chartier \cite{Chartier2007} also
attribute to the generating function method because the modified
equation just acts as the generating function as it is the formal
solution of the Hamilton--Jacobi equation. In fact, the fourth-order
modifying integrator of the implicit midpoint rule is just the same
as \eqref{foruth-gen}, while the second-order modifying integrator
of the symplectic Euler method \eqref{leftSymEuler} is the same as
the second-order generating function method generated by
$-(p,Q)=\psi^{(2)}(q,P,h)$ corresponding to the matrix
\eqref{firstMatrix}. More details on the generating function methods
can be found in the monograph \cite{Feng2010}.

On the one hand, according to \eqref{gen-symplectic}, the symplectic
method for general nonseparable Hamiltonian systems based on the
generating functions is explicit if and only if $C_\alpha=0$ which
cannot be satisfied by any matrix $M$ in
$CSp(\widetilde{J}_{4d},J_{4d})$. That is, there does not exist the
explicit symplectic generating function method, at least in the
sense of Feng's methodology for studying ``Geometric
Integration''.

On the other hand, according to \cite[Theorem~VI.5.1]{Hairer2006},
it is known that there exists such a generating function $\phi(Z,t)$ in principle
that only depends on $Z=(p,q)^\intercal$ corresponding to the phase flow
\eqref{flow-express}. However, even if the generating function
$\phi(p,q,t)$ could be expressed explicitly, the relations
connecting $\widehat{Z}=(P,Q)^\intercal$ with $Z=(p,q)^\intercal$ become
\begin{equation*}
P\frac{\partial Q}{\partial p} = \frac{\partial \phi}{\partial p},
\qquad
P\frac{\partial Q}{\partial q} - p = \frac{\partial \phi}{\partial q}.
\end{equation*}
Once the phase flow \eqref{flow-express} is unknown (i.e., the
explicit expressions of $\frac{\partial Q}{\partial p}$ and
$\frac{\partial Q}{\partial q}$ are unknown), then $P$ and $Q$
cannot be explicitly expressed by the generating function
$\phi(p,q,t)$ in terms of $(p,q)$. This point further excludes the
possibility of explicit symplectic methods based on generating
functions.

\subsubsection{Symplectic Runge--Kutta methods}\label{sec:symp-RK}
Using the notations $H_p(p,q)=\frac{\partial H}{\partial p}$ and
$H_q(p,q)=\frac{\partial H}{\partial q}$, the first so-called symplectic Euler
method is duo to De Vogelaere \cite{Vogelaere1956,Skeel2020}:
\begin{equation}\label{sym-Euler1}
\begin{aligned}
&p_{n+1}=p_n - hH_q(p_{n+1},q_n),
\\
&q_{n+1}=q_n+ hH_p(p_{n+1},q_n),
\end{aligned}
\quad \text{or} \quad
\begin{aligned}
&q_{n+1}=q_n+ hH_p(p_n,q_{n+1}),
\\
&p_{n+1}=p_n - hH_q(p_n,q_{n+1}),
\end{aligned}
\end{equation}
which are both of order one. Then, the composition of the two symplectic methods
in \eqref{sym-Euler1} results in the two following St\"{o}rmer--Verlet scheme
\cite{Hairer2003}
\begin{equation}\label{stormer1}
\begin{aligned}
&p_{n+1/2}=p_n - \frac{h}{2}H_q(p_{n+1/2},q_n),
\\
&q_{n+1}=q_n+
\frac{h}{2}\big(H_p(p_{n+1/2},q_n)+H_p(p_{n+1/2},q_{n+1})\big),
\\
&p_{n+1}=p_{n+1/2} - \frac{h}{2}H_q(p_{n+1/2},q_{n+1}),
\end{aligned}
\end{equation}
and
\begin{equation}\label{stormer2}
\begin{aligned}
&q_{n+1/2}=q_n+ \frac{h}{2}H_p(p_n,q_{n+1/2}),
\\
&p_{n+1}=p_n -
\frac{h}{2}\big(H_q(p_n,q_{n+1/2})+H_q(p_{n+1},q_{n+1/2})\big),
\\
&q_{n+1}=q_{n+1/2}+ \frac{h}{2}H_p(p_{n+1},q_{n+1/2}),
\end{aligned}
\end{equation}
which are symplectic and both of order two
\cite[Theorem~VI.3.4]{Hairer2006}.

It is noted that the symplectic Euler methods and
St\"{o}rmer--Verlet methods are explicit when applied
to the separable Hamiltonian $H(p,q)=T(p) + U(q)$. More general
situations for these methods to be explicit can be found in
\cite[pp.~189]{Hairer2006}. However, they are implicit for general
nonseparable Hamiltonian systems.

Another important method is the midpoint rule:
\begin{equation}\label{midpoint}
\begin{aligned}
&p_{n+1}=p_n -hH_q\big(\frac{p_n+p_{n+1}}{2},\frac{q_n+q_{n+1}}{2}\big),
\\
&q_{n+1}=q_n + hH_p\big(\frac{p_n+p_{n+1}}{2},\frac{q_n+q_{n+1}}{2}\big),
\end{aligned}
\end{equation}
which is symplectic, symmetric, and of order two. The midpoint rule is usually
implicit except for the rare case that the Hamiltonian system is linear, whose
solutions can be exactly obtained by the analytic method.

The generalization of  the symplectic Euler methods, the
St\"{o}rmer--Verlet scheme, and the midpoint rule results in the
symplectic RK method
\cite{Lasagni1998,Sanz-Serna1988,Suris1988}
\begin{equation}\label{RK}
\left\{
\begin{aligned}
& P_i=p_n - h\sum_{j=1}^{s}{a}_{ij}H_q(P_j,Q_j), \quad
i=1,\ldots,s,
\\
& Q_i=q_n + h\sum_{j=1}^{s}{a}_{ij}H_p(P_j,Q_j), \quad
i=1,\ldots,s,
\\
& p_{n+1}=p_n - h\sum_{i=1}^{s}{b}_{i}H_q(P_i,Q_i),
\\
& q_{n+1}=q_n+ h\sum_{i=1}^{s}{b}_{i}H_p(P_i,Q_i),
\end{aligned}\right.
\end{equation}
or the symplectic partitioned Runge--Kutta (PRK) method
\cite{Sun1993,Hairer2006,Feng2010}:
\begin{equation}\label{pRK}
\left\{
\begin{aligned}
& P_i=p_n - h\sum_{j=1}^{s}{a}_{ij}H_q(P_j,Q_j), \quad
i=1,\ldots,s,
\\
& Q_i=q_n + h\sum_{j=1}^{s}{\widehat{a}}_{ij}H_p(P_j,Q_j), \quad
i=1,\ldots,s,
\\
& p_{n+1}=p_n - h\sum_{i=1}^{s}{b}_{i}H_q(P_i,Q_i),
\\
& q_{n+1}=q_n+ h\sum_{i=1}^{s}{\widehat{b}}_{i}H_p(P_i,Q_i),
\end{aligned}\right.
\end{equation}
whose coefficients satisfy the symplectic conditions respectively
for the RK method and the PRK method. The well-known symplectic methods are
the Gauss--Legendre collocation methods (in RK type) and the
Lobatto IIIA-IIIB pairs (in PRK type).

It is noted that  high-order symplectic methods obtained by the
composition of the midpoint rule are just diagonally implicit
symplectic RK methods, while the composition of symplectic Euler
methods leads to the diagonally implicit symplectic PRK method (see
\cite{Sanz-Serna1991,Qin1992}, \cite[Theorem~VI.4.4]{Hairer2006}, or
\cite[Theorem~7.1.7]{Feng2010}). Only if the system admits an
integrable separation as $\dot{p}=f(q),~\dot{q}=g(p)$, there exist
explicit symplectic PRK methods. However, as claimed in
\cite[pp.~193]{Hairer2006}, ``such methods are not more general than
composition or splitting methods'',  and hence the explicit
symplectic PRK methods do not generate newer symplectic methods
beyond the symplectic splitting methods in such a separable
Hamiltonian case.

Moreover, according to the relation between symplectic PRK mehtods
and symplectic RK methods \cite{Sun2000}, the existence of explicit
symplectic PRK methods must lead to the existence of explicit
symplectic RK methods. However, as claimed in
\cite[Corollary~7.1.8]{Feng2010}, explicit RK methods cannot satisfy
the symplecticity conditions so that there does not
exist an explicit symplectic RK method. Therefore, it
is concluded that there do not exist explicit symplectic
RK or PRK methods for general nonseparable Hamiltonian
systems.

\subsection{Symplectic extended phase space methods}\label{sec:symp-extended}
\subsubsection{Extended phase space methods}\label{sec:explicit-extended}
To derive explicit integrators, Pihajoki \cite{Pihajoki2015}
extended the phase space of the Hamiltonian $H(p,q)$ from
$(p,q)\in\mathbb{R}^{2d}$ to $(p,x,q,y)\in\mathbb{R}^{4d}$ by adding
a pair of conjugate variables $(x,y)$, which corresponds to the
extended Hamiltonian
\begin{equation}\label{extend-Hamilton}
\Gamma(p,x,q,y)=H_A(p,y) + H_B(x,q),
\end{equation}
where $H_A(p,y)=H(p,y)$ and $H_B(x,q)=H(x,q)$. Then, the canonical
equations of $\Gamma(p,x,q,y)$ read

\begin{equation}\label{extended-equ}
\begin{aligned}
&\frac{\mathrm{d}p}{\mathrm{d}t}=-\frac{\partial\Gamma}{\partial q}
=-\frac{\partial H_B}{\partial q}, \quad
&\frac{\mathrm{d}q}{\mathrm{d}t}=+\frac{\partial\Gamma}{\partial p}
=+\frac{\partial H_A}{\partial p},
\\
&\frac{\mathrm{d}x}{\mathrm{d}t}=-\frac{\partial\Gamma}{\partial x}
=-\frac{\partial H_A}{\partial y}, \quad
&\frac{\mathrm{d}y}{\mathrm{d}t}=+\frac{\partial\Gamma}{\partial y}
=+\frac{\partial H_B}{\partial x}.
\end{aligned}
\end{equation}

If the equations \eqref{extended-equ} are prescribed by the initial
conditions $p(t_0)=x(t_0)=p_0$ and $q(t_0)=y(t_0)=q_0$, then it is
easily verified that the solution $\big(p(t),x(t),q(t),y(t)\big)$ of
\eqref{extended-equ} satisfies $p(t)=x(t)$ and $q(t)=y(t)$.
Moreover, the pair $\big(p(t),q(t)\big)$ is just the solution of the
original Hamiltonian $H(p,q)$ with the initial conditions
\eqref{initial-cod}. In this case, the Hamiltonian $\Gamma(p,x,q,y)$
determined by \eqref{extend-Hamilton} is just a couple of the original
Hamiltonian $H(p,q)$ and thus does not generate more information for
the solution of $H(p,q)$. However, the particular formulation of
$\Gamma(p,x,q,y)$ enables us to construct explicit integrators.

Because $H_A(p,y)$ only involves the variables $p$ and $y$,
its corresponding canonical equations are as follows:
\begin{equation}\label{extend-subA}
\begin{aligned}
&\frac{\mathrm{d}p}{\mathrm{d}t}=0,&
\qquad
&\frac{\mathrm{d}q}{\mathrm{d}t}=+\frac{\partial H_A}{\partial p},&
\\
&\frac{\mathrm{d}x}{\mathrm{d}t}=-\frac{\partial H_A}{\partial y},&
\qquad
&\frac{\mathrm{d}y}{\mathrm{d}t}=0.&
\end{aligned}
\end{equation}
By using the notations  $H_q(p,q)=\frac{\partial
H}{\partial q}$ and $H_p(p,q)=\frac{\partial
H}{\partial p}$, and denoting the phase flow of $H_A(p,y)$ by
$\exp(tX_{H_A})$, we explicitly express the solutions of
\eqref{extend-subA} by:
\begin{equation}\label{ext-subFlowA}
\exp(tX_{H_A}):\ (p_0,x_0,q_0,y_0)\mapsto \big(p_0,x_0-tH_q(p_0,y_0),
q_0 + tH_p(p_0,y_0) ,y_0\big).
\end{equation}
Similarly, the canonical equations of $H_B(x,q)$ are
\begin{equation}\label{extend-subB}
\begin{aligned}
&\frac{\mathrm{d}p}{\mathrm{d}t}=-\frac{\partial H_B}{\partial q}&
\qquad
&\frac{\mathrm{d}q}{\mathrm{d}t}=0,&
\\
&\frac{\mathrm{d}x}{\mathrm{d}t}=0,&
\qquad
&\frac{\mathrm{d}y}{\mathrm{d}t}=+\frac{\partial H_B}{\partial x}.&
\end{aligned}
\end{equation}
whose phase flow $\exp(tX_{H_B})$ could be expressed as
\begin{equation}\label{ext-subFlowB}
\exp(tX_{H_B}):\ (p_0,x_0,q_0,y_0)\mapsto \big(p_0-tH_q(x_0,q_0),x_0,
q_0  ,y_0+ tH_p(x_0,q_0)\big).
\end{equation}

According to \eqref{ext-subFlowA} and \eqref{ext-subFlowB}, the
Hamiltonians $H_A(p,y)$ and $H_B(x,q)$ are exactly and explicitly
solvable, and this means that the Hamiltonian $\Gamma(p,x,q,y)$ is
separable. Then, explicit symplectic methods via the splitting
approach in Section~\ref{sec:split} can be constructed for the
extended Hamiltonian $\Gamma(p,x,q,y)$. For example, the
second-order St\"{o}rmer--Verlet (or Leapfrog) scheme
\begin{equation}\label{Pihajoki-leapfrog}
\begin{aligned}
\Phi_h=\exp(\tfrac{1}{2}hX_{H_A})\exp(hX_{H_B})\exp(\tfrac{1}{2}hX_{H_A}),
\end{aligned}
\end{equation}
could be explicitly written as:
\begin{equation}\label{ext-leapfrog}
\begin{aligned}
&x_{n+1/2}=&&x_{n} - \tfrac{h}{2}H_q(p_n,y_n),
\\
&q_{n+1/2}=&&q_{n} + \tfrac{h}{2}H_p(p_n,y_n),
\\
&p_{n+1}~~=&&p_{n} - H_q(x_{n+1/2} ,q_{n+1/2}),
\\
&y_{n+1}~~=&&y_{n} + H_p(x_{n+1/2} ,q_{n+1/2}),
\\
&x_{n+1}~~=&&{x_{n+1/2}} - \tfrac{h}{2}H_q(p_{n+1},y_{n+1}),
\\
&q_{n+1}~~=&&{q_{n+1/2}} + \tfrac{h}{2}H_p(p_{n+1},y_{n+1}).
\end{aligned}
\end{equation}
High-order explicit integrators are similarly obtained
following the splitting approach described in Section~\ref{sec:split}.

However, it was pointed out in \cite{Pihajoki2015} that
the solutions $(p_n,q_n)$ and $(x_n,y_n)$ may diverge with time,
even though both of them can be regarded as the numerical solutions
of the original Hamiltonian $H(p,q)$. This phenomenon is also
numerically demonstrated by Tao \cite{Tao2016b} and Jayawardana \&
Ohsawa \cite{Jayawardana2023}. The divergence between $(p_n,q_n)$
and $(x_n,y_n)$ destroys the near conservation of $H(p_n,q_n)$,
though the whole Hamiltonian $\Gamma(p,x,q,y)$ is nearly conserved.
Moreover, the divergence may also destroy the symplecticity of the
underlying mapping $(p_n,q_n)\mapsto(p_{n+1},q_{n+1})$ in the
original phase space $T^*\mathbb{R}^{d}$, even though the mapping
$(p_n,x_n,q_n,y_n)\mapsto (p_{n+1},x_{n+1},q_{n+1},y_{n+1})$ defined
by \eqref{ext-leapfrog} is symplectic in the extended phase space
$T^*\mathbb{R}^{2d}$.

To remedy this drawback, Pihajoki \cite{Pihajoki2015} proposed the
idea as follows. We first split the Hamiltonian vectors $X_{H_A}$
and $X_{H_B}$ by $X_{H_A}=X_{A_1} + X_{A_2}$ and $X_{H_B}=X_{B_1} +
X_{B_2}$, where $X_{A_1} =\frac{\partial H_A}{\partial p}$, $X_{A_2}
=-\frac{\partial H_A}{\partial y}$, $X_{B_1} =\frac{\partial
H_B}{\partial x}$, and $X_{B_2} =-\frac{\partial H_B}{\partial q}$.
According to the commutativity between $X_{A_1}$ and  $X_{A_2}$ as
well as $X_{B_1}$ and $X_{B_2}$, the leapfrog method
\eqref{Pihajoki-leapfrog} could be rewritten as
\begin{equation}\label{ext-leapfrog-mod}
\begin{aligned}
&\exp(\tfrac{1}{2}hX_{A_1})\exp(\tfrac{1}{2}hX_{A_2})\exp(\tfrac{1}{2}hX_{B_1})
\exp(\tfrac{1}{2}hX_{B_2})
\\
&\circ\exp(\tfrac{1}{2}hX_{B_2})\exp(\tfrac{1}{2}hX_{B_1})
\exp(\tfrac{1}{2}hX_{A_2})\exp(\tfrac{1}{2}hX_{A_1}).
\end{aligned}
\end{equation}
By introducing the mixing mappings $M_i:~\mathbb{R}^{4d}\rightarrow
\mathbb{R}^{4d}$ for $i=1,2$, the leapfrog method is amended as
follows
\begin{equation}\label{ext-leapfrog-amend}
\begin{aligned}
\Psi_h=&\exp(\tfrac{1}{2}hX_{A_1})\exp(\tfrac{1}{2}hX_{A_2})\exp(\tfrac{1}{2}hX_{B_1})
\exp(\tfrac{1}{2}hX_{B_2})M_1
\\
&\circ\exp(\tfrac{1}{2}hX_{B_2})\exp(\tfrac{1}{2}hX_{B_1})
\exp(\tfrac{1}{2}hX_{A_2})\exp(\tfrac{1}{2}hX_{A_1})M_2,
\end{aligned}
\end{equation}
where
\begin{equation*}
M_i=\left(
      \begin{array}{cccc}
        \alpha_{M_i}I_{d}  & \widetilde{\alpha}_{M_i}I_{d} & O & O \\
        \widetilde{\alpha}_{M_i}I_{d} & \alpha_{M_i}I_{d}  & O & O \\
        O & O & \beta_{M_i}I_{d} & \widetilde{\beta}_{M_i}I_{d} \\
        O & O  & \widetilde{\beta}_{M_i}I_{d} & \beta_{M_i}I_{d} \\
      \end{array}
    \right)\in\mathbb{R}^{4d\times 4d},~i=1,2.
\end{equation*}
Suppose that $(p_n,x_n,q_n,y_n)$ is the numerical solution of
$\Gamma(p,x,q,y)$ obtained by $\Psi_h$ with the initial conditions
$(p_0,p_0,q_0,q_0)$, i.e.,
\begin{equation*}
(p_n,x_n,q_n,y_n)=(\Psi_h)^n(p_0,p_0,q_0,q_0),
\end{equation*}
then the numerical solution $(\widetilde{p}_n,\widetilde{q}_n)$ of
$H(p,q)$ is defined by

\begin{equation*}
(\widetilde{p}_n, \widetilde{q}_n )^\intercal
=P\,( p_n, x_n ,q_n ,y_n)^\intercal,
\end{equation*}
where
\begin{equation*}
P=\left(
    \begin{array}{cccc}
      \alpha_PI_{d} &  \widetilde{\alpha}_PI_{d}  & O & O \\
      O & O &  \beta_PI_{d} & \widetilde{\beta}_PI_{d} \\
    \end{array}
  \right)\in \mathbb{R}^{4d\times 2d}.
\end{equation*}
With suitable choice of $M_i$ and $P$, the numerical solution
$(\widetilde{p}_n,\widetilde{q}_n)$ performs good energy
preservation for the Hamiltonian $H(p,q)$.

Furthermore, Liu et al. \cite{Luo2017b} discussed the choice
of $M_i$ and $P$ and proposed some favored
options. In particular, they presented explicit
integrators with the midpoint permutation, with which the leapfrog
scheme is improved as:
\begin{equation*}
\Psi_h^*=M\exp(\frac{1}{2}hX_{H_A})\exp(hX_{H_B})\exp(\frac{1}{2}hX_{H_A}),
\end{equation*}
where
\begin{equation*}
M=\left(
      \begin{array}{cccc}
       \tfrac{1}{2}I_{d}  &\tfrac{1}{2}I_{d} & O & O \\
       \tfrac{1}{2}I_{d} & \tfrac{1}{2}I_{d}  & O & O \\
        O & O & \tfrac{1}{2}I_{d} & \tfrac{1}{2}I_{d} \\
        O & O  & \tfrac{1}{2}I_{d} & \tfrac{1}{2}I_{d} \\
      \end{array}
    \right)\in\mathbb{R}^{4d\times 4d}.
\end{equation*}
Numerical experiments showed a better performance of $\Psi_h^*$ than $\Psi_h$.

Tao \cite{Tao2016b} provided an alternative to suppress the divergence by extending the
Hamiltonian $H(p,q)$ as
\begin{equation}\label{Tao-ext}
{\Gamma}(p,x,q,y)=H_A(p,y) + H_B(x,q) + H_C(p,x,q,y),
\end{equation}
where
\begin{equation}\label{Tao-Hc}
H_C(p,x,q,y)=\frac{\omega}{2}\big(\|p-x\|^2+\|q-y\|^2\big),
\end{equation}
with some $\omega\in\mathbb{R}^{+}$. Similarly to
\eqref{extend-Hamilton}, the extended Hamiltonian \eqref{Tao-ext}
yields a copy of the exact solutions of $H(p,q)$ with the initial
conditions $p(t_0)=x(t_0)=p_0$ and $q(t_0)=y(t_0)=q_0$. In addition,
because the ODEs corresponding to $H_C(p,x,q,y)$ are linear, the
Hamiltonian $H_C(p,x,q,y)$ is completely integrable and explicitly
solvable.

Using the notation $X_{H_C}=\frac{\partial
H_C}{\partial p}\frac{\partial}{\partial q} -\frac{\partial
H_C}{\partial q}\frac{\partial}{\partial p} +\frac{\partial
H_C}{\partial x}\frac{\partial}{\partial y} -\frac{\partial
H_C}{\partial y}\frac{\partial}{\partial x}$, we
consider the following second-order leapfrog method for
\eqref{Tao-ext}:
\begin{equation}\label{Tao-leapfrog}
\begin{aligned}
&\overline{\Psi}_h=\exp(\tfrac{1}{2}hX_{H_A})\exp(\tfrac{1}{2}hX_{H_B})\exp(hX_{H_C})
\exp(\tfrac{1}{2}hX_{H_B})\exp(\tfrac{1}{2}hX_{H_A}),
\end{aligned}
\end{equation}
which has two more stages than \eqref{Pihajoki-leapfrog}, thereby
leading to a lower computational efficiency than
\eqref{Pihajoki-leapfrog}. Because of the intersection among all the
variable $p$, $x$, $q$, and $y$ in $H_C(p,x,q,y)$, the defect
$(p_n-x_n,q_n-y_n)$ of the numerical solution $(p_n,x_n,q_n,y_n)$
obtained by $\overline{\Psi}_h$ will be largely suppressed. The
integrator $\overline{\Psi}_h$ performs a good near preservation for
${\Gamma}(p_n,x_n,q_n,y_n)$, but not for $H(p_n,q_n)$.

It should be emphasized here that all the
above-mentioned explicit integrators are at most symplectic in the
extended phase space $T^*\mathbb{R}^{2d}$, but we cannot
assure that they are symplectic in the original phase space
$T^*\mathbb{R}^{d}$.

\subsubsection{Semiexplicit symplectic methods}\label{sec:semiexp}
In \cite{Jayawardana2023}, Jayawardana \& Ohsawa combined the
extended phase space method with the symmetric projection to derive
semiexplicit symplectic integrators. We describe the method in
\cite{Jayawardana2023} as follows. Let $A$ be a matrix defined by
\begin{equation}\label{matrixA}
A=\left(
    \begin{array}{cccc}
      I_d & -I_d & O & O \\
      O & O & I_d & -I_d \\
    \end{array}
  \right)\in\mathbb{R}^{4d\times2d},
\end{equation}
and then we have
$A(p,x,q,y)^\intercal=(p-x,q-y)^\intercal$.

Suppose that $\Phi_h$ is an explicit symplectic extended phase space
integrator for the Hamiltonian $\Gamma(p,x,q,y)$ in
\eqref{extend-Hamilton} with the initial conditions
$\big(p(t_0),x(t_0),q(t_0),y(t_0)\big)=(p_0,p_0,q_0,q_0)$. If there
exists such a vector $\varrho\in\mathbb{R}^{2d}$ that
\begin{equation}\label{semiexplicit}
(p_{n+1},p_{n+1},q_{n+1},q_{n+1}) = \Phi_h\big(
(p_{n},p_{n},q_{n},q_{n})+A^\intercal\varrho\big) +
A^\intercal\varrho,
\end{equation}
then $(p_{n+1},q_{n+1})$  is defined as the numerical
solution of the original Hamiltonian $H(p,q)$ at the $n+1$ step,
thereby providing a numerical method
$\widetilde{\Phi}_h$ from \eqref{semiexplicit} for $H(p,q)$ as
follows:
\begin{equation}\label{semiexplicit-orig}
(p_{n+1},q_{n+1})=\widetilde{\Phi}_h(p_{n},q_{n}).
\end{equation}

It is noted that the numerical method determined  by
\eqref{semiexplicit} is just an implicit symmetric projection method
projecting the vector in $T^*\mathbb{R}^{2d}$ onto its submanifold
$\mathcal{N}$:
\begin{equation}\label{manifold}
\mathcal{N}=\Big\{(p,p,q,q)\in T^*\mathbb{R}^{2d}\,\big|\,(p,q) \in
T^*\mathbb{R}^{d}\Big\}\subset T^*\mathbb{R}^{2d}.
\end{equation}
The existence of $(p_{n+1},p_{n+1},q_{n+1},q_{n+1})$ and $\varrho$
for the symmetric projection method \eqref{semiexplicit} can be
assured for sufficiently small stepsize $h$. Thus, the numerical
method $\widetilde{\Phi}_h$ given by \eqref{semiexplicit-orig} is
well defined.

The symplecticity of both the symmetric projection method
\eqref{semiexplicit} and its induced method $\widetilde{\Phi}_h$ is
successfully inherited from $\Phi_h$. Both the direct proof and the
geometric proof for the symplecticity can be found in
\cite{Jayawardana2023}, and we also present an intuitive and simple
proof in this paper by introducing some theorems. Although the
method \eqref{semiexplicit} is symplectic, the induced symplectic
integrator $\widetilde{\Phi}_h$ is implicit because iterations are
needed to find a suitable  $\varrho$. By using Broyden's method for
implicit iterations, Jayawardana \& Ohsawa \cite{Jayawardana2023}
showed that the semiexplicit method $\widetilde{\Phi}_h$ could be as
fast as Tao's explicit extended phase method for sufficiently small
stepsize.

However, with the increase of the stepsize $h$, the number of
iterations in a single stepsize significantly increases and thus
results in a low computational efficiency of the semiexplicit
symplectic integrator $\widetilde{\Phi}_h$. Moreover, even though
the semiexplicit symplectic integrator may be more efficient than
Tao's methods of the same order in some cases because Tao's methods
contain more stages, these semiexplicit symplectic integrators can
never be more  timesaving than Pihajoki's original explicit phase
space methods of the same order in any case with a given stepsize on
noting the fact that the former adds an implicit iteration procedure
based on the latter.

\section{Existence of explicit symplectic integrators}\label{sec:existence}
On the basis of the discussion on the extended phase space method,
our approach to designing explicit symplectic
integrators is inspired by the idea of extended phase
space methods. We first introduce the following lemma regarding the
relation between the symplecticity of the extended phase space and
that of the original phase space.

\begin{lemma}\label{theorem1}
Suppose that the mapping $\Phi:\ T^*\mathbb{R}^{2d}\rightarrow
T^*\mathbb{R}^{2d}$ is symplectic in the extended phase space.
If $\Phi(\mathcal{N})\subset \mathcal{N}$, i.e.,
for any $p,q\in\mathbb{R}^{d}$, there exist $\tilde{p},\tilde{q}\in
\mathbb{R}^{d}$ such that $\Phi(p,p,q,q) =(\tilde{p},\tilde{p},
\tilde{q},\tilde{q})$, then it defines a symplectic mapping
$\widetilde{\Phi} :\  T^*\mathbb{R}^{d}\rightarrow
T^*\mathbb{R}^{d}$ in the original phase space
such that $\widetilde{\Phi}(p,q)=(\tilde{p},\tilde{q})$.
\end{lemma}
\begin{proof}
Let $p,q,x,y\in\mathbb{R}^{d}$ and
$(\hat{p},\hat{x},\hat{q},\hat{y}) =\Phi(p,x,q,y)$. Since $\Phi$ is
symplectic, we have $\mathrm{d}\hat{p}\wedge\mathrm{d}\hat{q} + \mathrm{d}\hat{x}
\wedge\mathrm{d}\hat{y}=\mathrm{d}p\wedge\mathrm{d}q +
\mathrm{d}x\wedge\mathrm{d}y$. It then follows from $\Phi(p,p,q,q)
=(\tilde{p},\tilde{p},\tilde{q},\tilde{q})$ that
$\mathrm{d}\tilde{p}\wedge\mathrm{d}\tilde{q} + \mathrm{d}\tilde{p}
\wedge\mathrm{d}\tilde{q}=\mathrm{d}p\wedge\mathrm{d}q +
\mathrm{d}p\wedge\mathrm{d}q$, i.e., $\mathrm{d}\tilde{p}
\wedge\mathrm{d}\tilde{q}=\mathrm{d}p\wedge\mathrm{d}q$. According
to \eqref{symplectic-cert1}, the mapping $\widetilde{\Phi} :\
T^*\mathbb{R}^{d}\rightarrow T^*\mathbb{R}^{d}$ is certainly
symplectic.
\end{proof}

This lemma just states a fact that if the submanifold $\mathcal{N}$
defined by \eqref{manifold} is invariant under the symplectic
mapping $\Phi$, it then yields a symplectic mapping
$\widetilde{\Phi}$ induced by $\Phi$ in the original phase space
$T^*\mathbb{R}^d$. As we previously analyzed, the phase flow
$\exp(tX_{\Gamma})$ of the Hamiltonian $\Gamma(p,x,q,y)$ actually
satisfies the condition and is a typical example of $\Phi$ in
Lemma~\ref{theorem1}. With this lemma, the existence of explicit
symplectic integrators (in the original phase space) for
nonseparable Hamiltonian systems comes down to finding explicit
symplectic (extended phase space) methods under which the
submanifold $\mathcal{N}$ is invariant.

As we analyzed in Section~\ref{sec:explicit-extended}, although
Pihajoki's or Tao's explicit extended phase space methods are
symplectic in the extended phase space, they do not satisfy
$\Phi(\mathcal{N})\subset\mathcal{N}$ in general. Considering this
point, we raise the following question.

\begin{question}[Existence of explicit
symplectic mapping]\label{question1}
\emph{For an $r$th-order explicit symplectic extended phase space
method
$\Phi_h:~(p_n,p_n,q_n,q_n)\mapsto(\widetilde{p}_{n+1},\widetilde{x}_{n+1},
\widetilde{q}_{n+1},\widetilde{y}_{n+1})$, does there exist an explicit
symplectic projection mapping $M:~T^*\mathbb{R}^{2d}\rightarrow
\mathcal{N}$, i.e., exist $(p_{n+1},q_{n+1})\in T^*\mathbb{R}^{d}$
such that
\begin{equation*}
(p_{n+1},p_{n+1},q_{n+1},q_{n+1})=M(\widetilde{p}_{n+1},\widetilde{x}_{n+1},
\widetilde{q}_{n+1},\widetilde{y}_{n+1}) ,
\end{equation*}
and $(p_{n+1},q_{n+1})$ are explicitly obtained from $\widetilde{p}_{n+1}$,
$\widetilde{x}_{n+1}$, $\widetilde{q}_{n+1}$, and $\widetilde{y}_{n+1}$?}
\end{question}

If the answer to the question is positive, then we can just define
$(p_{n+1},q_{n+1})$ as the solutions of the original Hamiltonian
$H(p,q)$. The underlying numerical method
$\widetilde{\Phi}_h:~(p_n,q_n) \mapsto(p_{n+1},q_{n+1})$  is assured
to be symplectic in the original phase space. Then, the remaining
key issue is to explicitly determine $p_{n+1}$ and $q_{n+1}$ from
$\widetilde{p}_{n+1}$, $\widetilde{x}_{n+1}$, $\widetilde{q}_{n+1}$,
and $\widetilde{y}_{n+1}$ such that there exists the symplectic
mapping $M$ which maps
$(\widetilde{p}_{n+1},\widetilde{x}_{n+1},\widetilde{q}_{n+1},\widetilde{y}_{n+1})$
to $(p_{n+1},p_{n+1},q_{n+1},q_{n+1})$.

As both $(\widetilde{p}_{n+1}, \widetilde{q}_{n+1})$ and
$(\widetilde{x}_{n+1}, \widetilde{y}_{n+1})$ can be regarded as the
numerical solutions of order $r$ for the original Hamiltonian
$H(p,q)$, a natural idea is to define $p_{n+1}$ as the weighted
average of $\widetilde{p}_{n+1}$ and $\widetilde{x}_{n+1}$,
$q_{n+1}$ as the weighted average of $\widetilde{q}_{n+1}$ and
$\widetilde{y}_{n+1}$, and $M$ as a linear symplectic transformation
(symplectic matrix). Keeping this in mind, we present
the following theorem to determine the weight
coefficients and construct the symplectic matrix $M$.

\begin{theorem}\label{theorem2}
For any $p,q,x,y\in\mathbb{R}^{d}$, there exists a
symplectic matrix $M\in \mathbb{R}^{4d\times4d}$ and the
corresponding vectors
$\lambda=(\lambda^1,\ldots,\lambda^d)^\intercal,~
\xi=(\xi^1,\ldots,\xi^d)^\intercal\in \mathbb{R}^{d}$ such that
\begin{equation}\label{weight-add}
M(p,x,q,y)^\intercal = (\tilde{p},\tilde{p},\tilde{q},\tilde{q})^\intercal,
\end{equation}
where $\tilde{p}=\lambda\cdot p  + (1-\lambda)\cdot x
=\big(\lambda^1p^1+(1-\lambda^1)x^1,\ldots,\lambda^dp^d+(1-\lambda^d)x^d\big)$
and $\tilde{q}=\xi\cdot q  + (1-\xi)\cdot y
=\big(\xi^1q^1+(1-\xi^1)y^1,\ldots,\xi^dq^d+(1-\xi^d)y^d\big)$.
\end{theorem}
\begin{proof}
According to the values of the $k$th entries $p^k,x^k,q^k,y^k$, it
falls into the following four cases:
\begin{description}
  \item[\textbf{Case I}] $(p^k)^2+(x^k)^2\neq0 $ and $(q^k)^2+(y^k)^2\neq0 $;
  \item[\textbf{Case II}] $(p^k)^2+(x^k)^2=0 $ and $(q^k)^2+(y^k)^2\neq0 $;
  \item[\textbf{Case III}] $(p^k)^2+(x^k)^2\neq0 $ and $(q^k)^2+(y^k)^2=0 $;
  \item[\textbf{Case IV}] $(p^k)^2+(x^k)^2=0 $ and $(q^k)^2+(y^k)^2=0 $.
\end{description}
In what follows, we will construct $d$ matrices $M_k$
and the corresponding entries $\lambda^k$ and $\xi^k$
for all $k=1,\ldots,d$ under the four different cases.

\textbf{Case I:} $(p^k)^2+(x^k)^2\neq0 $ and $(q^k)^2+(y^k)^2\neq0
$. For this case, $\lambda^k$ can be arbitrary real number and
$\xi^k$ is selected such that $\tilde{q}^k=\xi^kq^k +
(1-\xi^k)y^k\neq0$. Because $(q^k)^2+(y^k)^2\neq0 $, the existence
of such $\xi^k$ is assured. By fixing the values of $\lambda^k$ and
$\xi^k$, we denote $\Delta p^k=\tilde{p}^k-{p}^k$, $\Delta
x^k=\tilde{p}^k-{x}^k$, $\Delta q^k=\tilde{q}^k-{q}^k$, and $\Delta
y^k=\tilde{q}^k-{y}^k$.

Let
$S_0=\left(
     \begin{array}{cc}
       a^k & b^k \\
       b^k & c^k \\
     \end{array}
   \right)
$
and
$T_0=\left(
     \begin{array}{cc}
       d^k & e^k \\
       e^k & f^k \\
     \end{array}
   \right)
$ be two $2\times 2$ real symmetric matrices,
$U_0=\left(
     \begin{array}{cc}
       I_2 & S_0 \\
       O & I_2 \\
     \end{array}
   \right),
$
and
$V_0=\left(
     \begin{array}{cc}
       I_2 & O \\
       T_0 & I_2 \\
     \end{array}
   \right)
$. We consider the linear equations
\begin{equation}\label{case1-eq1}
V_0(p^k ,x^k,q^k ,y^k)^\intercal
= (p^k ,x^k,\tilde{q}^k ,\tilde{q}^k)^\intercal,
\end{equation}
and
\begin{equation}\label{case1-eq2}
U_0(p^k,x^k ,\tilde{q}^k ,\tilde{q}^k)^\intercal
=(\tilde{p}^k ,\tilde{p}^k ,\tilde{q}^k ,\tilde{q}^k).
\end{equation}
Because $(p^k)^2+(x^k)^2\neq0$, at least one
of the inequalities $p^k\neq0$ and $x^k\neq0$ holds.

If $p^k\neq0$, we have the solutions
\begin{equation}\label{case1-eq3}
e^k=\frac{1}{p^k}\Big(\Delta y^k - f^k x^k \Big),
\qquad
d^k=\frac{1}{p^k}\Big(\Delta q^k - e^k x^k \Big),
\end{equation}
for the equations \eqref{case1-eq1}, where the free parameter $f^k$
could be arbitrarily selected. Otherwise, the solutions of
\eqref{case1-eq1} are explicitly expressed as
\begin{equation}\label{case1-eq4}
e^k=\frac{1}{x^k}\Big(\Delta q^k - d^k p^k \Big), \qquad
f^k=\frac{1}{x^k}\Big(\Delta y^k - e^k p^k \Big),
\end{equation}
for the case $x^k\neq0$, where $d^k$ is freely selected.

For the equations \eqref{case1-eq2}, due to $\tilde{q}^k\neq0$
we obtain the solutions
\begin{equation}\label{case1-eq5}
b^k=\frac{\Delta p^k}{\tilde{q}^k}-a^k,
\qquad
c^k=\frac{\Delta x^k}{\tilde{q}^k}-b^k,
\end{equation}
where $a^k$ is a free parameter.

Once the values of $a^k$ , $b^k$, $c^k$, $d^k$, $e^k$, and $f^k$ are
determined, we then define the $d\times d$ diagonal matrix $A$ as
follows:
\begin{equation}\label{case1-eq6}
A_{ij}=
\left\{
\begin{aligned}
& a^k,\quad i=j=k,
\\
& 0,\quad \text{otherwise},
\end{aligned}
\right.
\end{equation}
whose $k$th main diagonal entry is $a^k$ and all other entries are
zero. Likewise, we can define the $d\times d$  diagonal matrices
$B$, $C$, $D$, $E$ and $F$, whose $k$th main diagonal entries are
respectively $b^k$, $c^k$, $d^k$, $e^k$, and $f^k$, and all other
entries are zero.

Let
\begin{equation*}
T=\left(
    \begin{array}{cc}
      A & B \\
      B & C \\
    \end{array}
  \right)\in \mathbb{R}^{2d\times2d},
  \quad
  S=\left(
    \begin{array}{cc}
      D & E \\
      E & F \\
    \end{array}
  \right)\in \mathbb{R}^{2d\times2d},
\end{equation*}
and then $S$ and $T$ must be symmetric. We further
derive $U$ and $V$ as follows:
\begin{equation*}
U=\left(
    \begin{array}{cc}
      I_{2d} & S \\
      O & I_{2d} \\
    \end{array}
  \right)\in \mathbb{R}^{4d\times4d},
  \quad
  V=\left(
    \begin{array}{cc}
      I_{2d} & O \\
      T & I_{2d} \\
    \end{array}
  \right)\in \mathbb{R}^{4d\times4d}.
\end{equation*}
According to \cite[Proposition~4.1.7]{Feng2010}, both $U$ and $V$
are symplectic. Now, let $M_k=UV$, and then $M_k$ is
also symplectic as it is the product of two symplectic matrices.

Let $u, v, w, z\in\mathbb{R}^{d}$ be arbitrary vectors whose $k$th
entries are respectively $p^k$, $x^k$, $q^k$, and
$y^k$. We consequently define the vectors
$\tilde{u},\tilde{v},\tilde{w}, \tilde{z}\in\mathbb{R}^{d}$ as
follows:
\begin{equation*}
\tilde{u}_i=
\left\{
\begin{aligned}
&u_i,\  i\neq k,
\\
&\tilde{p}^k,\  i=k,
\end{aligned}
\right.
\quad
\tilde{v}_i=
\left\{
\begin{aligned}
&v_i,\  i\neq k,
\\
&\tilde{p}^k,\  i=k,
\end{aligned}
\right.
\quad
\tilde{w}_i=
\left\{
\begin{aligned}
&w_i,\  i\neq k,
\\
&\tilde{q}^k,\  i=k,
\end{aligned}
\right.
\quad
\tilde{z}_i=
\left\{
\begin{aligned}
&z_i,\  i\neq k,
\\
&\tilde{q}^k,\  i=k.
\end{aligned}
\right.
\end{equation*}
It then follows from the construction of $M_k$ that
\begin{equation}\label{case1-eq7}
M_k(u,v,w,z)^\intercal
=(\tilde{u},\tilde{v},\tilde{w},\tilde{z})^\intercal .
\end{equation}
That is, the linear mapping defined by the symplectic matrix $M_k$
maps the $k$th entries $p^k$, $x^k$, $q^k$, and $y^k$ respectively
to $\tilde{p}^k$, $\tilde{p}^k$, $\tilde{q}^k$, and $\tilde{q}^k$,
while it maintains all other entries of $u$, $v$, $w$, and $z$
unchanged.

\textbf{Case II:} $(p^k)^2+(x^k)^2=0$ and $(q^k)^2+(y^k)^2\neq0$.
The condition of $(p^k)^2+(x^k)^2=0$ means that
$p^k=x^k=0$, thereby leading to $\tilde{p}^k=\lambda^kp^k +
(1-\lambda^k)x^k=0$ for any $\lambda^k\in\mathbb{R}$. Similarly to
Case I, we take an arbitrary value for $\lambda^k$ and select the
value of $\xi^k$ such that $\tilde{q}^k=\xi^kq^k +
(1-\xi^k)y^k\neq0$.

Suppose that
$T_0=\left(
       \begin{array}{cc}
         a^k & b^k\\
         c^k & d^k \\
       \end{array}
     \right)
$
is a nonsingular $2\times2$ matrix. With the above $\lambda^k$, $\xi^k$, $\tilde{p}^k$,
and $\tilde{q}^k$, we now consider the linear equations
\begin{equation}\label{case2-eq1}
\left(
  \begin{array}{cc}
    T_0^{-\intercal} & O \\
    O & T_0 \\
  \end{array}
\right)
\left(
  \begin{array}{c}
    0 \\
    0 \\
    q^k \\
    y^k \\
  \end{array}
\right)
=
\left(
  \begin{array}{c}
    0 \\
    0 \\
    \tilde{q}^k \\
    \tilde{q}^k \\
  \end{array}
\right),
\end{equation}
which is identical to
\begin{equation}\label{case2-eq2}
\left(
  \begin{array}{cc}
    a^k & b^k\\
    c^k & d^k \\
  \end{array}
\right)
\left(
  \begin{array}{c}
    q^k \\
    y^k \\
  \end{array}
\right)
=
\left(
  \begin{array}{c}
     \tilde{q}^k \\
    \tilde{q}^k \\
  \end{array}
\right).
\end{equation}

Because of $(q^k)^2+(y^k)^2\neq0$, at least one of the
inequalities $q^k\neq0$ and $y^k\neq0$ holds. If $q^k\neq0$, the
solutions of \eqref{case2-eq2} are as follows:
\begin{equation}\label{case2-eq3}
a^k=\frac{\tilde{q}^k - b^ky^k}{q^k},
\qquad
c^k=\frac{\tilde{q}^k - d^ky^k}{q^k},
\end{equation}
where $b^k$ and $d^k$ are free parameters. In addition, the
nonsingularity of $T_0$ implies $\det(T_0)=a^kd^k-b^kc^k\neq0$,
which leads to $b^k\neq d^k$ as $q^k\neq0$ and $\tilde{q}^k\neq0$.
For convenience, we  set $b^k=0$ and $d^k=1$.

For the other case $y^k\neq0$, likewise, the solutions
of \eqref{case2-eq2} are written as:
\begin{equation}\label{case2-eq4}
b^k=\frac{\tilde{q}^k - a^kq^k}{y^k},
\qquad
d^k=\frac{\tilde{q}^k - c^kq^k}{y^k},
\end{equation}
where $a^k$ and $c^k$ are free parameters. To make $T_0$ nonsingular,
it just needs to set $a^k\neq c^k$ and thus we can conveniently set
$a^k=0$ and $c^k=1$.

With the above solutions of $a^k$, $b^k$, $c^k$, and $d^k$, we define
the diagonal matrices $A$, $B$, $C$, and $D$ as follows:
\begin{equation*}
A_{jj}=
\left\{
\begin{aligned}
& 1,\  j\neq k,
\\
& a^k,\  j=k,
\end{aligned}
\right.
~~
B_{jj}=
\left\{
\begin{aligned}
& 0,\  j\neq k,
\\
& b^k,\  j=k,
\end{aligned}
\right.
~~
C_{jj}=
\left\{
\begin{aligned}
& 0,\  j\neq k,
\\
& c^k,\  j=k,
\end{aligned}
\right.
~~
D_{jj}=
\left\{
\begin{aligned}
& 1,\  j\neq k,
\\
& d^k,\  j=k.
\end{aligned}
\right.
\end{equation*}
Let
$T=\left(
     \begin{array}{cc}
       A & B \\
       C & D \\
     \end{array}
   \right),
$ and then $T$ must be nonsingular as
$\det(T)=\det(T_0)$ and $\det(T_0)\neq 0$.

We next define the matrix $M_k$ by $M_k=\left(
       \begin{array}{cc}
         T^{-\intercal} & O\\
         O & T \\
       \end{array}
     \right)
$. Then, according to \cite[Proposition 4.1.8]{Feng2010}, the matrix
$M_k$ must be symplectic. Moreover, for the values
$p^k=x^k=\tilde{p}^k=0$, $(q^k)^2+(y^k)^2\neq0$, and
$\tilde{q}^k\neq0$ in this case, let $u,v,w,z$ and
$\tilde{u},\tilde{v},\tilde{w},\tilde{z}$ be defined the same as
those in Case I, and then we also have the same equality
\begin{equation*}
M_k(u,v,w,z)^\intercal
=(\tilde{u},\tilde{v},\tilde{w},\tilde{z})^\intercal.
\end{equation*}

\textbf{Case III:} $(p^k)^2+(x^k)^2\neq0$ and $(q^k)^2+(y^k)^2=0$.
The construction of $M_k$ is similar to Case II. As $\tilde{q}^k=0$,
the value of $\xi^k$ can be freely set, while $\lambda^k$ is selected
to satisfy $\tilde{p}^k=\lambda^kp^k+(1-\lambda^k)x^k\neq0$.

By fixing $\lambda^k$, $x^k$, and $\tilde{p}^k$, we consider the
following linear equations:
\begin{equation}\label{case3-eq1}
T_0
\left(
  \begin{array}{c}
    p^k \\
    x^k \\
  \end{array}
\right)
=
\left(
  \begin{array}{c}
     \tilde{p}^k \\
    \tilde{p}^k \\
  \end{array}
\right),
\end{equation}
where the matrix
$T_0=
\left(
  \begin{array}{cc}
    a^k & b^k\\
    c^k & d^k \\
  \end{array}
\right)$ is nonsingular. Similarly to Case II, for the
case $p^k\neq0$ we have the solutions
\begin{equation}\label{case2-eq3}
a^k=\frac{\tilde{p}^k - b^kx^k}{p^k},
\qquad
c^k=\frac{\tilde{p}^k - d^kx^k}{p^k},
\end{equation}
where the free parameters $b^k$ and $d^k$ should satisfy $b^k\neq d^k$
to meet the nonsingularity of $T_0$. Here, we conveniently set $b^k=0$
and $d^k=1$.

For the other case $x^k\neq0$, the solutions of \eqref{case3-eq1} will be
\begin{equation}\label{case2-eq3}
b^k=\frac{\tilde{p}^k - a^kp^k}{x^k},
\qquad
d^k=\frac{\tilde{p}^k - c^kp^k}{x^k},
\end{equation}
where the free parameters should also satisfy $a^k\neq c^k$ and
are conveniently set to $a^k=0$ and $c^k=1$.

With the  solutions of $a^k$, $b^k$, $c^k$, and $d^k$
obtained above, we define the diagonal matrices $A$, $B$, $C$, and
$D$ in the same way as those in Case II. By letting $T=\left(
     \begin{array}{cc}
       A & B \\
       C & D \\
     \end{array}
   \right),
$ we thus confirm that $T$ is nonsingular.
We then define $M_k$ as $M_k=\left(
       \begin{array}{cc}
        T & O\\
         O & T^{-\intercal} \\
       \end{array}
     \right)
$. The matrix $M_k$ must be symplectic according to the
nonsingularity of $T$ and Proposition~4.1.8 of \cite{Feng2010}. It
can be verified that the equality \eqref{case1-eq7} also holds on
noting the fact that $q^k=y^k=\tilde{q}^k=0$ and
$\tilde{p}^k\neq0$.

\textbf{Case IV:} $(p^k)^2+(x^k)^2=0 $ and $(q^k)^2+(y^k)^2=0 $. In
this case, we have $p^k=x^k=q^k=y^k=0$. The equality
$\tilde{p}^k=\tilde{q}^k=0$ holds for any
$\lambda^k,\xi^k\in\mathbb{R}$. Thus, the parameters $\lambda^k$ and
$\xi^k$ can be arbitrarily real numbers. Let $M_k=I_{4d}$, and then
the identity matrix $M_k$ is symplectic. On noting
$u=\tilde{u}$, $v=\tilde{v}$, $w=\tilde{w}$, and $z=\tilde{z}$ in
this case, the equality \eqref{case1-eq7} certainly holds.

On the basis of the above discussion on the four different cases, we
obtain $d$ symplectic matrices $M_k~(k=1,\ldots,d)$ and all the
entries $\lambda^k$ and $\xi^k$ respectively of the vectors
$\lambda$ and $\xi$. From the construction of $M_k$, we know that
the $k$th matrix $M_k$ maps $(p^k,x^k,q^k,y^k)$ to
$(\tilde{p}^k,\tilde{p}^k,\tilde{q}^k,\tilde{q}^k)$ while maintains
all other entries of $p,x,q,y$ unchanged. Then, by letting
$M=\prod_{k=1}^{d}M_k$, the equality \eqref{weight-add} holds. The
symplecticity of $M$ naturally follows from that of $M_k$.
\end{proof}

\begin{remark}\label{remark1}
It can be observed from the proof of Theorem~\ref{theorem2} that the
vectors $\lambda,\xi\in\mathbb{R}^{d}$ cannot be arbitrarily
selected. In fact, the  $k$th entries $\lambda^k$ and $\xi^k$ just
need to avoid the special values $(\lambda_*^k,\xi^k_*)$ such that
$\tilde{p}^k_* = \tilde{q}^k_* =0 $ with
$\tilde{p}^k_*=\lambda^k_*p^k+(1-\lambda^k_*)x^k$ and
$\tilde{q}^k_*=\xi^k_*q^k+(1-\xi^k_*)y^k$ under the
condition: $(p^k)^2+(x^k)^2+(q^k)^2 +(y^k)^2\neq0$. That is the
probability of the existence of such a symplectic matrix $M$ for a
random choice of $(\lambda,\xi)$ is $1$, and we can claim that the
vectors $\lambda,\xi\in\mathbb{R}^{d}$ are almost free.
\end{remark}

Theorem~\ref{theorem2} shows that there exist such explicit symplectic
mappings addressing Question \eqref{question1}, and the number of
such explicit symplectic mappings is infinite. According to
Lemma~\ref{theorem1} and Theorem~\ref{theorem2}, we are now in a
position to introduce the explicit symplectic integrator for the
Hamiltonian $H(p,q)$ as follows.
\begin{definition}\label{definition1}
Assuming that we have  an $r$th-order symplectic extended phase
space integrator $\Phi_h:~T^*\mathbb{R}^{2d}\rightarrow
T^*\mathbb{R}^{2d}$ for the extended Hamiltonian $\Gamma(p,x,q,y)$
and two vectors $\lambda,\xi\in\mathbb{R}^{d}$ such that
$\lambda^k,\xi^k\in (0,1)$ and $\lambda^k\neq\xi^k$ for all
$k=1,\ldots,d$, the numerical integrator
$\widetilde{\Phi}_h:~(p_n,q_n)\mapsto (p_{n+1},q_{n+1})$ for the
Hamiltonian system $H(p,q)$ with the initial conditions $(p_0,q_0)$
is induced and defined as follows:
\begin{enumerate}
  \item $(\tilde{p}_{n+1},\tilde{x}_{n+1},\tilde{q}_{n+1},\tilde{y}_{n+1})
  := \Phi_h(p_n,p_n,q_n,q_n)$;
  \item define
\begin{equation*}
  \left\{
    \begin{aligned}
    &P:=\lambda\cdot\tilde{p}_{n+1} + (1-\lambda)\cdot\tilde{x}_{n+1},
    \\
    &Q:=\lambda\cdot\tilde{q}_{n+1} + (1-\lambda)\cdot\tilde{y}_{n+1},
    \\
    &\widetilde{P}:=\xi\cdot\tilde{p}_{n+1} + (1-\xi)\cdot\tilde{x}_{n+1},
    \\
    &\widetilde{Q}:=\xi\cdot\tilde{q}_{n+1} + (1-\xi)\cdot\tilde{y}_{n+1};
    \end{aligned}
    \right.
\end{equation*}
  \item for $k=1,\ldots,d,$ let
\begin{equation*}
  (p_{n+1}^k,q_{n+1}^k):=
  \left\{
    \begin{aligned}
    &(P^k,Q^k),\quad (P^k)^2+(Q^k)^2\neq0,
    \\
    &(\widetilde{P}^k,\widetilde{Q}^k),\quad (P^k)^2+(Q^k)^2=0;
    \end{aligned}
    \right.
\end{equation*}
 \item define $(p_{n+1},q_{n+1}):=\widetilde{\Phi}_h(p_n,q_n)$.
\end{enumerate}
\end{definition}
\begin{theorem}\label{theorem3}
The explicit integrator $\widetilde{\Phi}_h:~(p_n,q_n)\mapsto
(p_{n+1},q_{n+1})$ defined by Definition~\ref{definition1} is
symplectic and of order $r$ for the Hamiltonian $H(p,q)$.
\end{theorem}
\begin{proof}
If $(p_{n+1}^k)^2 + (q_{n+1}^k)^2\neq0$, this case falls into one of
\textbf{Case I}, \textbf{Case II}, and \textbf{Case III} in the
proof of Theorem~\ref{theorem2}. Otherwise, $(p_{n+1}^k)^2 +
(q_{n+1}^k)^2=0$ leads to $(P^k)^2+(Q^k)^2=0$ and
$(\widetilde{P}^k)^2+(\widetilde{Q}^k)^2=0$. Because of
$\lambda^k\neq\xi^k$, we then yield
$\tilde{p}_{n+1}^k=\tilde{x}_{n+1}^k=
\tilde{q}_{n+1}^k=\tilde{y}_{n+1}^k=0$, which falls into
\textbf{Case IV}. Therefore, there exists such a symplectic matrix
$M$, which maps
$(\tilde{p}_{n+1},\tilde{x}_{n+1},\tilde{q}_{n+1},\tilde{y}_{n+1})$
to $(p_{n+1},p_{n+1},q_{n+1},q_{n+1})$, and derives the symplectic
mapping
$M\circ\Phi_h:~(p_{n},p_{n},q_{n},q_{n})\mapsto(p_{n+1},p_{n+1},q_{n+1},q_{n+1})$.
The symplecticity of $\widetilde{\Phi}_h$ thus follows from
Lemma~\ref{theorem1}.

Because $\Phi_h$ is of order $r$, both
$(\tilde{p}_{n+1},\tilde{q}_{n+1})$ and
$(\tilde{x}_{n+1},\tilde{y}_{n+1})$ possess the local truncation
error of $\mathcal{O}(h^{r+1})$ for the exact solution $(p,q)$ of
$H(p,q)$.  The weighted average of
$(\tilde{p}_{n+1},\tilde{q}_{n+1})$ and
$(\tilde{x}_{n+1},\tilde{y}_{n+1})$ gives the same local truncation
error of $\mathcal{O}(h^{r+1})$, which implies that the
explicit integrator $\widetilde{\Phi}_h$ is of order $r$. This
completes the proof.
\end{proof}

\begin{remark}\label{remark2}
It is known that the two vectors $\lambda$ and $\xi$ can be nearly
freely selected not only for each entry by $\lambda^k\neq\xi^k$, but
also at every time stepsize. That is, we can even use different
vectors $\lambda$ and $\xi$ for different $n$. This seems to give us
enormous freedom for the choice of the two vectors to achieve better
performance for the symplectic method $\widetilde{\Phi}_h$. In the
sense of practical computation, one may prefer the simple case
where $\lambda$ and $\xi$ are just selected as two
different constant vectors. Because a large $\lambda^k$ (or $\xi^k$)
will enlarge the local truncation error as
$p_{n+1}^{k}=\lambda^{k}\tilde{p}_{n+1}^k
 + (1-\lambda^{k})\tilde{x}_{n+1}^{k}=\tilde{x}_{n+1}^{k}+
 \lambda^{k}(\tilde{p}_{n+1}^{k}-\tilde{x}_{n+1}^{k})$
once $\tilde{p}_{n+1}^{k}\neq\tilde{x}_{n+1}^{k}$, the restriction
of the entries of $\lambda$ and $\xi$ on the interval $(0,1)$ is
reasonable, which is beneficial to the local truncation
error of $(p_{n+1},q_{n+1})$ between the error of
$(\tilde{p}_{n+1},\tilde{q}_{n+1})$ and the error of
$(\tilde{x}_{n+1},\tilde{y}_{n+1})$.
\end{remark}
Clearly, for the question ``\emph{Does there exist an explicit
symplectic method for general nonseparable Hamiltonian systems?}'',
we have a positive answer ``yes'' with Theorem~\ref{theorem3}.
It is quite easy to construct such explicit symplectic methods
 from the approach to verifying the existence of explicit
symplectic mappings from
$(\tilde{p}_{n+1},\tilde{x}_{n+1},\tilde{q}_{n+1},\tilde{y}_{n+1})$
of an explicit symplectic extended phase space method to
$(p_{n+1},p_{n+1},q_{n+1},q_{n+1})$ that is just a
weighted average of
$(\tilde{p}_{n+1},\tilde{x}_{n+1},\tilde{q}_{n+1},\tilde{y}_{n+1})$.

In the final part of this section, we turn to the necessity of the
two different vectors $\lambda$ and $\xi$ in
Definition~\ref{definition1}. If $\lambda^k=\xi^k\in (0,1)$, the
only case $(p_{n+1}^k)^2 + (q_{n+1}^k)^2=0$ under the condition
$(\tilde{p}^k_{n+1})^2+(\tilde{x}^k_{n+1})^2+(\tilde{q}^k_{n+1})^2
+(\tilde{y}^k_{n+1})^2\neq0$ is that $\tilde{p}^k_{n+1}$ and
$\tilde{x}^k_{n+1}$ have the opposite signs, and so do
$\tilde{q}^k_{n+1}$ and $\tilde{y}^k_{n+1}$. As we know that
$(\tilde{p}^k_{n+1},\tilde{q}^k_{n+1})$ and
$(\tilde{x}^k_{n+1},\tilde{y}^k_{n+1})$ are both
$\mathcal{O}(h^{r})$-accurate to the exact solution
$\big(p^k(t_{n+1}),q^k(t_{n+1})\big)$, the opposite signs between
$(\tilde{p}^k_{n+1},\tilde{q}^k_{n+1})$ and
$(\tilde{x}^k_{n+1},\tilde{y}^k_{n+1})$ lead to that both
$p^k(t_{n+1})$ and $q^k(t_{n+1})$ are $\mathcal{O}(h^{r})$-close to
zero, which rarely occurs in practice. Therefore, the restriction of
$\lambda^k=\xi^k$ on the interval $(0,1)$  will greatly decrease the
possibility of $(p_{n+1}^k)^2 + (q_{n+1}^k)^2=0$ for the case
$(\tilde{p}^k_{n+1})^2+(\tilde{x}^k_{n+1})^2+(\tilde{q}^k_{n+1})^2
+(\tilde{y}^k_{n+1})^2\neq0$. That is, only one vector $\lambda$
used in Definition~\ref{definition1} is highly possible except for
some extremely rare occasions. This point leads to a conjecture that
\emph{the explicit method $\widetilde{\Phi}_h$
determined by Definition~\ref{definition1} is still
symplectic even for the case where the two weight vectors $\lambda$
and $\xi$ are equal and freely selected.}

For this conjecture, it is known from Theorem~\ref{theorem2} that
there does not always exist a linear symplectic transformation
(symplectic matrix) mapping an arbitrary pair $(p,x,q,y)$ to
$(\tilde{p},\tilde{p},\tilde{q},\tilde{q})$ with
$\tilde{p}=\lambda\cdot p+(1-\lambda)\cdot x$ and
$\tilde{q}=\lambda\cdot q+(1-\lambda)\cdot y$ for an arbitrary
vector $\lambda$. To verify this conjecture, we need to consider
nonlinear symplectic transformations instead of linear symplectic
transformations (symplectic matrices) that fulfill the condition
required by Question~\ref{question1}. By considering the translation
mapping, we are hopeful of obtaining some surprising results  in the
next section.

\section{Symplectic integrators with projection}\label{sec:symp-projection}
As analyzed in the previous section, although explicit symplectic
methods could be constructed by Definition~\ref{definition1} via the
extended phase space, there still exists a flaw that two different
weight vectors $\lambda$ and $\xi$ are needed in the practical
computation for some special cases. To completely determine the
symplecticity of the method in Definition~\ref{definition1} with
only one free weight vector, i.e., $\lambda=\xi$, we consider
another approach to deriving the symplectic mapping $M$ that
projects the extended phase space $T^*\mathbb{R}^{2d}$ onto its
submanifold $\mathcal{N}$.

\subsection{Symplectic integrator with standard projection}\label{sec:stand-proj}
In the construction of semiexplicit symplectic methods by
Jayawardana \& Ohsawa \cite{Jayawardana2023}, the
authors used the symmetric projection so that the
symplectic methods are also symmetric. As we have known, the
symplecticity of the methods is due to the projection of the pair
$(p,x,q,y)$ onto the submanifold $\mathcal{N}$ defined by
\eqref{manifold}, while the symmetry originates from that of the
\emph{symmetric} projection and the \emph{symmetric} method
$\Phi_h$. Bearing this point in mind, we consider the classical
standard projection \cite{Hairer2006} and thus derive the following
symplectic method.

\begin{definition}[Symplectic integrator with standard projection]
\label{definition2}
Assuming that we have an $r$th-order symplectic extended phase space integrator
$\Phi_h:~T^*\mathbb{R}^{2d}\rightarrow T^*\mathbb{R}^{2d}$ for the
extended Hamiltonian $\Gamma(p,x,q,y)$.
The numerical integrator
$\widetilde{\Phi}_h:~(p_n,q_n)\mapsto
(p_{n+1},q_{n+1})$ for the Hamiltonian system $H(p,q)$ with the initial conditions
$(p_0,q_0)$ is induced and defined as follows:
\begin{enumerate}
  \item $(\tilde{p}_{n+1},\tilde{x}_{n+1},\tilde{q}_{n+1},\tilde{y}_{n+1})
  := \Phi_h(p_n,p_n,q_n,q_n)$;
  \item find $\delta p,\delta x,\delta q,\delta y\in\mathbb{R}^d$
  satisfying the condition $$({p}_{n+1},{p}_{n+1},{q}_{n+1},{q}_{n+1}):=
  (\tilde{p}_{n+1},\tilde{x}_{n+1},\tilde{q}_{n+1},\tilde{y}_{n+1})
  +(\delta p,\delta x,\delta q,\delta y)\in \mathcal{N};$$
  \item define $(p_{n+1},q_{n+1}):=\widetilde{\Phi}_h(p_n,q_n)$.
\end{enumerate}
\end{definition}
\begin{theorem}\label{theorem4}
The numerical integrator $\widetilde{\Phi}_h:~(p_n,q_n)\mapsto
(p_{n+1},q_{n+1})$ determined by
Definition~\ref{definition2} is symplectic for the Hamiltonian
$H(p,q)$. Moreover, the method $\widetilde{\Phi}_h$ is of order $r$
provided $(\delta p,\delta x,\delta q,\delta
y)=\mathcal{O}(h^{r+1})$.
\end{theorem}
\begin{proof}
The mapping
$M:~(\tilde{p}_{n+1},\tilde{x}_{n+1},\tilde{q}_{n+1},\tilde{y}_{n+1})
\mapsto({p}_{n+1},{p}_{n+1},{q}_{n+1},{q}_{n+1})$ can be regarded as
a translation with the translation vector $(\delta p,\delta x,\delta
q,\delta y)$. Thus, the Jacobian of
$({p}_{n+1},{p}_{n+1},{q}_{n+1},{q}_{n+1})$ with respect to
$(\tilde{p}_{n+1},\tilde{p}_{n+1},\tilde{q}_{n+1},\tilde{q}_{n+1})$
is the identity matrix, which is certainly symplectic and so is the
mapping $M$. Being a product of symplectic mappings $\Phi_h$ and
$M$, the mapping $M\circ\Phi_h:~(p_n,p_n,q_n,q_n)\mapsto
(p_{n+1},p_{n+1},q_{n+1},q_{n+1})$ is also symplectic. The
symplecticity of
$\widetilde{\Phi}_h:~(p_n,q_n)\mapsto(p_{n+1},q_{n+1})$ immediately
follows from Lemma~\ref{theorem1}.

Given the $r$th order of $\Phi_h$,
$(\tilde{p}_{n+1},\tilde{x}_{n+1},\tilde{q}_{n+1}, \tilde{y}_{n+1})$
will possess the local truncation error of $\mathcal{O}(h^{r+1})$.
If $(\delta p,\delta x,\delta q,\delta
y)=\mathcal{O}(h^{r+1})$, $(p_{n+1},p_{n+1}, q_{n+1},q_{n+1})$ gives
the same local truncation error:
$\mathcal{O}(h^{r+1})$, which indicates that
$M\circ\Phi_h$ is of order $r$. As the numerical solution
$(p_{n+1},p_{n+1},q_{n+1},q_{n+1})$ is two identical copies of
$(p_{n+1},q_{n+1})$ for the original Hamiltonian $H(p,q)$, this
implies the same $r$th order of $\widetilde{\Phi}_h$ for $H(p,q)$.
The proof is complete.
\end{proof}

By setting
\begin{equation}\label{translation-set}
\begin{aligned}
&\delta p = (\lambda-1)\cdot(\tilde{p}_{n+1}-\tilde{x}_{n+1}),
\quad &&\delta x = \lambda\cdot(\tilde{p}_{n+1}-\tilde{x}_{n+1}),
\\
&\delta q = (\xi-1)\cdot(\tilde{q}_{n+1}-\tilde{y}_{n+1}),
\quad &&\delta y = \xi\cdot(\tilde{q}_{n+1}-\tilde{y}_{n+1}),
\end{aligned}
\end{equation}
where the two vectors $\lambda,\xi\in\mathcal{R}^d$ are freely
selected, we immediately derive at that
\begin{equation*}
p_{n+1}=\lambda\cdot\tilde{p}_{n+1} +
(1-\lambda)\cdot\tilde{x}_{n+1}, \qquad
q_{n+1}=\xi\cdot\tilde{q}_{n+1} +
(1-\xi)\cdot\tilde{y}_{n+1},
\end{equation*}
which defines an explicit symplectic integrator
$\widetilde{\Phi}_h:~(p_n,q_n)\mapsto(p_{n+1},q_{n+1})$ according to
Definition~\ref{definition2} and Theorem~\ref{theorem4}. Considering
the complete degree of freedom of the two vectors $\lambda,\xi\in\mathbb{R}^d$
for the setting \eqref{translation-set}, the numerical method
$\widetilde{\Phi}_h$ obtained here completely covers that determined
by Definition~\ref{definition1}. In addition, the complete degree of freedom
of $\lambda$ and $\xi$ enables the special setting $\lambda=\xi$,
which further addresses the conjecture raised in the previous
section that only one free weight vector $\lambda$ can ensure the
symplecticity of $\widetilde{\Phi}_h$. Moreover, the $d$ degrees of
freedom of the single weight vector $\lambda$ using in the
symplectic method $\widetilde{\Phi}_h$ provides us with some other
possibilities, such as the preservation of the Hamiltonian energy or
other first integrals. In principle, the total $d$ degrees of
freedom could completely preserve at most $d$ isolated first
integrals of $H(p,q)$. A further discussion on the choice of the
weight vectors is made in Section~\ref{sec:nonint}.

\subsection{Symplectic integrator with generalized projection}
The semiexplicit symplectic methods of Jayawardana \& Ohsawa
\cite{Jayawardana2023} adopt the symmetric projection that adds the
same perturbation at both the beginning and the end of the numerical
integration. As we know from Theorem~\ref{theorem4} the standard
projection adding the perturbation only at the end of the numerical
integration also yields the symplecticity of the induced method. The
two points motivate us to consider the generalized projection that
uses two different perturbations respectively at the beginning and
the end of the numerical integration, where the perturbation at the
end is thought of as a standard projection from the
extended phase space onto the submanifold $\mathcal{N}$ at the
current step while the perturbation at the beginning is regarded as
an ``inverse projection'' of the projection at the previous step.
Following this idea, the new generalized projection method is
defined as follows.

\begin{definition}[Symplectic integrator with generalized projection]
\label{definition3} Assuming that we have an
$r$th-order symplectic extended phase space integrator
$\Phi_h:~T^*\mathbb{R}^{2d}\rightarrow T^*\mathbb{R}^{2d}$ for the
extended Hamiltonian $\Gamma(p,x,q,y)$, the numerical
integrator $\widetilde{\Phi}_h:~(p_n,q_n)\mapsto (p_{n+1},q_{n+1})$
for the Hamiltonian system $H(p,q)$ with the initial conditions
$(p_0,q_0)$ is induced and defined as follows:
\begin{enumerate}
 \item find $\Delta p,\Delta x,\Delta q,\Delta y,\delta p,\delta x,
 \delta q,\delta y\in\mathbb{R}^d$ satisfying the condition
 $$({p}_{n+1},{p}_{n+1},{q}_{n+1},{q}_{n+1}):=
  (\tilde{p}_{n+1},\tilde{x}_{n+1},\tilde{q}_{n+1},\tilde{y}_{n+1})
  +(\delta p,\delta x,\delta q,\delta y)\in \mathcal{N},$$
where
 $$(\tilde{p}_{n+1},\tilde{x}_{n+1},\tilde{q}_{n+1},\tilde{y}_{n+1})
  = \Phi_h(\tilde{p}_n,\tilde{x}_n,\tilde{q}_n,\tilde{y}_n)$$
  and
  $$(\tilde{p}_{n},\tilde{x}_{n},\tilde{q}_{n},\tilde{y}_{n})=(p_n,p_n,q_n,q_n) +
 (\Delta p,\Delta x,\Delta q,\Delta y);$$
  \item define $(p_{n+1},q_{n+1}):=\widetilde{\Phi}_h(p_n,q_n)$.
\end{enumerate}
\end{definition}
\begin{theorem}\label{theorem5}
The numerical integrator $\widetilde{\Phi}_h:~(p_n,q_n)\mapsto
(p_{n+1},q_{n+1})$ determined by
Definition~\ref{definition3} is symplectic for the Hamiltonian
$H(p,q)$. Moreover, the method $\widetilde{\Phi}_h$ is of order $r$
provided $(\Delta p,\Delta x,\Delta q,\Delta y)=\mathcal{O}(h^{r+1})$
and $(\delta p,\delta x,\delta q,\delta y)=\mathcal{O}(h^{r+1})$.
\end{theorem}
\begin{proof}
The proof is similar to that of Theorem~\ref{theorem4}
and thus omitted here.
\end{proof}
\begin{remark}\label{remark3}
It is known that the generalized projection symplectic integrator
determined by Definition~\ref{definition3} can be viewed as the
generalization of both the explicit standard projection symplectic
integrator defined by Definition~\ref{definition2} and the
semiexplicit symmetric projection symplectic integrator
considered in \cite{Jayawardana2023}, because the
generalized projection symplectic integrator respectively reduces to
the explicit standard projection symplectic integrator as $\Delta p
=\Delta x=\Delta q=\Delta y=0$ and to the semiexplicit symmetric
projection symplectic integrator as $(\Delta p,\Delta x,\Delta
q,\Delta y)=(\delta p,\delta x,\delta q,\delta y)$. This fact
immediately leads to the  existence of the generalized projection
symplectic integrator, while the solutions of $\Delta p,\Delta
x,\Delta q,\Delta y,\delta p,\delta x,\delta q,\delta y$ are not
unique in general. Here, we do not pay much attention to the
generalized projection symplectic integrator but only use it to show
the symplecticity of Pihajoki's or Tao's method in the next
subsection.
\end{remark}

\subsection{Are Pihajoki's or Tao's explicit methods symplectic?}
As we have known, all the present results in the existing literature
do not claim the sympleciticy (in the original phase space) of
Pihajoki's \cite{Pihajoki2015} or Tao's \cite{Tao2016b} extended
phase space methods. On the contrary, these results just confirm the
sympleciticy in the extended phase space and oppose the sympleciticy
in the original phase space. However, with the generalized
projection symplectic integrator determined by
Definition~\ref{definition3}, we are surprised by presenting the
following theorem to show that all symplectic extended phase space
methods derived by Pihajoki's or Tao's approach are symplectic in
the original phase space as well.

\begin{theorem}[Symplecticity of  extended phase space methods]\label{theorem6}
Suppose that $(p_n,x_n,q_n,y_n)$ are the numerical solutions
obtained by an $r$th-order explicit extended phase space method
$\Phi_h$ (Pihajoki's or Tao's approach) with the initial conditions
$(p_0,p_0,q_0,q_0)$, i.e.,
$(p_{n+1},x_{n+1},q_{n+1},y_{n+1})=\Phi_h(p_n,x_n,q_n,y_n)$. Then,
the induced method defined by $(p_{n+1},q_{n+1})
=\widetilde{\Phi}_h(p_n,q_n)$ or $(x_{n+1},y_{n+1})
=\widetilde{\Phi}_h(x_n,y_n)$ is also of order $r$ and symplectic
for the original Hamiltonian $H(p,q)$ with the initial conditions
$(p_0,q_0)$.
\end{theorem}
\begin{proof}
Because $(p_{n+1},x_{n+1},q_{n+1},y_{n+1})$ is globally of
accuracy $\mathcal{O}(h^r)$
for the exact solution $\big(p(t_n),p(t_n),q(t_n),q(t_n)\big)$ of the
extended Hamiltonian $\Gamma(p,x,q,y)$ once imposed the initial
conditions $(p_0,p_0,q_0,q_0)$, both $(p_{n+1},q_{n+1})$ and $(x_{n+1},y_{n+1})$
can be regarded as an $\mathcal{O}(h^r)$-approximation to the exact
solution $\big(p(t_n),q(t_n)\big)$ of the  Hamiltonian $H(p,q)$, which shows the
$r$th order of $\widetilde{\Phi}_h$.

Define the translation mapping $\widetilde{T}_n$ as follows
\begin{equation*}
\widetilde{T}_n:~(p_n,x_n,q_n,y_n)\mapsto (p_n,x_n,q_n,y_n)
+ (0,\delta_{n} x,0,\delta_{n} y),
\end{equation*}
where $\delta_{n} x= p_n-x_n$ and  $\delta_{n} y= q_n-y_n$. It is
clear that the mapping $\widetilde{T}_n$ translates
$(p_n,x_n,q_n,y_n)$ to $(p_n,p_n,q_n,q_n)$, and its inverse
$\widetilde{T}_n^{-1}$ translates $(p_n,p_n,q_n,q_n)$ back to
$(p_n,x_n,q_n,y_n)$. Now, let $\Psi^{(n+1)}_h=
\widetilde{T}_{n+1}\circ\Phi_h\circ\widetilde{T}_{n}^{-1}$ for
$n\in\mathbb{N}$. We then have
\begin{equation}\label{theo6-prf-eq1}
\Psi^{(n+1)}_h:~(p_n,p_n,q_n,q_n)\mapsto (p_{n+1},p_{n+1},q_{n+1},q_{n+1}),
\end{equation}
where $(p_n,q_n)$ is just the numerical solution obtained by  $\Phi_h$. It immediately
follows from Lemma~\ref{theorem1} and Theorem~\ref{theorem5} that the induced
mapping $\widetilde{\Phi}_h:~(p_n,q_n)\mapsto(p_{n+1},q_{n+1})$ is also symplectic.
The symplecticity of $\widetilde{\Phi}_h:~(x_n,y_n)\mapsto(x_{n+1},y_{n+1})$
could be proved similarly.
This completes the proof.
\end{proof}

Theorem~\ref{theorem6} definitely shows that all symplectic extended
phase space methods derived by  Pihajoki's or Tao's approach are
also symplectic in the original phase space. However, as numerically
verified by many experiments, the numerical performance of
Pihajoki's or Tao's explicit symplectic methods is not desirable as
good as classical symplectic methods whose linear error growth holds
for a long time. In particular, Pihajoki's methods without any
mixing step remain accurate only for a very short time. These
numerical observations seem to violate the theoretical
good performance of symplectic methods. What is the reason for this
phenomenon? To address this key issue and provide a convincing
interpretation, we present a backward error analysis in the next
section, which profoundly depends on both the existence of the
global modified Hamiltonian and the regular or chaotic motion of the
system.

\subsection{Symmetry of symplectic integrators with projection}
In  this subsection, we turn to the symmetry of the previously
discussed symplectic methods. Similarly to the symplecticity, the
symmetry should also be considered in two different senses, i.e., in
the extended phase space and in the original phase space. If
$(\Delta p,\Delta q,\Delta x,\Delta y)=(\delta p,\delta q,\delta
x,\delta y)$ and $\Phi_h$ is symmetric, the symplectic integrator
with generalized projection defined by Definition~\ref{definition3}
becomes the symmetric projection method proposed in
\cite{Jayawardana2023}. The results from \cite{Hairer2000} or
\cite{Jayawardana2023} illustrate the symmetry of the mapping
$\overline{\Phi}_h:~(p_n,p_n,q_n,q_n)\mapsto(p_{n+1},p_{n+1},q_{n+1},q_{n+1})$
in the extended phase space. Then, the symmetry of the induced
integrator $\widetilde{\Phi}_h:~(p_n,q_n)\mapsto(p_{n+1},q_{n+1})$
in the original phase space immediately follows.

However, for the symplectic integrator with standard projection
defined by Definition~\ref{definition2}, according to the result on
standard projection methods in \cite[pp.~161]{Hairer2006}, it is
known that the mapping $\overline{\Phi}_h:~(p_n,p_n,q_n,q_n)
\mapsto(p_{n+1},p_{n+1},q_{n+1},q_{n+1})$ determined by ${\Phi}_h$ with
standard projection is no longer symmetric in the extended phase
space even if $\Phi_h$ is symmetric, and neither is the
induced numerical integrator
$\widetilde{\Phi}_h:~(p_n,q_n)\mapsto(p_{n+1},q_{n+1})$.

For Pihajoki's extended phase space methods without mixing or Tao's
extended phase space methods, the symmetry of the numerical
integrator $\widetilde{\Phi}_h:~(p_n,q_n) \mapsto(p_{n+1},q_{n+1})$
is equivalent to that of the mapping
$\Psi^{(n+1)}_h:~(p_n,p_n,q_n,q_n)\mapsto (p_{n+1},p_{n+1},
q_{n+1},q_{n+1})$ defined by \eqref{theo6-prf-eq1}. Because of
$\Psi^{(n+1)}_h=\widetilde{T}_{n+1}\circ\Phi_h\circ
\widetilde{T}_{n}^{-1}$, the symmetry holds if and only if $\Phi_h$
is symmetric and $\widetilde{T}_{n+1}=\widetilde{T}_{n}$ holds for
all $n\in\mathbb{N}$ and arbitrary Hamiltonian $H(p,q)$.
As we have known,
$\widetilde{T}_{n+1}=\widetilde{T}_{n}$ does not always hold even
for a linear Hamiltonian. Thus, we could conclude that
$\widetilde{\Phi}_h$ is nonsymmetric for a general Hamiltonian
$H(p,q)$ even if $\Phi_h$ is symmetric.

\section{Backward error analysis of symplectic integrators}\label{sec:error}
This section mainly concerns the backward error analysis of the
proposed symplectic methods for nonseparable Hamiltonian systems.
The main result is that some of the explicit symplectic integrators
proposed in this paper naturally possess a linear growth of global
errors and a near-preservation of first integrals for (near-)
integrable Hamiltonian systems and for regular trajectories in the
regular region of nonintegrable Hamiltonian systems. Moreover, the
effective estimated time interval of linear error growth is
essential. The effective estimated interval is similar to that of
the classical implicit symplectic integrators for the case of
(near-) integrable Hamiltonian systems, while it will be reduced by
a factor of the radius of the regular region for nonintegrable
Hamiltonian systems. The backward error analysis will be made in two
different cases: (i) integrable or near-integrable Hamiltonian
systems and (ii) nonintegrable Hamiltonian systems.

\subsection{Existence of global modified Hamiltonian}\label{sec:modifiedH}
As analyzed in the previous sections, there exist many types of
explicit symplectic integrators including Pihajoki's or Tao's
extended phase space methods for general nonseparable Hamiltonian
systems. However, as revealed in the literature, the
actual performance of Pihajoki's or Tao's methods is not as good as
classical symplectic methods constructed in the original phase
space, such as the symplectic RK method. As is known, the linear
growth of global errors and the uniform bound of energy errors are
the most important advantages for symplectic integrators. Hence, we
should make it clear which explicit symplectic integrators
constructed in the previous sections have such good properties. On
noting that the linear growth of global errors proved by the
backward error analysis in \cite{Hairer2006} requires the existence
of the globally well-defined modified Hamiltonian for the
corresponding symplectic integrator, we will provide some special
explicit symplectic integrators and prove the existence of the
global modified Hamiltonian in this subsection.

\begin{definition}[Symplectic integrator with single factor]
\label{definition4} Assuming we have an $r$th-order
symplectic extended phase space integrator
$\Phi_h:~T^*\mathbb{R}^{2d}\rightarrow T^*\mathbb{R}^{2d}$ for the
extended Hamiltonian system $\Gamma(p,x,q,y)$ in
\eqref{extend-Hamilton} and a positive scalar $\lambda_0$,
the symplectic integrator
$\widetilde{\Phi}_h:~(p_n,q_n)\mapsto (p_{n+1},q_{n+1})$ for the
Hamiltonian system $H(p,q)$ with the initial conditions $(p_0,q_0)$
is induced and defined as follows:
\begin{enumerate}
  \item $(\tilde{p}_{n+1},\tilde{x}_{n+1},\tilde{q}_{n+1},\tilde{y}_{n+1})
  := \Phi_h(p_n,p_n,q_n,q_n)$;
  \item set $(p_{n+1},q_{n+1}) :=\lambda_0\cdot\big(\tilde{p}_{n+1},\tilde{q}_{n+1})
  + (1-\lambda_0)\cdot(\tilde{x}_{n+1},\tilde{y}_{n+1})$;
  \item define $(p_{n+1},q_{n+1}):=\widetilde{\Phi}_h(p_n,q_n)$.
\end{enumerate}
\end{definition}

\begin{definition}[Symplectic integrator with double factors]
\label{definition5} Assuming we have an $r$th-order
symplectic extended phase space integrator
$\Phi_h:~T^*\mathbb{R}^{2d}\rightarrow T^*\mathbb{R}^{2d}$ for the
extended Hamiltonian system $\Gamma(p,x,q,y)$ in
\eqref{extend-Hamilton} and two positive scalars $\lambda_0$ and
$\mu_0$, the symplectic integrator
$\widetilde{\Phi}_h:~(p_n,q_n)\mapsto (p_{n+1},q_{n+1})$ for the
Hamiltonian system $H(p,q)$ with the initial conditions $(p_0,q_0)$
is induced and defined as follows:
\begin{enumerate}
  \item $(\tilde{p}_{n+1},\tilde{x}_{n+1},\tilde{q}_{n+1},\tilde{y}_{n+1})
  := \Phi_h(p_n,p_n,q_n,q_n)$;
  \item if $n$ is odd, set
  $$(p_{n+1},q_{n+1}) :=\big(\lambda_0\tilde{p}_{n+1}+(1-\lambda_0)\tilde{x}_{n+1},
  \mu_0\tilde{q}_{n+1}+(1-\mu_0)\tilde{y}_{n+1}\big),$$
  if $n$ is even, set
  $$(p_{n+1},q_{n+1}) :=\big(\mu_0\tilde{p}_{n+1}+(1-\mu_0)\tilde{x}_{n+1},
  \lambda_0\tilde{q}_{n+1}+(1-\lambda_0)\tilde{y}_{n+1}\big);$$
  \item define $(p_{n+1},q_{n+1}):=\widetilde{\Phi}_h(p_n,q_n)$.
\end{enumerate}
\end{definition}
The single-factor symplectic integrator defined by
Definition~\ref{definition4} is just a special case of
Definition~\ref{definition1} by setting
$\lambda=\xi=(\lambda_0,\ldots,\lambda_0)\in\mathbb{R}^d$ and also a
special case of Definition~\ref{definition2} by admitting the
setting \eqref{translation-set} with
$\lambda=\xi=(\lambda_0,\ldots,\lambda_0)\in\mathbb{R}^d$. The
double-factor symplectic integrator defined by
Definition~\ref{definition5} is an extension of
Definition~\ref{definition4} by using two different scalars
respectively for odd and even steps. In particular, if
$\lambda_0=\mu_0$, the double-factor symplectic integrator reduces
to the single-factor symplectic integrator. In general, we select
the values of the weight coefficients $\lambda_0$ and $\mu_0$ on the
interval $[0,1]$ as the approximation of
$(\tilde{p}_{n+1},\tilde{q}_{n+1})$ to the  exact solution
$\big(p(t_{n+1}),q(t_{n+1})\big)$ is in the same magnitude of
$\mathcal{O}(h^r)$ as $(\tilde{x}_{n+1},\tilde{y}_{n+1})$. We also
note that both the  single-factor symplectic integrator  and the
double-factor symplectic integrator are explicit. The following
lemma is involved in the global modified Hamiltonian of the
single-factor or double-factor symplectic integrator defined above.
\begin{lemma}\label{lemma1}
Suppose that the Hamiltonian $H(p,q)$ is smooth on the open set
$D\subset T^{*}\mathbb{R}^{d}$. Then, there exists a globally
well-defined modified Hamiltonian $\widetilde{H}(p,q)$ for the
single-factor symplectic integrator defined by
Definition~\ref{definition4} and it is smooth on the same open set
$D$ as $H(p,q)$. It also holds for the two-factor symplectic
integrator defined by Definition~\ref{definition5}.
\end{lemma}
\begin{proof}
According to \cite{Pihajoki2015}, the explicit extended phase space
method $\Phi_h$ can be rewritten as a PRK method. Let $u=(p,y)$,
$v=(x,q)$, and $f=J^{-1}\nabla H$, then the canonical equations
corresponding to the extended Hamltonian $\Gamma(p,x,q,y)$ in
\eqref{extend-Hamilton} could be written as
\begin{equation}\label{lemma1-prf-eq1}
\begin{aligned}
\dot{u}=f(v),\qquad
\dot{v}=f(u).
\end{aligned}
\end{equation}
We then consider the P-series of PRK methods based on the bi-colored
trees. Let $\mathcal{TP}_u$ be the set of bi-colored trees whose
root is $\bullet$ that corresponds to the derivatives of $u$,
$\mathcal{TP}_v$ be the set of bi-colored
 trees whose root is
$\circ$ that corresponds to the derivatives of $v$, and
$\mathcal{TP}=\mathcal{TP}_u\cup\mathcal{TP}_v$. The colors of
$\bullet$ and $\circ$ are called, respectively, black and white.
Other definitions such as the elementary differential
$\mathcal{F}(\tau)$, the symmetry coefficient $\sigma(\tau)$, and
the order $|\tau|$ of the bi-colored tree $\tau\in\mathcal{TP}$ can
be found in \cite{Hairer2006} and are not introduced in detail here.

Let $\mathcal{STP}_u\subset \mathcal{TP}_u$ be the set of bi-colored
trees whose root is black, white vertices have only black sons, and
black vertices have only white sons. The set $\mathcal{STP}_v$ is
correspondingly defined. Then, the vertices in the same generation
(or level) of a tree in $\mathcal{STP}_u$ or $\mathcal{STP}_v$ have
the same color. Moreover, for each bi-colored tree
$\tau\in\mathcal{STP}_u$, there exists a unique bi-colored tree
$\tau^*\in\mathcal{STP}_v$ such that $\tau^*$ is obtained by
changing the color (from white to black or from black to white) of
all vertices of $\tau$, and vice versa. We say that
such a $\tau^*$ is adjoint to $\tau$.

Denote $\Phi_h:~(u,v)\mapsto(\tilde{u},\tilde{v})$,
$\tilde{u}=(\tilde{p},\tilde{y})$, and $\tilde{v}=(\tilde{x},\tilde{q})$.
Because the numerical method
$\Phi_h$ can be expressed as a PRK method, according to the particular formulation
of the differential equations in \eqref{lemma1-prf-eq1} and Theorem~III.2.4
of \cite{Hairer2006}, the numerical solution of $\Phi_h$ is a P-series as follows:
\begin{equation}\label{lemma1-prf-eq2}
\begin{aligned}
\tilde{u}=u +
\sum_{\tau\in\mathcal{STP}_u}\frac{h^{|\tau|}}{\sigma(\tau)}
a(\tau)\mathcal{F}(\tau)(u,v),
\\
\tilde{v}=v +
\sum_{\tau\in\mathcal{STP}_v}\frac{h^{|\tau|}}{\sigma(\tau)}
a(\tau)\mathcal{F}(\tau)(u,v),
\end{aligned}
\end{equation}
where the mapping
$a:\mathcal{TP}\cup\{\emptyset_u,\emptyset_v\}\rightarrow
\mathbb{R}$ is defined as that in \cite[Theorem~III.2.4]{Hairer2006}
and we omit the details here.

Let $\mathcal{T}$ be the set of rooted trees whose vertices are all
black. For each rooted tree $\theta\in\mathcal{T}$, there exists a
unique $\tau\in\mathcal{STP}_u$ and its adjoint
$\tau^*\in\mathcal{STP}_v$ such that $\theta$ could be obtained by
changing the color of all vertices of $\tau$ or $\tau^*$ to black
\cite{Sanz-Serna1994}. We denote such three rooted or bi-colored
trees $\theta\in \mathcal{T}$, $\tau\in \mathcal{STP}_u$, and
$\tau^*\in \mathcal{STP}_v$ as a triple $(\theta,\tau,\tau^*)$. From
the definitions of the symmetry coefficient respectively for the
rooted tree and the bi-colored tree, the equality
\begin{equation}\label{lemma1-prf-eq3}
\sigma(\theta)=\sigma(\tau)=\sigma(\tau^*),
\end{equation}
holds for all such triples $(\theta,\tau,\tau^*)$.

It follows from the results of \cite[pp.~46]{Sanz-Serna1994} that the equality
\begin{equation}\label{lemma1-prf-eq4}
a(\tau)=a(\tau^*)
\end{equation}
holds for all adjoint bi-colored trees once $|\tau|=|\tau^*|\leq
r$. By induction, we could further deduce that
\eqref{lemma1-prf-eq4} holds for all $\tau\in\mathcal{STP}_u$ and
the adjoint $\tau^*\in\mathcal{STP}_v$. Because of the particular
formulation of \eqref{lemma1-prf-eq1}, once $u=v=(p,q)$, it
immediately follows from the result \cite[pp.~46]{Sanz-Serna1994}
that
$\mathcal{F}(\tau)(u,v)=\mathcal{F}(\tau^*)(u,v)=\mathcal{F}(\theta)(p,q)$
holds for the triple  $(\theta,\tau,\tau^*)$, where
$\mathcal{F}(\theta)(p,q)$ is the elementary differential
corresponding to the rooted tree $\theta\in\mathcal{T}$.

Let $\mathcal{F}^{[1]}$ denote the vector in $\mathbb{R}^d$ whose
entries are the first $d$ entries of
$\mathcal{F}$, and $\mathcal{F}^{[2]}$ denote the vector in
$\mathbb{R}^d$ whose entries are the last $d$
entries of $\mathcal{F}$. On noting
$\tilde{u}=(\tilde{p},\tilde{y})$ and
$\tilde{v}=(\tilde{x},\tilde{q})$, it immediately follows from the
equations \eqref{lemma1-prf-eq2}, \eqref{lemma1-prf-eq3}, and
\eqref{lemma1-prf-eq4} that
\begin{equation}\label{lemma1-prf-eq5}
\begin{aligned}
\widehat{p}:=\lambda_0\tilde{p} + (1-\lambda_0)\tilde{x} =p +
\sum_{\theta\in\mathcal{T}}\frac{h^{|\theta|}}{\sigma(\theta)}
a'(\theta)\mathcal{F}^{[1]}(\theta)(p,q),
\\
\widehat{q}:=\lambda_0\tilde{y} + (1-\lambda_0)\tilde{q} =q +
\sum_{\theta\in\mathcal{T}}\frac{h^{|\theta|}}{\sigma(\theta)}
a'(\theta)\mathcal{F}^{[2]}(\theta)(p,q),
\end{aligned}
\end{equation}
once $u=v=(p,q)$, where the coefficient $a'(\theta)$ is defined by
$a'(\theta):=a(\tau)=a(\tau^*)$ for the triple $(\theta,\tau,\tau^*)$.
Let $\widehat{u}=(\widehat{p},\widehat{q})$, we then can express
the numerical method $\widetilde{\Phi}_h:~u\mapsto \widehat{u}$ by
the following B-series
\begin{equation*}
\widehat{u} = u + \sum_{\theta\in\mathcal{T}}\frac{h^{|\theta|}}
{\sigma(\theta)}a'(\theta)\mathcal{F}(\theta)(u).
\end{equation*}
It finally follows from \cite[Theorem IX.9.8]{Hairer2006} that there
exists a globally well-defined modified Hamiltonian
$\widetilde{H}(p,q)$.

Likewise, the numerical solution of the two-factor symplectic
integrator could be expressed as a P-series, and the existence of
the global modified Hamiltonian $\widetilde{H}(p,q)$ follows from
\cite[Theorem IX.10.9]{Hairer2006}. This completes the proof.
\end{proof}

\begin{remark}\label{remark5}
Because the global modified (original) Hamiltonian is essential to
the linear error growth of symplectic integrators, the lack of
global modified (original) Hamiltonian may also result in a lack of
linear error growth. In this lemma, we only prove the existence of
global modified (original) Hamiltonian for the single-factor or
double-factor symplectic integrator defined in this subsection.
Nevertheless, we are failing to confirm the existence of global
modified Hamiltonian for general standard/generalized projection
symplectic integrators. Due to the multiplicity of the translation
vectors $(\delta p,\delta x,\delta q,\delta y)$ and $(\Delta
p,\Delta x,\Delta q,\Delta y)$ in determining the
standard/generalized projection symplectic integrators with a given
symplectic extended phase space integrator $\Phi_h$, there must
exist such standard/generalized projection symplectic integrators
that cannot possess the linear error growth, and this point is
illustrated by the numerical experiment in
Section~\ref{sec:num}, where the quadratic error growth
of a special standard projection symplectic integrator applied to a
completely integrable Hamiltonian system is shown. Being specific to
Pihajoki's or Tao's extended phase space methods, the existence of
not the global modified original Hamiltonian but the global modified
extended Hamiltonian could be assured for $\Phi_h$ according to the
result in \cite{Yoshida1993} and the fact that $\Phi_h$ is a
splitting method. Hence, this fact may account for the short-term
good performance of Pihajoki's or Tao's methods. The detailed
analysis on Pihajoki's or Tao's methods will be shown in the next
subsection.
\end{remark}

\subsection{Integrable and near-integrable Hamiltonian systems}
Now, we assume that the Hamiltonian $H(p,q)$ is near-integrable, that is,
\begin{equation*}
H(p,q)=H_0(p,q) + \varepsilon H_1(p,q),
\end{equation*}
where $H_0(p,q)$ is completely integrable, $ \varepsilon$ is a small
parameter, and $H(p,q)$ is real-analytic. The Hamiltonian $H(p,q)$
reduces to a completely integrable system once
$\varepsilon\rightarrow0$. Suppose that $H_0(p,q)$ possesses $d$
isolated first integrals $I_j(p,q)~(j=1,\ldots,d)$. For the
near-integrable Hamiltonian $H(p,q)$ with $\varepsilon>0$, we assume
that there exist $d'~(<d)$ isolated integrals
$\widetilde{I}_j(p,q)~(j=1,\ldots,d')$.

Since $H_0(p,q)$ is completely integrable, there exists a symplectic
change of coordinates $(p,q)\mapsto(a,\theta)$, where
$(a,\theta)=(a^1,
\ldots,a^d,\theta^1~\text{mod}~2\pi,\ldots,\theta^d~\text{mod}~2\pi)$
are action-angle variables such that the Hamiltonian $H_0(p,q)$
becomes $\mathcal{H}_0(a)$ whose canonical equations are
\begin{equation*}
\dot{a}=0,\qquad \dot{\theta}=\omega(a),
\end{equation*}
where $\omega(a)=\frac{\partial\mathcal{H}_0}{\partial a}$ and
$\mathcal{H}_0(a)=H_0(p,q)$. For convenience, we will still use
$H_0(a)$ instead of $\mathcal{H}_0(a)$ to denote the new Hamiltonian
of $H_0(p,q)$ in the coordinates $(a,\theta)$. In the following
contents, the same strategy is applied for the Hamiltonian in
different coordinates.

We rewrite the extended Hamiltonian $\Gamma(p,x,q,y)$
in \eqref{extend-Hamilton} as follows:
\begin{equation*}\label{extend-Hamilton-2}
\Gamma(p,x,q,y)=H(p,y) + H(x,q),
\end{equation*}
with the conjugated variables $(p,q)$ and $(x,y)$. Due to the
interaction between $(p,q)$ and $(x,y)$ in the Hamiltonian
$\Gamma(p,x,q,y)$, the functions $\widetilde{I}_j(p,q) +
\widetilde{I}_j(x,y)$ or $\widetilde{I}_j(p,y)
+\widetilde{I}_j(x,q)$ is probably no longer the first integrals of
$\Gamma(p,x,q,y)$,  which means the nonintegrability of
$\Gamma(p,x,q,y)$. In particular,  even if the original Hamiltonian
$H(p,q)$ is completely integrable (i.e., $\varepsilon=0$) with $d$
isolated first integrals $I_j(p,q)~(j=1\ldots,d)$, some of $I_j(p,q)
+ I_j(x,y)$ are probably no longer the first integrals of
$\Gamma(p,x,q,y)$, thereby leading to the  nonintegrability of
$\Gamma(p,x,q,y)$. A simple case addressing the nonintegrability and
chaoticity of $\Gamma(p,x,q,y)$ with a completely integrable
$H(p,q)$ ( i.e., $\varepsilon=0$) can be found in \cite{Tao2016b}.

In spite of the nonintegrable case for $\Gamma(p,x,q,y)$, there
still exist occasions the integrability of the extended Hamiltonian
$\Gamma(p,x,q,y)$ remains the same as the original  Hamiltonian
$H(p,q)$. For example, if $H(p,q)$ is separable in terms of $(p,q)$,
i.e., $H(p,q)=T(p)+V(q)$, then we have
$$\Gamma(p,x,q,y)=T(p)+V(y) + T(x)+V(q)=H(p,q) + H(x,y),$$
which is just two independent copies of the Hamiltonian $H(p,q)$. A
simple case illustrating this point is the  harmonic oscillator
$H(p,q)=\frac{1}{2}(p^2+q^2)$. Moreover, explicit symplectic
integrators could be directly constructed by the splitting of
$H(p,q)$ in this case.

In fact, if $H(p,q)$ is nonseparable, the extended
Hamiltonian $\Gamma(p,x,q,y)$ is most likely to be nonintegrable or
chaotic. As the explicit symplectic integrators proposed in this
paper are based on the symplectic extended phase space method for
$\Gamma(p,x,q,y)$, we cannot derive the linear error growth of the
single-factor or double-factor symplectic integrator directly via
\cite[Theorem~X.3.1]{Hairer2006} because of the lack of (near-)
integrability for $\Gamma(p,x,q,y)$. To apply the backward error
analysis to both the symplectic extended phase space method
and the explicit symplectic integrator proposed in this paper, we
introduce a parameter relaxation and restriction technique that
bounds the discrepancy $\|(p-x,q-y)\|$ to a small parameter
$\varepsilon$ firstly, then drives the Taylor series or truncated
Lindstedt--Poincar\'{e} series of the extended Hamiltonian or
modified extended Hamiltonian in terms of the small parameter
$\varepsilon$, and finally bounds the numerical discrepancy
$\|(p_n-x_n,q_n-y_n)\|$.

According to the definition of the submanifold $\mathcal{N}$ in
\eqref{manifold}, we define the open set
$\mathcal{B}_{\varepsilon}(\mathcal{N})$ as follows:
\begin{equation}\label{openSet}
\mathcal{B}_{\varepsilon}(\mathcal{N}):=
\Big\{(p,x,q,y)\in T^*\mathbb{R}^{2d}\,\big|\,\|(p-x,q-y)\|<\varepsilon
\Big\}\subset T^*\mathbb{R}^{2d},
\end{equation}
which can be regarded as a neighbourhood of the submanifold
$\mathcal{N}$ with the radius $\varepsilon>0$. We consider to
restrict the extended Hamiltonian $\Gamma(p,x,q,y)$ in the open set
$\mathcal{B}_{\varepsilon}(\mathcal{N})$. In view of the analyticity
of $H(p,q)$, we express $\Gamma(p,x,q,y)$ as follows:
\begin{equation*}
\begin{aligned}
\Gamma(p,x,q,y)=
&H(p,q)  +\sum_{\alpha=1}^{\infty}
\frac{\mathcal{D}_q^\alpha H(p,q)}{\alpha!}(y-q)^\alpha
\\
&+ H(x,y)+ \sum_{\alpha=1}^{\infty}
\frac{\mathcal{D}_y^\alpha H(x,y)}{\alpha!}(q-y)^\alpha,
\end{aligned}
\end{equation*}
where $\mathcal{D}_q=
\dfrac{\partial ^{|\alpha|}}{(\partial q^1)^{\alpha_1}
\cdots(\partial q^d)^{\alpha_d}}$,
$\mathcal{D}_y=\dfrac{\partial ^{|\alpha|}}{(\partial y^1)^{\alpha_1}
\cdots(\partial y^d)^{\alpha_d}}$,
$\alpha=(\alpha_1,\ldots,\alpha_d)\in\mathbb{N}^d$,
$|\alpha|=\sum_{k=1}^{d}\alpha_k$, $\alpha!=\prod_{k=1}^{d}\alpha_k!$,
and $(q-y)^\alpha=\prod_{k=1}^{d}(q^k-y^k)^{\alpha_k}$.
As $(p,x,q,y)\in\mathcal{B}_{\varepsilon}(\mathcal{N})$,
the extended Hamiltonian $\Gamma(p,x,q,y)$ can be expressed as
\begin{equation}\label{nearInt-extH}
\Gamma(p,x,q,y) = H_0(p,q) + H_0(x,y)
+ \varepsilon \big( H_1(p,q)+H_1(x,y)+G(p,x,q,y)\big),
\end{equation}
where $H_1$ and $G$ could be derived according to the Taylor expansion of
$\Gamma(p,x,q,y)$.

It is easy to see from \eqref{nearInt-extH} that
$\Gamma(p,x,q,y)$ is a perturbed integrable system for sufficiently
small $\varepsilon$ once $(p,x,q,y)\in
\mathcal{B}_{\varepsilon}(\mathcal{N})$. That is, although the
extended Hamiltonian $\Gamma(p,x,q,y)$ may be not a near-integrable
system globally defined on the open set $D$, it is near-integrable
once the solution is restricted to
$\mathcal{B}_{\varepsilon}(\mathcal{N})$ for sufficiently small
$\varepsilon$. Considering this point, we present the following
theorem regarding the backward error analysis of symplectic extended
phase space methods.
\begin{theorem}\label{theorem7}
Consider applying an $r$th-order $(r\geq1)$ symplectic integrator
$$\Phi_h:~(p_n,x_n,q_n,y_n)\mapsto({p}_{n+1},{x}_{n+1},{q}_{n+1},{y}_{n+1})$$
to the perturbed integrable Hamiltonian system \eqref{nearInt-extH},
whose $\tilde{d}~(\leq 2d)$ first integrals are
$I^*_{1}(p,x,q,y),\ldots,I^*_{\tilde{d}}(p,x,q,y)$. Suppose that the
modified Hamiltonian of $\Phi_h$ is globally well-defined on the
same open set as the extended Hamiltonian $\Gamma(p,x,q,y)$ and
$\omega(a^*)$ satisfies the diophantine condition
\begin{equation}\label{diophantine}
|k\cdot\omega|\geq\gamma|k|^{-\nu},\qquad k\in\mathbb{Z}^d,~k\neq0
\end{equation}
for some positive constants $\gamma$, $\nu$. Then, there exist
positive constants $C_0$, $c$, $h_0$, and $\varepsilon_0$ such that
the following results hold for the parameter
$\varepsilon\leq\varepsilon_0$ and all stepsizes
$h\leq\min\{h_0,\varepsilon_0^{1/r}\}$: every numerical solution
starting with $\|I(p_0,q_0)-a^*\|+\|I(x_0,y_0) -a^*\|\leq c|\log
h|^{-(\nu + 1)}$, $(p_0,x_0,q_0,y_0)\in\mathcal{B}_{\varepsilon_0}
(\mathcal{N})$, and $\delta_0=\|(p_0-x_0,q_0-y_0)\|<\varepsilon_0$
satisfies
\begin{subequations}
\begin{eqnarray}\label{theorem7-equ-parent}
&&\big\|(p_n,x_n,q_n,y_n)-\big(p(t),x(t),q(t),y(t)\big)\big\|\leq C_0th^r,
\label{theorem7-subequ1}
\\
&&\big|I_j^*(p_n,x_n,q_n,y_n)-I_j^*(p_0,x_0,q_0,y_0)\big|\leq
C_0h^r,\quad j=1,\ldots,\tilde{d},
\label{theorem7-subequ2}
\end{eqnarray}
\end{subequations}
for $t=nh\leq\min\{h^{-r},\varepsilon_0/\delta_0\}$  with such $n$
that $(p_n,x_n,q_n,y_n)\in\mathcal{B}_{\varepsilon_0}$,
where the constants $C_0$, $c$, $h_0$, $\varepsilon_0$
depend on $d$, $\gamma$, $\nu$, on bounds of the real-analytic
Hamiltonian $H(p,q)$ on a complex neighbourhood of the torus
$\{(p,q):~I(p,q)=a^*\}$, and on the numerical method $\Phi_h$.
\end{theorem}
\begin{proof}
(a) We first consider the simple case where the
parameter $\varepsilon>0$ is free and independent of the discrepancy
$\|(p-x,q-y)\|$. In general, it is reasonable to assume that
$\varepsilon,h\in(0,1)$. Therefore, for any $\varepsilon,h\in(0,1)$
there must exist a constant $\beta>0$ such that
$h^{\beta+1}<\varepsilon\leq h^{\beta}$. In addition, for
$\delta_0\in[0,\varepsilon_0)$, we let $\delta_0^{-1}= +\infty$ and
$\mu<\delta_0^{-1}$ for any $\mu\in\mathbb{R}$ once
$\delta_0\rightarrow 0^+$.

According to \cite[Lemma~X.2.1]{Hairer2006}, for the fixed
$N\geq\{\frac{2r}{\beta}+1,\frac{3r}{\beta},2\}\geq2$, there exists
a positive number $\varepsilon_0$ such that for any
$\varepsilon\leq\varepsilon_0$ there exists a real-analytic
symplectic change of coordinates
$(p,x,q,y)\mapsto(a,a',\theta,\theta')$ such that the Hamiltonian
$\Gamma(p,x,q,y)$ in \eqref{nearInt-extH} has the following
expression in the new coordinates $(a,a',\theta,\theta')$:
\begin{equation}\label{theo7-prf-eq1}
\Gamma(a,a',\theta,\theta')=H_0(a) + H_0(a') +
\sum_{j=1}^{N-1}\varepsilon^{j}\overline{K}_j(a,a')
+ \varepsilon^{N}R_{N}(a,a',\theta,\theta').
\end{equation}
Because $\Gamma(p,x,q,y)$ is symmetric with respect to
the variables $(p,q)$ and $(x,y)$, i.e.,
$\Gamma(p,x,q,y)=\Gamma(x,p,y,q)$, so are the functions
$\overline{K}_j(a,a')$ and $R_{N}(a,a',\theta,\theta')$, i.e.,
$\overline{K}_j(a,a')=\overline{K}_j(a',a)$ and
$R_{N}(a,a',\theta,\theta')=R_{N}(a',a,\theta',\theta)$.

Let $$H_{\varepsilon,N}(a,a')=H_0(a) + H_0(a') +
\sum_{j=1}^{N-1}\varepsilon^{j}\overline{K}_j(a,a'),\quad
\omega_{\varepsilon,N}(a,a')=\frac{\partial H_{\varepsilon,N}}{\partial a},$$
then the canonical equations corresponding to \eqref{theo7-prf-eq1} read
\begin{equation}\label{theo7-prf-eq2}
\begin{aligned}
&\frac{\mathrm{d}a}{\mathrm{d}t}=-\varepsilon^{N}\frac{\partial R_{N}}{\partial \theta},
\qquad
&&\frac{\mathrm{d}a'}{\mathrm{d}t}=-\varepsilon^{N}\frac{\partial R_{N}}{\partial \theta'},
\\
&\frac{\mathrm{d}\theta}{\mathrm{d}t}=\omega_{\varepsilon,N}(a,a')
+\varepsilon^{N}\frac{\partial R_{N}}{\partial a},
\qquad
&&\frac{\mathrm{d}\theta'}{\mathrm{d}t}
= \omega_{\varepsilon,N}(a',a)+\varepsilon^{N}\frac{\partial R_{N}}{\partial a'}.
\end{aligned}
\end{equation}
We derive with \cite[Lemma X.2.1]{Hairer2006} that
\begin{equation}\label{theo7-prf-eq2.2}
\begin{aligned}
&\|a(t)-a(0)\|\leq Ct\varepsilon^{N},
\\
&\|a'(t)-a'(0)\|\leq Ct\varepsilon^{N},
\end{aligned}
\end{equation}
for $t\leq \varepsilon^{-N+1}$ and
\begin{equation}\label{theo7-prf-eq3}
\begin{aligned}
&\|\theta(t)-\omega_{\varepsilon,N}\big(a(0),a'(0)\big)t-
\theta(0)\|\leq C(t^2 +t|\log\varepsilon|^{\nu+1})\varepsilon^{N},
\\
&\|\theta'(t)-\omega_{\varepsilon,N}\big(a'(0),a(0)\big)t-
\theta'(0)\|\leq C(t^2 +t|\log\varepsilon|^{\nu+1})\varepsilon^{N},
\end{aligned}
\end{equation}
for $t^2\leq \varepsilon^{-N+1}$.

 As mentioned in Remark~\ref{remark5}, the symplectic
integrator $\Phi_h$ is usually selected as the splitting method in
the extended phase space and thus we admit the assumption that the
modified Hamiltonian of $\Phi_h$ is globally well-defined on the
same open set as the extended Hamiltonian $\Gamma(p,x,q,y)$. We next
consider the truncated modified Hamiltonian of the numerical method
$\Phi_h$:
\begin{equation}\label{theo7-prf-eq4}
\widetilde{\Gamma}(p,x,q,y)=\Gamma(p,x,q,y) + h^{r}H_{r+1}(p,x,q,y)
+\cdots+h^{s}H_{s+1}(p,x,q,y),
\end{equation}
whose solution with initial values $(p_0,x_0,q_0,y_0)$ is denoted
by $\big(\widetilde{p}(t),\widetilde{x}(t),\widetilde{q}(t),\widetilde{y}(t)\big)$.
The inequality $h\leq\varepsilon_0^{1/r}$ indicates $h^r\leq\varepsilon_0$. Thus, with the
previously fixed $N$ and the new coordinates $(a,a',\theta,\theta')$, the Hamiltonian
$\widetilde{\Gamma}(p,x,q,y)$ becomes
\begin{equation}\label{theo7-prf-eq5}
\widetilde{\Gamma}(a,a',\theta,\theta')=H_{\varepsilon,N}(a,a')+
 \varepsilon^{N}R_{N}(a,a',\theta,\theta')
+h^{r}E_{r+1}(a,a',\theta,\theta'),
\end{equation}
where $E_{r+1}$ satisfies
$E_{r+1}(p,x,q,y)=\sum_{j=0}^{s-r}h^{j}H_{j+r+1}(p,x,q,y)$ in the
original coordinates $(p,x,q,y)$. Because of
$N\geq\frac{2r}{\beta}+1$ and $h^{\beta+1}<\varepsilon \leq
h^\beta$, we have $\varepsilon^N\leq h^{\beta N}\leq
h^{2r+\beta}\leq h^r$. By letting $\widetilde{\varepsilon}=h^r$, the
Hamiltonian $\widetilde{\Gamma}(a,a',\theta,\theta')$ in
\eqref{theo7-prf-eq5} could be expressed as
\begin{equation}\label{theo7-prf-eq5.2}
\widetilde{\Gamma}(a,a',\theta,\theta')=H_{\varepsilon,N}(a,a')+
\widetilde{\varepsilon}{E}(a,a'\theta,\theta'),
\end{equation}
where ${E}(a,a',\theta,\theta')=
E_{r+1}(a,a',\theta,\theta')
+\frac{\varepsilon^{N}}{\widetilde{\varepsilon}}R_{N}(a,a',\theta,\theta')$.

Applying Lemma~X.2.1 of \cite{Hairer2006} to the Hamiltonian
\eqref{theo7-prf-eq5.2} once again derives that for any fixed
$\widetilde{N}\geq3$ there exist
$\widetilde{\varepsilon}_0>0$ and a symplectic change of
coordinates $(a,a',\theta,\theta')\mapsto(b,b',\varphi,\varphi')$
that is $\mathcal{O}(\widetilde{\varepsilon})$-close to the
identity, such that for any $\widetilde{\varepsilon}\leq
\widetilde{\varepsilon}_0$, i.e., $h\leq h_0:=
\widetilde{\varepsilon}_0^{1/r}$, the modified Hamiltonian
\eqref{theo7-prf-eq4} in the new coordinates
$(b,b',\varphi,\varphi')$ is of the form
\begin{equation}\label{theo7-prf-eq6}
\widetilde{\Gamma}({b},{b}',{\theta},
{\theta}')=H_{\varepsilon,N}({b},{b}')+\sum_{j=1}^{\widetilde{N}-1}
\widetilde{\varepsilon}^{j}\overline{J}_j(b,b')+
 \widetilde{\varepsilon}^{\widetilde{N}}\widetilde{R}_{\widetilde{N}}
 (b,b',\varphi,\varphi').
\end{equation}
If we denote the solution
$\big(\widetilde{p}(t),\widetilde{x}(t),
\widetilde{q}(t),\widetilde{y}(t)\big)$ of the modified Hamiltonian
$\widetilde{\Gamma}$ in the new coordinates
$(b,b',\varphi,\varphi')$ by
$(\widetilde{b},\widetilde{b}',\widetilde{\varphi},\widetilde{\varphi}')$
and let
\begin{equation*}
H_{h,\varepsilon,N}({b},{b}')=H_{\varepsilon,N}({b},{b}')
+\sum_{j=1}^{\widetilde{N}-1}\widetilde{\varepsilon}^{j}\overline{J}_j(b,b'),
\quad
\omega_{h,\varepsilon,N}(b,b')=\frac{\partial H_{h,\varepsilon,N}}{\partial b},
\end{equation*}
then the canonical equations read
\begin{equation}\label{theo7-prf-eq7}
\begin{aligned}
&\frac{\mathrm{d}\widetilde{b}}{\mathrm{d}t}=-\widetilde{\varepsilon}^{\widetilde{N}}
\frac{\partial \widetilde{R}_{\widetilde{N}}}{\partial \widetilde{\varphi}},
\qquad
&&\frac{\mathrm{d}\widetilde{b}'}{\mathrm{d}t}=-\widetilde{\varepsilon}^{\widetilde{N}}
\frac{\partial \widetilde{R}_{\widetilde{N}}}{\partial \widetilde{\varphi}'},
\\
&\frac{\mathrm{d}\widetilde{\varphi}}{\mathrm{d}t}=
\omega_{h,\varepsilon,N}(\widetilde{b},\widetilde{b}')+\widetilde{\varepsilon}^{N}
\frac{\partial \widetilde{R}_{\widetilde{N}}}{\partial \widetilde{b}},
\qquad
&&\frac{\mathrm{d}\widetilde{\varphi}'}{\mathrm{d}t}=
\omega_{h,\varepsilon,N}(\widetilde{b}',\widetilde{b})+\widetilde{\varepsilon}^{N}
\frac{\partial \widetilde{R}_{\widetilde{N}}}{\partial \widetilde{b}'},
\end{aligned}
\end{equation}
whose solution satisfies
\begin{equation}\label{theo7-prf-eq8}
\begin{aligned}
&\|\widetilde{b}(t)-\widetilde{b}(0)\|\leq
\widetilde{C}t\widetilde{\varepsilon}^{\widetilde{N}},
\\
&\|\widetilde{b}'(t)-\widetilde{b}'(0)\|\leq
\widetilde{C}t\widetilde{\varepsilon}^{\widetilde{N}},
\end{aligned}
\end{equation}
for $t\leq \widetilde{\varepsilon}^{-\widetilde{N}+1}$ and
\begin{equation}\label{theo7-prf-eq9}
\begin{aligned}
&\|\widetilde{\varphi}(t)-\omega_{h,\varepsilon,N}\big(\widetilde{b}(0),
\widetilde{b}'(0)\big)t-
\widetilde{\varphi}(0)\|\leq \widetilde{C}(t^2 +t|\log\widetilde{\varepsilon}|^{\nu+1})
\widetilde{\varepsilon}^{\widetilde{N}},
\\
&\|\widetilde{\varphi}'(t)-\omega_{h,\varepsilon,N}\big(\widetilde{b}'(0),
\widetilde{b}(0)\big)t-
\widetilde{\varphi}'(0)\|\leq \widetilde{C}(t^2 +t|\log\widetilde{\varepsilon}|^{\nu+1})
\widetilde{\varepsilon}^{\widetilde{N}},
\end{aligned}
\end{equation}
for $t^2\leq \widetilde{\varepsilon}^{-N+1}$.

Because the symplectic transformation
$(a,a',\theta,\theta')\mapsto(b,b',\varphi,\varphi')$ is
$\mathcal{O}(\widetilde{\varepsilon})$-close to the identity and
$\omega_{h,\varepsilon,N}(b,b')=\omega_{\varepsilon,N}(b,b')
+\mathcal{O}(\widetilde{\varepsilon})$ where the constant symbolized by
the $\mathcal{O}$-notation is dependent of $h$ and $t$,
the solution of the modified Hamiltonian $\widetilde{\Gamma}$ in the
coordinates $(a,a',\theta,\theta')$ satisfies
\begin{equation}\label{theo7-prf-eq10}
\begin{aligned}
&\|\widetilde{a}(t)-\widetilde{a}(0)\|\leq
\widetilde{C}t\widetilde{\varepsilon}^{\widetilde{N}} +
\mathcal{O}(\widetilde{\varepsilon}),
\\
&\|\widetilde{a}'(t)-\widetilde{a}'(0)\|\leq
\widetilde{C}t\widetilde{\varepsilon}^{\widetilde{N}}+
\mathcal{O}(\widetilde{\varepsilon}),
\end{aligned}
\end{equation}
for $t\leq \widetilde{\varepsilon}^{-\widetilde{N}+1}$ and
\begin{equation}\label{theo7-prf-eq11}
\begin{aligned}
&\|\widetilde{\theta}(t)-\omega_{\varepsilon,N}(\widetilde{a}(0),\widetilde{a}'(0))t-
\widetilde{\theta}(0)\|\leq \widetilde{C}(t^2 +t|\log\widetilde{\varepsilon}|^{\nu+1})
\widetilde{\varepsilon}^{\widetilde{N}} + te_h,
\\
&\|\widetilde{\theta}'(t)-\omega_{\varepsilon,N}
(\widetilde{a}'(0),\widetilde{a}(0))t- \widetilde{\theta}'(0)\|\leq
\widetilde{C}(t^2 +t|\log\widetilde{\varepsilon}|^{\nu+1})
\widetilde{\varepsilon}^{\widetilde{N}}+ te'_h,
\end{aligned}
\end{equation}
for $t^2\leq \widetilde{\varepsilon}^{-\widetilde{N}+1}$,
where $e_h=\|\omega_{h,\varepsilon,N}\big(\widetilde{b}(0),\widetilde{b}'(0)\big)
-\omega_{\varepsilon,N}\big(\widetilde{a}(0),\widetilde{a}'(0)\big)\|$
and $e'_h=\|\omega_{h,\varepsilon,N}\big(\widetilde{b}'(0),\widetilde{b}(0)\big)
-\omega_{\varepsilon,N}\big(\widetilde{a}'(0),\widetilde{a}(0)\big)\|$. Because of
$\omega_{h,\varepsilon,N}-\omega_{\varepsilon,N}=\mathcal{O}
(\widetilde{\varepsilon})$ and the $\mathcal{O}(\widetilde{\varepsilon})$-closeness
of the transformation between $(\widetilde{a},\widetilde{a}',\widetilde{\theta},
\widetilde{\theta}')$ and $(\widetilde{b},\widetilde{b}',\widetilde{\varphi},
\widetilde{\varphi}')$, we thus have $e_h=\mathcal{O}(\widetilde{\varepsilon})$
and $e'_h=\mathcal{O}(\widetilde{\varepsilon})$.

Since the same initial values $(p_0,x_0,q_0,y_0)$ are imposed for both
the extended Hamiltonian $\Gamma(p,x,q,y)$ and the modified Hamiltonian
$\widetilde{\Gamma}(p,x,q,y)$, it derives $a(0)=\widetilde{a}(0)$,
$a'(0)=\widetilde{a}'(0)$, $\theta(0)=\widetilde{\theta}(0)$, and
$\theta'(0)=\widetilde{\theta}'(0)$. Then, the comparison between
\eqref{theo7-prf-eq2.2}-\eqref{theo7-prf-eq3} and
\eqref{theo7-prf-eq10}-\eqref{theo7-prf-eq11} yields the difference
between the solutions of $\Gamma(p,x,q,y)$ and $\widetilde{\Gamma}(p,x,q,y)$
in the $(a,a',\theta,\theta')$ as follows:
\begin{equation}\label{theo7-prf-eq12}
\begin{aligned}
&\|\widetilde{a}(t)-{a}(t)\|\leq
\widetilde{C}t\widetilde{\varepsilon}^{\widetilde{N}}
+Ct\varepsilon^{N}+\mathcal{O}(\widetilde{\varepsilon}),
\\
&\|\widetilde{a}'(t)-{a}'(t)\|\leq
\widetilde{C}t\widetilde{\varepsilon}^{\widetilde{N}}
+Ct\varepsilon^{N}+\mathcal{O}(\widetilde{\varepsilon}),
\end{aligned}
\end{equation}
for $t\leq \min\{{\varepsilon}^{-{N}+1},
\widetilde{\varepsilon}^{-\widetilde{N}+1}\}$ and
\begin{equation}\label{theo7-prf-eq13}
\begin{aligned}
&\|\widetilde{\theta}(t)-{\theta}(t)\|\leq
\widetilde{C}(t^2 +t|\log\widetilde{\varepsilon}|^{\nu+1})
\widetilde{\varepsilon}^{\widetilde{N}}
+C(t^2 +t|\log\varepsilon|^{\nu+1})\varepsilon^{N}
+\mathcal{O}(t\widetilde{\varepsilon}),
\\
&\|\widetilde{\theta}'(t)-{\theta}'(t)\|\leq
\widetilde{C}(t^2 +t|\log\widetilde{\varepsilon}|^{\nu+1})
\widetilde{\varepsilon}^{\widetilde{N}}
+C(t^2 +t|\log\varepsilon|^{\nu+1})\varepsilon^{N}
+\mathcal{O}(t\widetilde{\varepsilon}),
\end{aligned}
\end{equation}
for $t^2\leq\min\{\varepsilon^{-N+1},\widetilde{\varepsilon}^{-\widetilde{N}+1}\}$.

Because $\varepsilon\leq h^\beta$ and
$N\geq\{\frac{2r}{\beta}+1,\frac{3r}{\beta},2\}$, we obtain
$\varepsilon^{N}<\varepsilon^{N-1}\leq h^{2r}$ and
$\varepsilon^{-N+1}\geq h^{-2r}>h^{-r}$. In addition,
$\widetilde{\varepsilon}=h^r$ and $\widetilde{N}\geq3$ yield
$\widetilde{\varepsilon}^{\widetilde{N}}\leq
\widetilde{\varepsilon}^3=h^{3r}$ and
$\widetilde{\varepsilon}^{-\widetilde{N}+1}
\geq\widetilde{\varepsilon}^{-2}=h^{-2r}>h^{-r}$. These estimates
along with \eqref{theo7-prf-eq12} derive that the estimation
\begin{equation}\label{theo7-prf-eq14}
\begin{aligned}
\widetilde{a}(t)-a(t)=\mathcal{O}(h^r),
\qquad
\widetilde{a}'(t)-a'(t)=\mathcal{O}(h^r),
\end{aligned}
\end{equation}
holds for $t\leq h^{-r}$.

Moreover, the inequalities of $\varepsilon^{-N+1}\geq
h^{-2r}$ and $\widetilde{\varepsilon} ^{-\widetilde{N}+1}\geq
h^{-2r}$ mean that $t\leq h^{-r}$ leads to
$t^2\leq\min\{\varepsilon^{-N+1},\widetilde{\varepsilon}^{-\widetilde{N}+1}\}$.
Since
$\lim\limits_{\varepsilon\rightarrow0^+}\varepsilon|\log\varepsilon|^{\nu+1}=0$,
we have $\varepsilon|\log\varepsilon|^{\nu+1}<1$ and
$\widetilde{\varepsilon} |\log\widetilde{\varepsilon}|^{\nu+1}<1$
for sufficiently small $\varepsilon$ and $h$. It then follows from
the inequalities of $\varepsilon^{N-1}\leq h^{2r}$ and
$\widetilde{\varepsilon} ^{\widetilde{N}-1}\leq h^{2r}$, together
with \eqref{theo7-prf-eq13} that for $t\leq h^{-r}$ it holds
\begin{equation}\label{theo7-prf-eq15}
\begin{aligned}
\widetilde{\theta}(t)-\theta(t)=\mathcal{O}(th^r),
\qquad
\widetilde{\theta}'(t)-\theta'(t)=\mathcal{O}(th^r).
\end{aligned}
\end{equation}

On noting the real-analyticity of the  symplectic change of coordinates
$(p,x,q,y)\mapsto(a,a',\theta,\theta')$, it follows from \eqref{theo7-prf-eq14}
and \eqref{theo7-prf-eq15} that for $t\leq h^{-r}$ the difference between the solution
$\big(\widetilde{p}(t),\widetilde{x}(t),\widetilde{q}(t),\widetilde{y}(t)\big)$ of the
modified Hamiltonian $\widetilde{\Gamma}(p,x,q,y)$ and the solution
$\big(p(t),x(t),q(t),y(t)\big)$ of the extended Hamiltonian $\Gamma(p,x,q,y)$ satisfies
\begin{equation}\label{theo7-prf-eq16}
\begin{aligned}
\big\|\big(\widetilde{p}(t),\widetilde{x}(t),\widetilde{q}(t),\widetilde{y}(t)\big)
-\big(p(t),x(t),q(t),y(t)\big)\big\|\leq C'th^r,
\end{aligned}
\end{equation}
where $C'$ takes the maximum of $C$, $\widetilde{C}$, and the constants
symbolized by all the previously appearing $\mathcal{O}$-notations.

(b) Now, we reconsider the restriction relation $\|(p-x,q-y)\|\leq\varepsilon_0$.
Let $\delta(t)=\big\|\big(p(t)-x(t),q(t)-y(t)\big)\big\|$,
$\widetilde{\delta}(t)=\big\|\big(\widetilde{p}(t)-\widetilde{x}(t),\widetilde{q}(t)
-\widetilde{y}(t)\big)\big\|$, $\delta_0=\delta(0)$, and
$\widetilde{\delta}_0=\widetilde{\delta}(0)$. If $\delta_0=\widetilde{\delta}_0=0$,
then $\delta(t)=\widetilde{\delta}(t)=0$ holds for all $t\geq0$ such that
$\big(p(t),x(t),q(t),y(t)\big)\in \mathcal{B}_{\varepsilon_0}(\mathcal{N})$
and $\big(\widetilde{p}(t),\widetilde{x}(t), \widetilde{q}(t),\widetilde{y}(t)\big)
\in \mathcal{B}_{\varepsilon_0}(\mathcal{N})$ naturally hold.

If $\delta_0>0$, it derives from
\eqref{theo7-prf-eq2.2} and \eqref{theo7-prf-eq3} that
\begin{equation}\label{theo7-prf-eq17}
\begin{aligned}
&a(t)-a'(t)=\mathcal{O}(\delta_0+Ct\varepsilon^{N})
\qquad for\quad t\leq \varepsilon^{-N+1},
\\
&\theta(t)-\theta'(t)=\mathcal{O}\big(\delta_0+t\delta_0+C(t^2
+t|\log\varepsilon|^{\nu+1})\varepsilon^{N}\big)
\quad for\quad t^2\leq \varepsilon^{-N+1}.
\end{aligned}
\end{equation}
Since $N\geq\{\frac{2r}{\beta}+1,\frac{3r}{\beta},2\}$
and $\varepsilon\leq h^{\beta}$, we have
$\varepsilon^{N-1}\leq h^{2r}$, $\varepsilon^{N}\leq h^{3r}$,
and $\varepsilon^{-N+1}\geq h^{-2r}\geq h^{-r}$, which together
with \eqref{theo7-prf-eq17} lead to
\begin{equation}\label{theo7-prf-eq17.2}
\begin{aligned}
a(t)-a'(t)=\mathcal{O}(\delta_0+h^r),
\qquad
\theta(t)-\theta'(t)=\mathcal{O}(\delta_0+t\delta_0+h^r),
\end{aligned}
\end{equation}
for $t\leq h^{-r}$. A similar argument on \eqref{theo7-prf-eq10}
and \eqref{theo7-prf-eq11} also gives
\begin{equation}\label{theo7-prf-eq18}
\begin{aligned}
\widetilde{a}(t)-\widetilde{a}'(t)=\mathcal{O}(\delta_0+h^r), \qquad
\widetilde{\theta}(t)-\widetilde{\theta}'(t)=\mathcal{O}(\delta_0+t\delta_0+h^r),
\end{aligned}
\end{equation}
for $t\leq h^{-r}$ on noting that $a(0)-a'(0)$,
$\theta(0)-\theta'(0)$, $\widetilde{a}(0)-\widetilde{a}'(0)$, and
$\widetilde{\theta}(0)-\widetilde{\theta}'(0)$ are all in the
same magnitude of $\mathcal{O}(\delta_0)$. The estimates
\eqref{theo7-prf-eq17.2} and \eqref{theo7-prf-eq18} in turn derive
\begin{equation}\label{theo7-prf-eq19}
\delta(t)=\mathcal{O}(\delta_0+t\delta_0+h^r),
\qquad
\widetilde{\delta}(t)=\mathcal{O}(\delta_0+t\delta_0+h^r),
\end{equation}
for $t\leq h^{-r}$.

The restriction $\delta(t),\widetilde{\delta}(t)<\varepsilon_0$
along with $\varepsilon\leq\varepsilon_0$, $h^{r}\leq\varepsilon_0$, and
the estimation \eqref{theo7-prf-eq19} further requires
$t\leq\varepsilon_0/\delta_0$, which finally yields
\begin{equation}\label{theo7-prf-eq20}
\begin{aligned}
\big\|\big(\widetilde{p}(t),\widetilde{x}(t),\widetilde{q}(t),\widetilde{y}(t)\big)
-\big(p(t),x(t),q(t),y(t)\big)\big\|\leq C'th^r,
\end{aligned}
\end{equation}
for $t\leq\min\{h^{-r},\varepsilon_0/\delta_0\}$.

(c) The next step is to estimate the difference between the numerical
solution $(p_n,x_n,q_n,y_n)$ and the exact solution $\big(\widetilde{p}(t),
\widetilde{x}(t),\widetilde{q}(t),\widetilde{y}(t)\big)$ of the modified
Hamiltonian $\widetilde{\Gamma}(p,x,q,y)$. Under the symplectic change
of coordinates $(p,x,q,y)\mapsto(b,b',\varphi,\varphi')$, the numerical
solution $(p_n,x_n,q_n,y_n)$ is denoted by $(b_{n},b'_{n},\varphi_{n},
\varphi'_{n})$ in the coordinate $(b,b',\varphi,\varphi')$. Then,
the solution of \eqref{theo7-prf-eq7} advancing one step at the time $t_n$
with the initial values $(b_{n},b'_{n},\varphi_{n},\varphi'_{n})$ satisfies
\begin{equation}\label{theo7-prf-eq21}
\begin{aligned}
&\|{b}(t_n+h)-b_n\|\leq
\widetilde{C}h\widetilde{\varepsilon}^{\widetilde{N}},
\qquad
\|{b}'(t_n+h)-b'_n\|\leq
\widetilde{C}h\widetilde{\varepsilon}^{\widetilde{N}},
\\
&\|{\varphi}(t_n+h)-\omega_{h,\varepsilon,N}(b_n,b'_n)h-
\varphi_n\|\leq \widetilde{C}(h^2 +h|\log\widetilde{\varepsilon}|^{\nu+1})
\widetilde{\varepsilon}^{\widetilde{N}},
\\
&\|{\varphi}'(t_n+h)-\omega_{h,\varepsilon,N}(b'_n,b_n)h-
\varphi'_n\|\leq \widetilde{C}(h^2 +h|\log\widetilde{\varepsilon}|^{\nu+1})
\widetilde{\varepsilon}^{\widetilde{N}}.
\end{aligned}
\end{equation}

Because $\widetilde{\Gamma}(p,x,q,y)$ is an $s$th-order truncation
of the completely modified Hamiltonian whose solution advancing one
step starting with $(p_n,x_n,q_n,y_n)$ at the time $t_n$ is
$(p_{n+1},x_{n+1},q_{n+1},y_{n+1})$, it follows that
\begin{equation}\label{theo7-prf-eq22}
\begin{aligned}
&{b}(t_n+h) = b_{n+1} + \mathcal{O}(h^{s+1}),
&&{b}'(t_n+h) = b'_{n+1} + \mathcal{O}(h^{s+1}),
\\
&{\varphi}(t_n+h) = \varphi_{n+1} + \mathcal{O}(h^{s+1}),
&&{\varphi}'(t_n+h) = \varphi'_{n+1} + \mathcal{O}(h^{s+1}).
\end{aligned}
\end{equation}
By setting $s=2r$ and noting $\widetilde{\varepsilon}=h^{r}$ and $\widetilde{N}\geq3$,
the comparison between \eqref{theo7-prf-eq21} and \eqref{theo7-prf-eq22} gives
\begin{equation*}
\begin{aligned}
&b_{n+1}=b_n + \mathcal{O}(h^{2r+1}),
&&\varphi_{n+1} = \omega_{h,\varepsilon,N}(b_n,b'_n)h + \varphi_{n}
+ \mathcal{O}(h^{2r+1}),
\\
&b'_{n+1}=b'_n + \mathcal{O}(h^{2r+1}),
&&\varphi'_{n+1} = \omega_{h,\varepsilon,N}(b'_n,b_n)h + \varphi'_{n}
+ \mathcal{O}(h^{2r+1}),
\end{aligned}
\end{equation*}
which further yields
\begin{equation}\label{theo7-prf-eq23}
\begin{aligned}
&b_{n}=b_0 + \mathcal{O}(th^{2r}),
\qquad b'_{n}=b'_0 + \mathcal{O}(th^{2r}),
\\
&\varphi_{n} = \omega_{h,\varepsilon,N}(b_0,b'_0)h + \varphi_{0}
+ \mathcal{O}(th^{2r})+ \mathcal{O}(t^2h^{2r}),
\\
&\varphi'_{n} = \omega_{h,\varepsilon,N}(b'_0,b_0)h + \varphi'_{0}
+ \mathcal{O}(th^{2r})+ \mathcal{O}(t^2h^{2r}),
\end{aligned}
\end{equation}
for $t=nh$. On noting $\widetilde{\varepsilon}^{\widetilde{N}}\leq h^{3r}$
and $(b_0,b'_0,\varphi_0,\varphi'_0)=\big(\widetilde{b}(0),\widetilde{b}'(0),
\widetilde{\varphi}(0),\widetilde{\varphi}'(0)\big)$,
the comparison between \eqref{theo7-prf-eq8}-\eqref{theo7-prf-eq9}
and \eqref{theo7-prf-eq23} gives
\begin{equation}\label{theo7-prf-eq24}
\begin{aligned}
&b_{n}=\widetilde{b}(t) + \mathcal{O}(th^{2r}),
&&b'_{n}=\widetilde{b}'(t) + \mathcal{O}(th^{2r}),
\\
&\varphi_{n} = \widetilde{\varphi}(t)+\mathcal{O}(t^2h^{2r}),
&&\varphi'_{n} = \widetilde{\varphi}'(t)+\mathcal{O}(t^2h^{2r}).
\end{aligned}
\end{equation}
If $t\leq\min\{h^{-r},\varepsilon_0/\delta_0\}$, it is true that
\begin{equation}\label{theo7-prf-eq26}
\begin{aligned}
&b_{n}=\widetilde{b}(t) + \mathcal{O}(h^{r}),
&&b'_{n}=\widetilde{b}'(t) + \mathcal{O}(h^{r}),
\\
&\varphi_{n} = \widetilde{\varphi}(t)+\mathcal{O}(th^{r}),
&&\varphi'_{n} = \widetilde{\varphi}'(t)+\mathcal{O}(th^{r}),
\end{aligned}
\end{equation}
which yields the estimation
\begin{equation}\label{theo7-prf-eq25}
\big\|(p_n,x_n,q_n,y_n)-\big(\widetilde{p}(t),\widetilde{x}(t),
\widetilde{q}(t),\widetilde{y}(t)\big)\big\|\leq \widetilde{C}_0th^{r},
\end{equation}
for $t\leq\min\{h^{-r},\varepsilon_0/\delta_0\}$,
where the constant $\widetilde{C}_0$ is independent of the stepsize $h$.

(d) The previous estimation in \eqref{theo7-prf-eq25} is estimated not under the condition
$(p_n,x_n,q_n,y_n)\in\mathcal{B}_{\varepsilon_0}(\mathcal{N})$. It is noted that
any numerical solution $(p_n,x_n,q_n,y_n)$ without this restriction
cannot be guaranteed the applicability of the symplectic change
of coordinates $(p,x,q,y)\mapsto(b,b',\varphi,\varphi')$. Hence, the
restriction $(p_n,x_n,q_n,y_n)\in\mathcal{B}_{\varepsilon_0}(\mathcal{N})$
for the numerical solution $(p_n,x_n,q_n,y_n)$ should be added.
The triangle inequality together with \eqref{theo7-prf-eq20} and
\eqref{theo7-prf-eq25} finally gives the result \eqref{theorem7-subequ1}
for $t=nh$ with such $n$ that  $t=nh\leq\min\{h^{-r},\varepsilon_0/\delta_0\}$
and $(p_n,x_n,q_n,y_n)\in\mathcal{B}_{\varepsilon_0}(\mathcal{N})$.

(e) Since the Hamiltonian $\Gamma(p,x,q,y) $ possesses $\tilde{d}$
first integrals $I_j^*(p,x,q,y)~(j=1,\ldots,\tilde{d})$,  the
symplectic change of coordinates
$(p,x,q,y)\mapsto(a,a',\theta,\theta')$ could be determined such
that
$$\widetilde{I}_j^*\big(a(p,x,q,y),a'(p,x,q,y)\big)=
I_j^*(p,x,q,y)=I_j^*(p_0,x_0,q_0,y_0)$$ are constants for
$j=1,\ldots,\tilde{d}$ during the evolution of the extended
Hamiltonian $\Gamma(p,x,q,y) $. Due to the
$\mathcal{O}(h^r)$-closeness to the identity of the symplectic
change of coordinates
$(a,a',\theta,\theta')\mapsto(b,b',\varphi,\varphi')$, it follows
from \eqref{theo7-prf-eq26} that
\begin{equation}\label{theo7-prf-eq27}
a_{n}=\widetilde{a}(t) + \mathcal{O}(h^{r}),
\qquad
a'_{n}=\widetilde{a}'(t) + \mathcal{O}(h^{r}).
\end{equation}
The triangle inequality together with \eqref{theo7-prf-eq27}  and
\eqref{theo7-prf-eq14} gives
\begin{equation}\label{theo7-prf-eq28}
\|a_n - a(t)\|\leq C_0h^r,
\qquad
\|a'_n - a'(t)\|\leq C_0h^r,
\end{equation}
which implies
\begin{equation*}
\widetilde{I}_j^*\big(a_n,a'_n\big)
-\widetilde{I}_j^*\big(a(t),a'(t)\big)=\mathcal{O}(h^r),
\end{equation*}
for $j=1,\ldots,\tilde{d}$. This finally gives the result
\eqref{theorem7-subequ2}.
\end{proof}
\begin{remark}\label{remark4}
It is worth emphasizing that the requirement
$(p_n,x_n,q_n,y_n)\in\mathcal{B}_{\varepsilon_0}$ for $n$ in this
theorem should be considered in a precise sense that except for the
numerical solution $(p_n,x_n,q_n,y_n)$, all other intermediate
variables such as
$\tilde{p}_{n},\tilde{x}_{n},\tilde{q}_{n},\tilde{y}_{n}$ in
Definition~\ref{definition2} or Definition~\ref{definition3} have to
satisfy the condition
$(\tilde{p}_{n},\tilde{x}_{n},\tilde{q}_{n},\tilde{y}_{n}) \in
{B}_{\varepsilon_0}(\mathcal{N})$. It is also pointed
out that the effective estimated time interval is dependent not
only on the constants $\varepsilon_0$ and $\delta_0$, but also on
the numerical discrepancy $\|(p_n-x_n,q_n-y_n)\|$.  The optimal case
occurs if $\delta_0=0$ and $\|(p_n-x_n,q_n-y_n)\|$ holds for all
$n\in\mathbb{N}$, where the effective estimated time interval has
the same estimation as classical symplectic integrators applied to
an  integrable (or a near-integrable) Hamiltonian system.
\end{remark}
The results in Theorem~\ref{theorem7} also hold for the integrable
Hamiltonian $H(p,q)$ as $\varepsilon=0$, where the parameter
$\varepsilon_0$ is still needed to bound $\|(p-x,q-y)\|$ so that the
extended Hamiltonian $\Gamma(p,x,q,y)$ is near-integrable. This
theorem states the linear error growth and near-preservation of
first integrals of symplectic extended phase space methods
when applied to the original (near-) integrable
Hamiltonian system $H(p,q)$. The behavior is similar to that of
classical symplectic methods derived in the original phase space but
with different estimated time intervals. With
Theorem~\ref{theorem7}, we immediately derive the error estimation
on the single-factor or double-factor explicit symplectic integrator
defined by Definition~\ref{definition4} or
Definition~\ref{definition5}.

\begin{theorem}\label{theorem8}
Consider an $r$th-order single-factor explicit symplectic integrator
determined by Definition~\ref{definition4} or
double-factor explicit symplectic integrator determined by
Definition~\ref{definition5} when applied to the original (near-)
integrable Hamiltonian $H(p,q)$ which has $d'~(\leq d)$ first
integrals $\widetilde{I}_j(p,q)~(j=1,\ldots,d')$. Suppose that the
parameters satisfy the conditions in Theorem~\ref{theorem7}, then
\begin{subequations}
\begin{eqnarray}\label{theorem8-equ-parent}
&&\big\|(p_n,q_n)-\big(p(t),q(t)\big)\big\|\leq C_0th^r,
\label{theorem8-subequ1}
\\
&&\big|\widetilde{I}_j(p_n,q_n)-\widetilde{I}_j(p_0,q_0)\big|\leq
C_0h^r,\quad j=1,\ldots,d', \label{theorem8-subequ2}
\end{eqnarray}
\end{subequations}
hold for $t=nh\leq h^{-r}$.
\end{theorem}
\begin{proof}
Let $T_{n}$ denote the translation mapping that projects
$(\tilde{p}_{n+1},\tilde{x}_{n+1},\tilde{q}_{n+1},\tilde{y}_{n+1})$
to $(p_{n+1},x_{n+1},q_{n+1},y_{n+1})\in \mathcal{N}$. Usually,
$T_n$ varies with different $n$. Let
$\Psi_{h}^{(n)}=T_n\circ\Phi_h$, where $\Phi_h$ is just the
$r$th-order explicit extended phase space method. Then,
$\Psi_{h}^{(n)}$ could be viewed as an $r$th-order single-factor or
double-factor explicit symplectic integrator for the extended
Hamiltonian $\Gamma(p,x,q,y)$. For the imposed initial conditions
$(p_0,q_0)=(x_0,y_0)$, the numerical solution $(p_n,x_n,q_n,y_n)$ of
$\Psi_{h}^{(n)}$ for the extended Hamiltonian $\Gamma(p,x,q,y)$
satisfies $(p_n,q_n)=(x_n,y_n)$, i.e., $(p_n,x_n,q_n,y_n)\in
\mathcal{N}$ for all $n\in\mathbb{N}$. It follows from
Lemma~\ref{lemma1} that if the initial conditions satisfy
$(p_0,q_0)=(x_0,y_0)$, there exists a globally well-defined modified
Hamiltonian $\widetilde{\Gamma}(p,x,q,y)$ for the symplectic
integrator $\Psi_{h}^{(n)}$, which is just two identical copies of
the modified Hamiltonian $\widetilde{H}(p,q)$ obtained in
Lemma~\ref{lemma1}, i.e.,
$\widetilde{\Gamma}(p,x,q,y)=\widetilde{H}(p,q)+\widetilde{H}(x,y)$.

In addition, $(\tilde{p}_{n+1},\tilde{x}_{n+1},\tilde{q}_{n+1},
\tilde{y}_{n+1})= \Phi_h(p_n,p_n,q_n,q_n)$ means that the intermediate variables
$\tilde{p}_{n+1},\tilde{x}_{n+1},\tilde{q}_{n+1},\tilde{y}_{n+1}$ satisfy
$\tilde{p}_{n+1}-\tilde{x}_{n+1}=\mathcal{O}(h^{r+1})$ and
$\tilde{q}_{n+1}-\tilde{y}_{n+1}=\mathcal{O}(h^{r+1})$ for all $n\in\mathbb{N}$,
where the constant symbolized by the $\mathcal{O}$-notation is independent of
$n$ and $h$. Due to the restriction on the stepsize that $h\leq\{h_0,\varepsilon_0^{1/r}\}$,
it consequently follows that $h^{r+1}<\varepsilon_0$ and thus
$\tilde{\delta}_n=\|(\tilde{p}_n-\tilde{x}_n,\tilde{q}_n-
\tilde{y}_n)\|<\varepsilon_0$ holds for all $n\in\mathbb{N}$.
That is, the time interval of the
linear error growth of the symplectic integrator $\Psi_h^{(n)}$ starting with $(p_0,q_0)=(x_0,y_0)$ will be
$t=nh\leq h^{-r}$. Then, the result \eqref{theorem8-subequ1} immediately follows
from $(p_n,q_n)=(x_n,y_n)$, $\big(p(t),q(t))=(x(t),y(t)\big)$, and the estimation
\eqref{theorem7-subequ1}.

Selecting a suitable  symplectic change of coordinates
$(p,q)\mapsto(a,\theta)$ such that
$a^{j}(p,q)=\widetilde{I}_j(p,q)=\widetilde{I}_j(p_0,q_0)$ for
$j=1,\ldots,d'$, we obtain the result
\eqref{theorem8-subequ2} from \eqref{theo7-prf-eq28}. This completes
the proof.
\end{proof}
Theorem~\ref{theorem8} shows that both the single-factor explicit
symplectic integrator defined by Definition~\ref{definition4} and
the double-factor explicit symplectic integrator defined by
Definition~\ref{definition5} absolutely give the linear growth of
global errors and the uniformly near preservation of first integrals
when applied to (near-) integrable Hamiltonian systems. We here give
an interpretation for the numerical performance of the existing
extended phase space methods: (i) the symmetric projection
symplectic integrator, i.e., the semiexplicit symplectic methods
proposed by Jayawardana \& Ohsawa \cite{Jayawardana2023};
(ii) Pihajoki's original extended phase space methods
without mixing \cite{Pihajoki2015}; (iii) Tao's extended phase space
methods \cite{Tao2016b}.

For the semiexplicit symplectic methods in \cite{Jayawardana2023}, because
$\delta_n=0$ holds for all $n\in \mathbb{N}$,  the same linear error growth and near-preservation
of first integrals as in Theorem~\ref{theorem8} are soundly expected provided
that the implicit iteration during the implementation of the method is convergent.

For Pihajoki's extended phase space methods \cite{Pihajoki2015}
such as \eqref{ext-leapfrog} where no additional
operation such as the mixing is conducted on the numerical solution
$(p_n,x_n,q_n,y_n)$, the linear error growth stated by
\eqref{theorem7-subequ1} may give the worst estimation
$\delta_n=\|(p_n-x_n,q_n-y_n)\|=\mathcal{O}(th^{r})$ as
$\big(p(t),q(t)\big)=\big(x(t),y(t)\big)$. Then,  the restriction
$(p_n,x_n,q_n,y_n) \in\mathcal{B}_{\varepsilon_0}$ will reduce the
time interval of the linear error growth by a small factor of
$\varepsilon_0<1$, i.e., from $h^{-r}$ to $\varepsilon_0h^{-r}$. On
noting the small magnitude of the perturbation parameter
$\varepsilon_0$, this point provides a reasonable interpretation for
the short-term reliability of Pihajoki's extended phase space
methods.

As to Tao's methods, the extended Hamiltonian becomes
\begin{equation}\label{Tao-extTwo}
\overline{\Gamma}(p,x,q,y)={\Gamma}(p,x,q,y)+ H_C(p,x,q,y),
\end{equation}
where $H_C(p,x,q,y)=\frac{\omega}{2}\big(\|p-x\|^2 +\|q-y\|^2 \big)$
for some positive $\omega\in\mathbb{R}^+$. Due to the boundedness of
both the Hamiltonian $\overline{\Gamma}(p,x,q,y)$ and the modified
Hamiltonian $\widehat{\Gamma}(p,x,q,y)$ of the symplectic method
applied to $\overline{\Gamma}(p,x,q,y)$, the discrepancies $\|p-x\|$
and $\|q-y\|$ of both the exact solution
$\big(p(t),x(t),q(t),y(t)\big)$ and the numerical solution
$(p_n,x_n,q_n,y_n)$ are bounded by $\mathcal{O}(1/\sqrt{\omega})$
for $t\leq\sqrt{\omega}$. Hence, as analyzed in \cite{Tao2016b},
there exists $\omega_0>0$ such that for any $\omega\geq\omega_0$ the
linear error growth of an $r$th-order Tao's symplectic method when
applied to $\overline{\Gamma}(p,x,q,y)$ holds for
$t\leq\min\{\omega^{-1}h^{-r},\omega^{1/2}\}$ for the rare
integrable case of $\overline{\Gamma}(p,x,q,y)$ (More precisely, the
integrability should be attributed to $H(p,q)$ rather than
$\overline{\Gamma}(p,x,q,y)$, since a large magnitude of
$\delta_0=\mathcal{O}(1/\sqrt{\omega})$ will result in the
nonintegrability of $\overline{\Gamma}(p,x,q,y)$ even for a large
$\omega$ once $H(p,q)$ is nonintegrable. That is, the estimation
made  in \cite{Tao2016b} should be incorporated with a small
$\delta_0$. In fact, the complete integrability of
$\overline{\Gamma}(p,x,q,y)$ will lead to the classical estimation
on the linear error growth that the effective time interval for $t$
does not depend on $\omega$ even if there is no the restriction
$\delta_0=0$).

Here, we give another estimation by applying Theorem~\ref{theorem7}
to Tao's methods \cite{Tao2016b}. A multiscale analysis (see Appendix of  \cite{Tao2016b}) shows that
for sufficiently large $\omega\geq\omega_0$, the discrepancy $\widetilde{\delta}(t)
=\big\|\big(\widetilde{p}(t)-\widetilde{x}(t),\widetilde{q}(t)
-\widetilde{y}(t)\big)\big\|$ of the modified Hamiltonian
$\widehat{\Gamma}(p,x,q,y)$ starting with $\widetilde{\delta}_0={\delta}_0=0$
satisfies
\begin{equation}\label{est-tilde-delta}
\widetilde{\delta}(t)=\mathcal{O}\big(\omega^{-3/2}tC_3
\exp\big((\omega^{-1/2}C_1+C_2)t\big)\big).
\end{equation}

Similarly to \eqref{nearInt-extH}, the extended
Hamiltonian $\overline{\Gamma}(p,x,q,y)$ could be expressed as
\begin{equation*}
\begin{aligned}
\overline{\Gamma}(p,x,q,y)= &H_0(p,q) + H_0(x,y)
+ \varepsilon \big( H_1(p,q)+H_1(x,y)+G(p,x,q,y)\big)
\\
&+H_C(p,x,q,y).
%\frac{\omega}{2}\big(\|p-x\|^2 +\|q-y\|^2 \big).
\end{aligned}
\end{equation*}
To express $\overline{\Gamma}(p,x,q,y)$ as a near-integrable
Hamiltonian system with a small perturbed parameter $\varepsilon$,
it requires $H_C\leq \varepsilon_0$. In addition, to apply
Theorem~\ref{theorem7}, we also need ${\delta}(t)\leq\varepsilon_0$.
The two inequalities become
$\omega\widetilde{\delta}^2(t)\leq2\varepsilon_0$ and
$\widetilde{\delta}(t)\leq\varepsilon_0$ when applied to the
modified Hamiltonian $\widehat{\Gamma}(p,x,q,y)$ and lead to the
estimation
\begin{equation}\label{est-tao}
\big\|(p_n,q_n)-\big(p(t),q(t)\big)\big\|=\mathcal{O}(\omega th^{r})
\end{equation}
for
$t\leq\min\{\omega^{-1}h^{-r},\omega^{1/2},\sqrt{2}\omega\varepsilon_0^{1/2},
\omega^{3/2}\varepsilon_0\}$ by applying Theorem~\ref{theorem7} and
noting $\widetilde{\varepsilon}=\omega h^{r}$. For the
case of $1\leq\omega<\omega_0$, since the estimation
\eqref{est-tilde-delta} is no longer satisfied, the general
estimation is $\widetilde{\delta}(t)=\mathcal{O}(\omega t h^{r})$.
It then follows from
$\omega\widetilde{\delta}^2(t)\leq2\varepsilon_0$ and
$\widetilde{\delta}(t)\leq\varepsilon_0$ that the estimation
\eqref{est-tao} holds only for  $t\leq\min\{\varepsilon_0\omega^{-1}
h^{-r},\sqrt{2}\varepsilon_0^{1/2}\omega^{-3/2}h^{-r}\}$. For the
last case where $0\leq\omega<1$, the estimation \eqref{est-tao}
holds for $t\leq \varepsilon_0h^{-r}$ as $\varepsilon_0<1$, which is
the same as Pihajoki's methods.

One may argue that the effective estimated time interval of the
linear error growth of Tao's symplectic methods could be extended
just by enlarging the parameter $\omega$ and reducing the stepsize
$h$. However, as previously analyzed, the  appearance of $\omega$ in
the global error means that enlarging $\omega$ also results in an
enlarged global error with an unchanged $h$. Another inconvenience
of Tao's method is that the stepsize should be bounded by $h\leq
\omega^{-1/r}$.

Overall, to our best knowledge and as analyzed above,
we may claim that if no other operation such as projection onto the
submanifold $\mathcal{N}$ is conducted on the numerical solutions
of Pihajoki's or Tao's symplectic extended phase space methods, the
linear error growth only holds for a short time, which will be much
shorter than that of the single-factor explicit symplectic
integrator defined by Definition~\ref{definition4},  the
double-factor explicit symplectic integrator defined by
Definition~\ref{definition5}, and  the semiexplicit methods with
symmetric projection proposed in \cite{Jayawardana2023}.

We finally point out that the near-preservation of first integrals
established in Theorem~\ref{theorem7} and Theorem~\ref{theorem8}
only holds for at most polynomially lone time interval as $t\leq
h^{-r}$. However, the single consideration of the Hamiltonian energy
gives rise to the following estimation \cite[Theorem
IX.8.1]{Hairer2006}
\begin{equation*}
H(p_n,q_n)-H(p_0,q_0)=\mathcal{O}(h^r)
\end{equation*}
for exponentially long time interval as $t\leq e^{h_0/2h}$ for an
$r$th-order symplectic method when applied to the
Hamiltonian $H(p,q)$ regardless of its integrability, where the
analyticity of $H(p,q)$ and the existence of the globally
well-defined modified Hamiltonian are the main imposition. That is,
the near-preservation of energy of symplectic integrators may hold
for a much longer time interval than the linear error growth, no
matter whether the Hamiltonian is near-integrable or not.

\subsection{Nonintegrable Hamiltonian systems}\label{sec:nonint}
The motion in the near-integrable Hamiltonian system $H(p,q)$ with
the initial values satisfying the diophantine condition is regular,
i.e., periodic or quasi-periodic. That is, the canonical conjugate
variable $(p,q)$ could be transformed by a symplectic transformation
into action-angle variables $(a,\theta)$ at least in a finite time
interval. However, if the perturbation parameter $\varepsilon$ is
large enough to go beyond the threshold value $\varepsilon_0$, or
the initial values do not satisfy the diophantine condition, the
motion in such a Hamiltonian system may be chaotic such that the
backward error analysis based on the modified Hamiltonian for
symplectic methods is no longer valid.

To present a clear estimation of the error growth of symplectic
integrators when applied to nonintegrable Hamiltonian
systems, we look through the classifications of trajectories, which
are labeled as regular or chaotic.  As investigated by Henon \&
Heiles \cite{Henon1964}, the picture of nonintegrable Hamiltonian
systems is mostly like:
\begin{enumerate}
  \item there is an infinite number of chains of islands, where each chain of islands
  corresponds to a stable periodic or quasi-periodic trajectory;
  \item the set of all the islands is dense everywhere;
  \item there exists an ergodic sea on which the ergodic trajectory is dense.
\end{enumerate}
Moreover, it follows from \emph{Carath\'{e}odory--Jacobi--Lie}
theorem \cite{Lee2013} that under mild regularity assumptions, for a
fixed point $z=(p,q)\in T^*\mathbb{R}^d$ (apart from singular points
of the Hamiltonian), there exist locally $d$ first integrals for the
Hamiltonian $H(p,q)$ in a small neighborhood of $z$. That is, the
Hamiltonian $H(p,q)$ is locally integrable \cite{Kozlov1983}. As
displayed by the picture of  Henon \& Heiles \cite{Henon1964}, an
intuitive but unproved deduction naturally follows that the
existence of the $d$ local isolating first integrals (or
quasi-integrals) of the point in the chain of islands could be
extended to the neighborhood of $\phi_H^t(z)$ for any $t$ that
included in the chain.

For the ergodic sea, although there exist regular trajectories, one
cannot distinguish a regular trajectory from chaotic trajectories in
the sea with the up-to-date techniques such as the Lyapunov
characteristic exponent or the power spectra, because of the
denseness of ergodic (chaotic) trajectories. In what
follows, we admit the following hypothesis to make the error
analysis based on the investigation in the literature.
\begin{hypothesis}\label{hypothesis-1}
There exist some connected open sets (regular regions) in the phase
space such that the nonintegrable Hamiltonian in each set is locally
integrable with $d$ first integrals (or quasi-integrals). The
connected set is invariant under the phase flow of the Hamiltonian
and is like a fiber pipe such that the evolution of the neighborhood
of a point (a set of initial conditions) is still included in the
connected open set. The set containing all points outside the
regular regions is the ergodic region, where the trajectory starting
with the point in this region is most likely chaotic.
\end{hypothesis}
It is yielded from this hypothesis that the distance
of two neighboring points in the same regular region increases about
linearly with time, while it grows roughly exponentially for two
neighboring points in the ergodic region. This point is also well
consistent with the numerical investigation in \cite{Henon1964}.
\subsubsection{Trajectories in regular region}\label{sec:regular}
In this subsection, we consider the trajectory lying in the regular
region. Suppose that the initial value $z_0=(p_0,q_0)$ of the
nonintegrable Hamiltonian $H(p,q)$ is in a connected regular open
region $D_0$. According to Hypothesis~\ref{hypothesis-1}, we define
the minimum regular radius $r(z_0,D_0)$ of a point
$z_0$ in the regular region $D_0$ as follows:
\begin{equation*}
r_0(z_0,D_0) :=\inf_{\substack{\forall\,t\in \mathbb{R} \\ z=\phi_H^{t}(z_0)}}
\sup\Big\{ r\in\mathbb{R}^{+}\,\big|\,
\mathcal{B}_{r}(z)\subset D_0\Big\},
\end{equation*}
where $\phi_H^{t}$ is the phase flow of the Hamiltonian $H(p,q)$ and
$\mathcal{B}_{r}(z)$ is the open ball in $T^*\mathbb{R}^d$ with
radius $r$ and center $z$. According to \emph{Liouville's theorem},
symplectic transformation preserves volume in the phase space for
Hamiltonian systems, we have $r_0(z_0,D_0)>0$ if there exists an
open ball $\mathcal{B}_{r}(z_0)$ with $r>0$ such that
$\mathcal{B}_{r}(z_0)\subset D_0$.

Suppose that the initial values $(p_0,q_0)$ satisfy the diophantine
condition \eqref{diophantine} and the stepsize $h$ is sufficiently
small such that
$h\leq\min\{h_0,\varepsilon_0^{1/r},r_0^{1/(r+1)}\}$. It
then follows from Theorem~\ref{theorem8} that the $r$th-order
single-factor or double-factor explicit symplectic integrator
 has the following  conservative estimation
\begin{equation*}
\big\|(p_n,q_n)-\big(p(t),q(t)\big)\big\|=\mathcal{O}(th^{r})
\end{equation*}
for at least $t\leq\min\{h^{-r},r_0h^{-r}\}$ in the worst case,
where the global error vector $\big(p_n-p(t),q_n-q(t)\big)$ is
always perpendicular to the instantaneous vector filed
$J_{2d}^{-1}\nabla H$ and thus the numerical solution $(p_n,q_n)$
reaches outside the regular region in the fastest path.

It is known from the previous estimation that the effective
estimated time interval of linear error growth largely depends on
the minimum regular radius $r_0(z_0,D_0)$. Usually, a smaller
$r_0(z_0,D_0)$ indicates a shorter effective estimated time
interval. Since $r_0(z_0,D_0)$ depends on both $z_0$ and $D_0$, the
effective estimated time interval varies with different regular
regions and with different initial points in the regular region. For
initial points in a given regular region, the closer to the boundary
of the regular region, the shorter effective estimated time interval
will be. Due to the existence of stable periodic orbit in the middle
of a regular chain of islands \cite{Henon1964}, we consider that
such a stable periodic orbit corresponds to the maximum effective
estimated time interval in this given regular region. In addition,
due to the infinite number of chains of islands, i.e., the infinite
number of regular regions, there may not exist a uniform positive
lower bound for the maximum effective estimated time interval.  This
point leads to an issue for the application of the proposed
symplectic integrators to nonintegrable Hamiltonian systems, that
is, the linear error growth of the symplectic integrator starting
with the point in one regular region holds for a certain time
interval, while the time interval may be sharply decreased once
starting with the point in another regular region even for the same
Hamiltonian.
\subsubsection{Trajectories in ergodic region} \label{sec:ergodic}
As we have previously
analyzed, one cannot distinguish a regular trajectory from chaotic
trajectories in the ergodic region. In this sense, we consider that
the perturbation $(\delta p,\delta q)$ to the trajectory starting with
the initial values $(p_0,q_0)$ located in the ergodic region is exponentially
sensitive to the sufficiently small initial perturbation $(\delta p_0,\delta q_0)$
and thus chaotic.

Suppose that the trajectory $\big(p(t),q(t)\big)$ starting with $(p_0,q_0)$ is chaotic,
that is, it has a positive Lyapunov characteristic exponent
\cite{Benettin1986,Benettin1994,Mei2018} defined as follows
\begin{equation*}
\sigma=
\lim_{\substack{t\rightarrow+\infty \\ d_{0}\rightarrow 0^+}}
\frac{1}{t}\ln\frac{d(t)}{d_0},
\end{equation*}
where $d(t)=\big\|\big(p(t),q(t)\big)-\big(p'(t),q'(t)\big)\big\|$, $\big(p'(t),q'(t)\big)$ is the
neighboring trajectory of $\big(p(t),q(t)\big)$ starting with initial conditions
$(p_0+\delta p_0,q_0+\delta q_0)$, and $d_0=d(0)=\|(\delta p_0,\delta q_0)\|$
is the magnitude of the initial perturbation.

It is a common view that the exponential divergence of nearby
trajectories causes numerical errors to be exponentially magnified
with time \cite{Hayes2005,Hayes2007}. Although the forward error
analysis for an $r$th-order symplectic method can yield the
exponential growth as follows
\begin{equation*}
\big\|(p_n,q_n)-\big(p(t),q(t)\big)\big\|=\mathcal{O}(Ce^{tL}h^{r})
\end{equation*}
with some positive constants $C$ and $L$, however, this estimation
is valid for general numerical methods and not limited to integrable
or nonintegrable Hamiltonian systems. That is, the exponential
growth of symplectic methods applied to chaotic trajectories of
nonintegrable Hamiltonian systems is indefinite because
the exponential estimation is due to a technical scaling of
inequalities.

To show the exponential error growth of symplectic integrators
caused by the exponential sensitivity of chaotic trajectories to the
initial perturbation, the backward error analysis
\cite{Hairer1994,Hairer1997,Reich1999} based on the modified
Hamiltonian may be invalid, since it is unclear about the dynamic
behavior of the trajectory
$\big(\widetilde{p}(t),\widetilde{q}(t)\big)$ in the modified
Hamiltonian starting with the same initial values $(p_0,q_0)$.
Hence, it is impossible to estimate the discrepancy
$\big({p}(t),{q}(t)\big)-\big(\widetilde{p}(t),\widetilde{q}(t)\big)$
without the information of
$\big(\widetilde{p}(t),\widetilde{q}(t)\big)$.

We here adopt the shadowing theory
\cite{Chow1992,Coomes1995,Coomes1997,Hayes2005,Hayes2007}, i.e.,
another backward error analysis method that modifies the initial
values while remains the Hamiltonian $H(p,q)$ unchanged, to estimate
the global error of a symplectic integrator. We first note that
although the transversal shadowing method has a longer shadowing
time than the map shadowing method \cite{Coomes1995}, the former
rescales the time that does not match the discrete times
$\{t_n=nh\}_{n=1}^{N_T}$  while the latter does. As our attention is
focused on the difference between the numerical solution and the
exact solution at time $t_n$, not the global closeness of the
numerical solution to the exact solution that may start with
different initial conditions, we then adopt the map
shadowing method here.

According to the shadowing theory \cite{Coomes1995}, the numerical
solution $(p_n,q_n)$ of an $r$th-order symplectic integrator is
defined as an $\widehat{\varepsilon}$-pseudo-trajectory of $H(p,q)$:
\begin{equation*}
\|(p_{n+1},q_{n+1})-\phi_H^{h}(p_n,q_n)\|\leq \widehat{\varepsilon},
\end{equation*}
where $\phi_H^{h}$ is the phase flow of $H(p,q)$ and
$\widehat{\varepsilon}=\mathcal{O}(h^{r+1})$. The result in
\cite{Coomes1995} shows that under technical conditions there exist
an exact trajectory $\big(\widetilde{p}(t),\widetilde{q}(t)\big)$ in
$H(p,q)$ starting with $(\widetilde{p}_0,\widetilde{q}_0)$ and a
positive number $\hat\delta$ such that
\begin{equation}\label{shadow-est}
\|(p_n,q_n)-(\widetilde{p}_n,\widetilde{q}_n)\|\leq \hat\delta,
\end{equation}
where $(\widetilde{p}_n,\widetilde{q}_n)=\big(\widetilde{p}(t_n),
\widetilde{q}(t_n)\big)$, $t_n = nh$ for
$n=0,1,\ldots,N_T$, and $\hat\delta$ is the shadowing
distance.

Suppose that $(p_0,q_0)$ is in the ergodic region and corresponds
 to a chaotic trajectory $\big(p(t),q(t)\big)$, then the trajectory
corresponding to $(p_0,p_0,q_0,q_0)$ in the extended Hamiltonian
$\Gamma(p,x,q,y)$ or $\overline{\Gamma}(p,x,q,y)$ must be located in
the ergodic region and exponentially sensitive to initial perturbation
as well. As we have just analyzed, the application of the
single-factor or double-factor explicit symplectic integrator leads
to the estimation \eqref{shadow-est}. Since both the exact
trajectory $\big(p(t),q(t)\big)$ and the shadowing trajectory
$\big(\widetilde{p}(t),\widetilde{q}(t)\big)$ obey the same
differential equations, i.e., the canonical equations of the
Hamiltonian $H(p,q)$ but with different initial values $(p_0,q_0)$
and $(\widetilde{p}_0,\widetilde{q}_0)$, it follows from the
chaoticity of the trajectory $(p(t),q(t))$ and $\|(p_0,q_0)
-(\widetilde{p}_0,\widetilde{q}_0)\|\leq\hat\delta$ that the
discrepancy
$\big(\widetilde{p}(t),\widetilde{q}(t)\big)-\big(p(t),q(t)\big)$
grows exponentially with time as follows:
\begin{equation}\label{shadow-est1}
\big\|\big(\widetilde{p}(t),\widetilde{q}(t)\big)-\big(p(t),q(t)\big)\big\|=\mathcal{O}(e^{t\sigma}\hat\delta).
\end{equation}
A comprehensive consideration
of \eqref{shadow-est} and \eqref{shadow-est1}yields the following estimation
\begin{equation*}
\big\|(p_n,q_n)-\big(p(t),q(t)\big)\big\|=\mathcal{O}(e^{t\sigma}\hat\delta + \hat\delta ).
\end{equation*}

Usually, the shadowing distance is depending on $n$ and has the form
$\hat\delta=\mathcal{O}(\widehat{L}\widehat{\varepsilon})$, where
the ``magnification factor'' $\widehat{L}$ may be in a large
magnitude. It is noted that an accurate estimation of $\widehat{L}$
has not yet been obtained in the literature. As suggested by the
numerical experiments in the literature \cite{Chow1992} that
$\widehat{L}\propto t/h= n$, a possible estimation to the global
error is guessed of the form
\begin{equation}\label{shadow-est3}
\big\|(p_n,q_n)-\big(p(t),q(t)\big)\big\|=\mathcal{O}(te^{t\sigma}h^{r}),
\end{equation}
for $t\leq T_{Lyapunov}$, where $T_{Lyapunov}$ is the predictable
time (or the Lyapunov time) inversely proportional to the maximum Lyapunov
exponent $\sigma$. If $t>T_{Lyapunov}$, i.e., the time goes beyond
the predictable time, the rule of the global error growth cannot be
determined because of the unpredictability of chaotic motions in a
long-term scale, even though the duration of the shadowing could be
much larger than the Lyapunov time.

We finally emphasize that although symplectic integrators will
behave like generic methods that present an exponential error growth
for chaotic trajectories, one can also get benefit from the use of
symplectic integrators. As reported by Luo et al.
\cite{Luo2021,Luo2023} for the nonseparable PN Hamiltonian system
of compact objects, the application of the  symplectic
extended phase space symplectic integrator to
chaotic trajectories may lead to a blow-up phenomenon such that the
energy errors grow suddenly and dramatically to interrupt the
computation. Due to the existence of additional first integrals
except for the energy, the authors suggested that the
revised methods by projection onto the submanifold defined by such
integrals could partially improve the performance. As analyzed in
Section~\ref{sec:stand-proj}, even though the complete degree of freedom of
the weight vector $\lambda$ makes the revised methods still
symplectic in both the extended phase space and the original phase
space, the existence of globally well-defined modified Hamiltonian
of the revised symplectic method becomes uncertain (we consider the
existence does not hold because the weight vector will be different
for each time step). Another key issue is that the improvement of
the revised methods on the error performance is at the expense of
solving of implicit equations to determine the weight vector
$\lambda$. Moreover, we consider that a small stepsize for the
single-factor explicit symplectic integrator could avoid the blow-up
phenomenon. Hence, there is a tradeoff between the proposed
single-factor or double-factor explicit symplectic integrator and
the revised multi-projection integrators in the numerical simulation
of chaotic trajectories.

\section{Numerical experiments}\label{sec:num}

This section describes results of applying our explicit
simplectic integrators described in Section \ref{sec:existence} and
Section \ref{sec:symp-projection} to several test problems. We test
our single-factor or double-factor explicit symplectic
integrator compared with the existing implicit symplectic
integrators. We consider two nonseparable Hamiltonian
problems, including a completely integrable Hamiltonian system and
the nonintegrable PN Hamiltonian system of the spinning compact
binaries. To present a detailed performance comparison, we select
the following symplectic integrators:
\begin{itemize}
  \item IRK2: the second-order implicit midpoint method;
  \item SemiSymp2: the second-order semiexplicit method proposed in \cite{Jayawardana2023};
  \item ExpSymp2:  the second-order double-factor explicit symplectic integrator
  with $(\lambda_0,\mu_0)=(\frac{1}{e},\frac{1}{\pi})$ proposed in this paper;
  \item  IRK4: the fourth-order symplectic Legendre--Gauss collocation method;
  \item SemiSymp4: the fourth-order semiexplicit method proposed in  \cite{Jayawardana2023};
  \item ExpSymp4:   the fourth-order double-factor explicit symplectic integrator
  with $(\lambda_0,\mu_0)=(\frac{1}{e},\frac{1}{\pi})$ proposed in this paper.
\end{itemize}
For the implicit methods, the error tolerance of the
iterative solution is given by $\varepsilon_{tol}=10^{-13}$ unless
otherwise specified. For the two implicit RK methods IRK2 and IRK4,
the fixed-point iteration is used as recommended in
\cite{Hairer2006}, while the simplified Newton iteration is used for
the two semiexplicit methods as discussed in \cite{Jayawardana2023}.

Throughout the numerical experiments, the two-norm is used to measure
the global error ($GE_n$) at time $t_n=nh $,
which is defined by the difference between
the numerical solutions and the exact or reference solutions, i.e.,
\begin{equation*}
GE_n =\big \|(p_n,q_n)-\big(p(t_n),q(t_n)\big)\big\|.
\end{equation*}
The global error of Hamiltonian energy is defined by
\begin{equation*}
GHE_n = |H(p_n,q_n) - H(p_0,q_0)|.
\end{equation*}
The global (energy) error GE (or GHE) over the integration time interval $[0,T_{end}]$ is set
to the maximum of $GE_n$, i.e.,
\begin{equation*}
GE=\max_{1\leq n\leq N_T}GE_{n},
\quad
GHE=\max_{1\leq n\leq N_T}GHE_{n}.
\end{equation*}
We carry out our numerical experiments using
MATLAB R2016a on Lenovo desktop Qitian M437 with 3.10
GHz CPU i5-10500.

\begin{problem}\label{problem-1}
We first consider the simple case of 1 degree of freedom system with
$H(p,q)=\frac{1}{2}(1+p^2)(1+q^2)$ and the initial condition
$\big(p(0),q(0)\big)=(0,-3)$, which is completely integrable
but nonseparable. This Hamiltonian system is solvable and considered
in \cite{Tao2016b} and \cite{Jayawardana2023}. The exact solution can be expressed
by Jacobian elliptic function.
\end{problem}

Figure~\ref{Fig1}(a) plots the exact solution of this
problem, which shows the integrability of the system. However, the
extended Hamiltonian $\Gamma(p,x,q,y)=\frac{1}{2}(1+p^2)(1+y^2)
+\frac{1}{2}(1+x^2)(1+q^2)$ is no longer integrable as demonstrated
by Figure~\ref{Fig1}(b), where 41 trajectories on the constant
energy surface $\Gamma=10$ are plotted. It can be observed from
Figure~\ref{Fig1}(b) that there exist two notable islands where the
orbit is regular (periodic or quasi-periodic), while there also
exists an ergodic sea in which the trajectory is chaotic. The
numerical performance of Pihajoki's original second-order extended
phase space integrator \eqref{Pihajoki-leapfrog} is displayed in
Figure~\ref{Fig1}(c), where the exponential growth of both the
global error and the discrepancy $\|(p_n,q_n)-(x_n,y_n)\|$ of the
numerical method are clearly observed.

The global errors and energy errors of the six tested symplectic
integrators with the stepsize $h=0.01$ and $T_{end}=1000$ are shown in Figure~\ref{Fig2}.
The linear growth of global errors and the uniform bound of energy errors
in this figure clearly demonstrate the theoretical result in
Theorem~\ref{theorem8}. In addition to the comparable accuracy of the same order
methods shown in Figure~\ref{Fig2}, the remarkable difference is their
computational efficiency. The CPU time are
4.7731, 15.4917, 0.1397, 14.3253, 8.7199, and 0.1507 seconds respectively
for IRK2, SemiSymp2, ExpSymp2, IRK4, SemiSymp4, and ExpSymp4. This soundly
supports the much higher efficiency of the explicit symplectic integrators
proposed in this paper.

\begin{figure*}[htb]
\center{
\subfigure[Exact solutions]
{\includegraphics[scale=0.36]{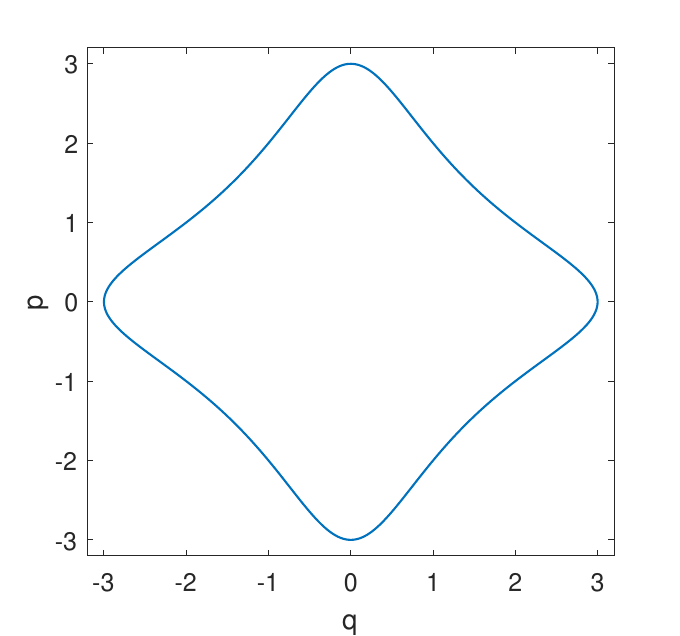}}
\subfigure[Poincar\'{e} sections]
{\includegraphics[scale=0.36]{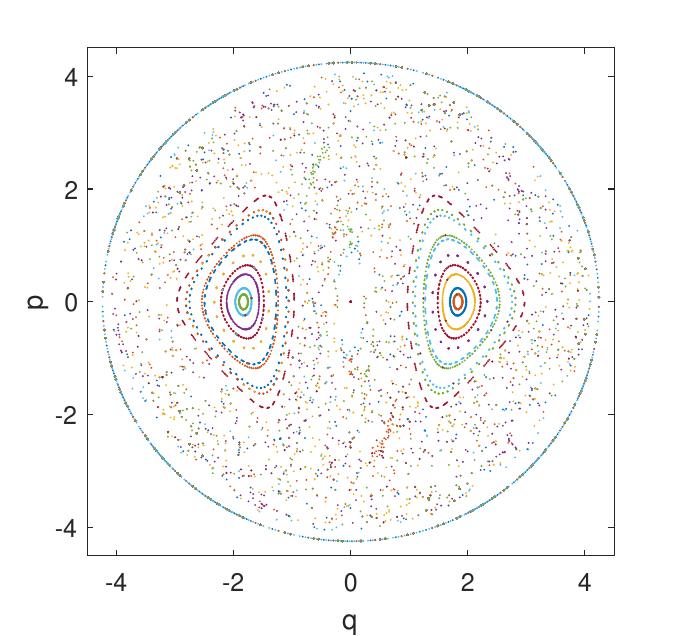}}
\subfigure[Error growth]
{\includegraphics[scale=0.36]{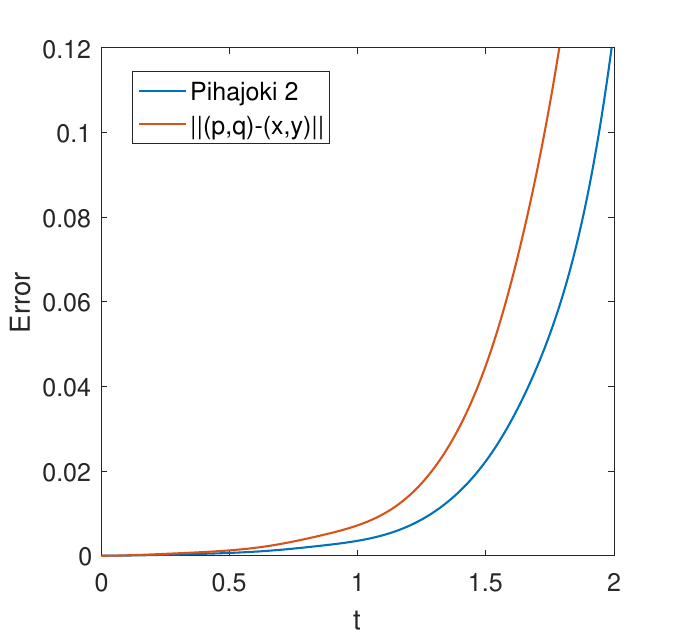}}
\caption{The exact solutions of the original Hamiltonian, the
Poincar\'{e} sections at the surface $x=0$ of the extended Hamiltonian,
and the exponential error growth of the 2nd-order Pihajoki's
original extended phase space integrator for Problem~\ref{problem-1}.}
\label{Fig1}}
\end{figure*}

\begin{figure*}[htb]
\centering{
\includegraphics[scale=0.50]{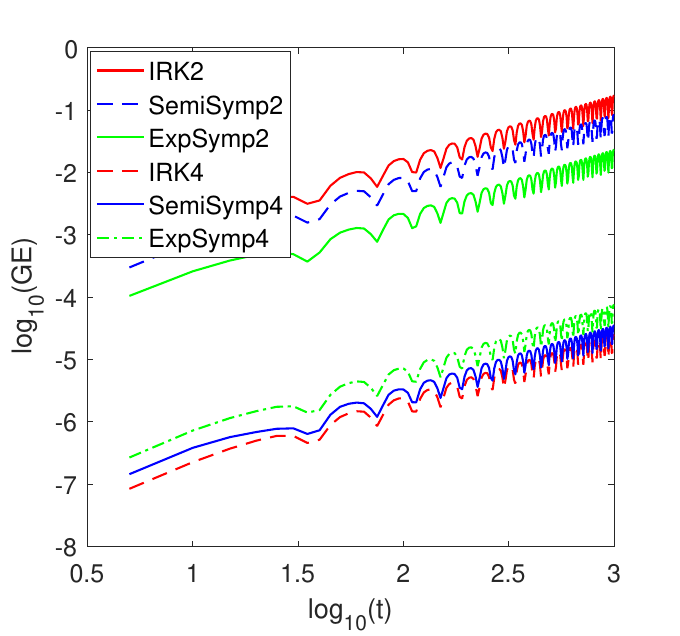}
\includegraphics[scale=0.50]{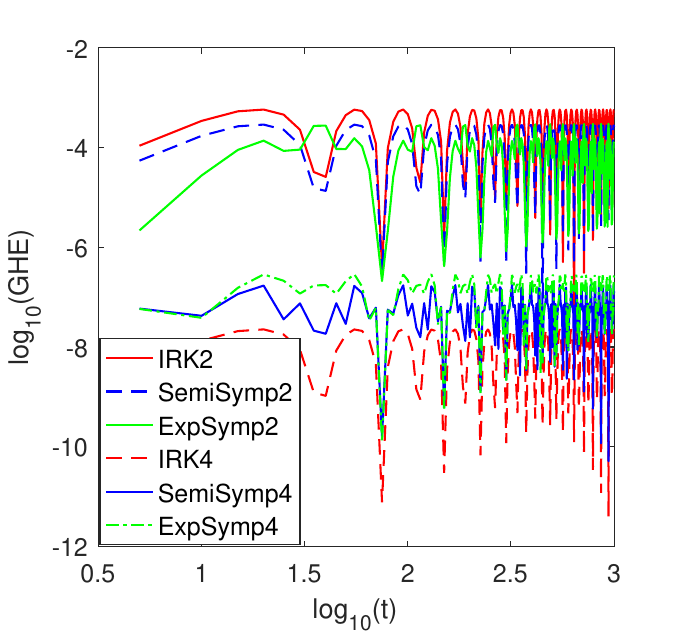}
}
\caption{\label{Fig2}The global errors and
energy errors with $h=0.01$.}
\end{figure*}

\begin{figure*}[htb]
\centering{
\includegraphics[scale=0.50]{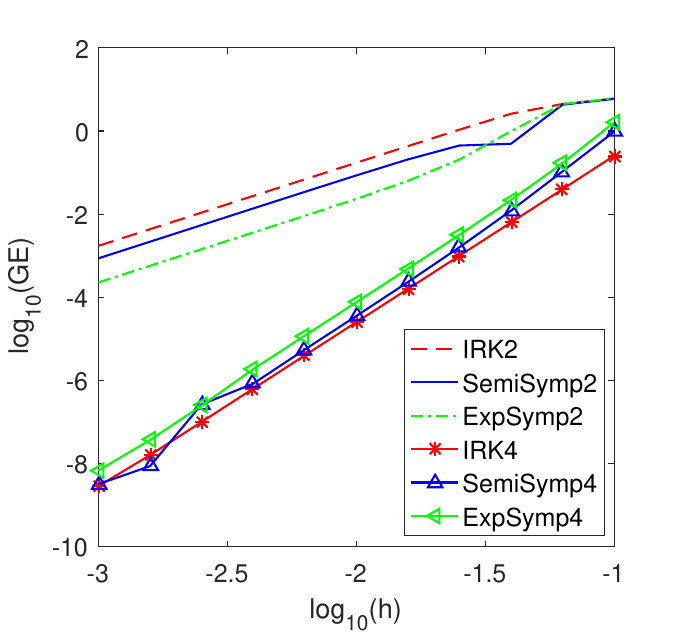}
\includegraphics[scale=0.50]{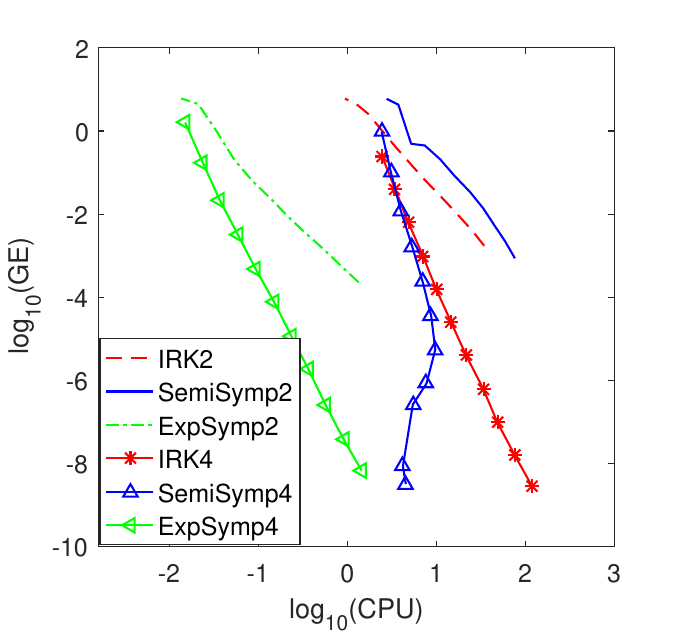}
}
\caption{\label{Fig3}The convergence orders and the
efficiency curves.}
\end{figure*}

We further show the convergence orders and efficiency curves of the
tested symplectic integrators with $T_{end}=1000$ in
Figure~\ref{Fig3}. It is observed that ExpSymp2 has the same second
order as IRK2 and SemiSymp2, while ExpSymp4 has the same fourth
order as IRK4 and SemiSymp4. The right panel clearly shows the
higher efficiency of the proposed explicit symplectic integrators,
i.e., ExpSymp2 and ExpSymp4, compared with the same order implicit
methods. Besides, a noticeable phenomenon is that the CPU time (in
seconds) consumed by the fourth-order method SemiSymp4 does not
always increase with the decrease of the stepsize, which is
different from the other implicit methods IRK2, SemiSymp2, and IRK4.
This phenomenon is also consistent with the result in
\cite{Jayawardana2023} that the smaller stepsize indicates the
smaller number of iterations in a single step. As claimed in
\cite{Jayawardana2023}, this phenomenon is attributed to that the
difference $\|(p_{n+1}^{(1)},q_{n+1}^{(1)})-
(p_{n+1}^{(1)},q_{n+1}^{(1)})\|$ of SemiSymp4 is proportional to
$\mathcal{O}(h^5)$. In particular, once the tolerance of the
iteration satisfies $\varepsilon_{tol}>\mathcal{O}(h^5)$, SemiSymp4 just needs
the least one iteration in a single step. In fact, this phenomenon
can also occur for SemiSymp2 just by setting a large
$\varepsilon_{tol}$ and a small $h$ such that
$\varepsilon_{tol}>\mathcal{O}(h^3)$. However, it is
important to emphasize two points  here. Firstly, the $r$th-order
semiexplicit symplectic integrator cannot be more efficient than the
same order explicit symplectic integrator proposed in this paper,
because the latter needs ``one iteration'' without regard to the
values of $h$ and $\varepsilon_{tol}$, while the number of
iterations for the former in a single step is certainly not less
than 1 and is exactly 1 merely for the special case
$\varepsilon_{tol}>\mathcal{O}(h^{r+1})$. Secondly, a large
$\varepsilon_{tol}$ not only destroys the symplecticity of the
semiexplicit integrator but also introduces a long-term tolerance
error for the method.

As analyzed in Section~\ref{sec:modifiedH}, the two factors
$\lambda_0$ and $\mu_0$ of the double-factor explicit symplectic
integrator will  reduce to the single-factor integrator once
$\lambda_0=\mu_0$. Here, we further explore the dependence of the
fourth-order double-factor explicit symplectic integrator ExpSymp4
on the factors $\lambda_0$ and $\mu_0$ by testing the following sets
of values for $(\lambda_0,\mu_0)$ with $h=0.01$:
$(\frac{1}{e},\frac{1}{\pi})$, $(\frac{1}{2},\frac{1}{2})$,
$(\frac{1}{3},\frac{3}{4})$, and the random parameter $rand(2)
=(0.2752,0.0731)$ generated by MATLAB. It is shown in
Figure~\ref{Fig4} that despite the
small difference within about half a magnitude among the four cases,
all the numerical results indicate a linear global error growth and
a uniformly bounded energy error.

Theorem~\ref{theorem8} only supports the linear growth of global
errors of the double-factor or single-factor explicit symplectic
integrator which is a special case of the standard projection
integrator in Definition~\ref{definition2}. As we analyzed, although
the standard projection integrator is symplectic in a global sense
regardless of the value of the weight vector $\lambda$, the linear
growth of global errors does not hold for general choices of
$\lambda$ once the globally well-defined modified Hamiltonian does
not exist. To illustrate this point, we finally test three different
choices for the fourth-order explicit symplectic integrator ExpSymp4
with $h=0.01$: choice 1--using a random parameter
$rand(2)=(0.6657,0.4910)$ and its reverse respectively for odd and
even steps (the double-factor explicit symplectic integrator with
a constant weight vector
$(\lambda_0,\mu_0)=(0.6657,0.4910)$), choice 2--using the same
parameter $(0.6657,0.4910)$ for all the steps (a standard projection
symplectic integrator with a constant weight vector), and choice
3--using different parameters $rand(2)$ randomly generated by MATLAB
and the reverse respectively for odd and even steps (the double-factor
explicit symplectic integrator with varying weight vectors
$(\lambda_0,\mu_0)$). It
can be observed from Figure~\ref{Fig5} that choice 1 performs like
classical symplectic methods that give a linear growth of global
errors and a uniform bound of energy errors. However, for the case
of choice 2, although the method is still globally symplectic, it
behaves like the traditional nonsymplectic method that possesses a
quadratic growth of global errors and a linear growth of energy
errors. The third case of choice 3 performs strangely as it does not
indicate an obvious power law growth of errors, and the errors are
generally between choice 1 and choice 2.

\begin{figure*}[htb]
\centering{
\includegraphics[scale=0.5]{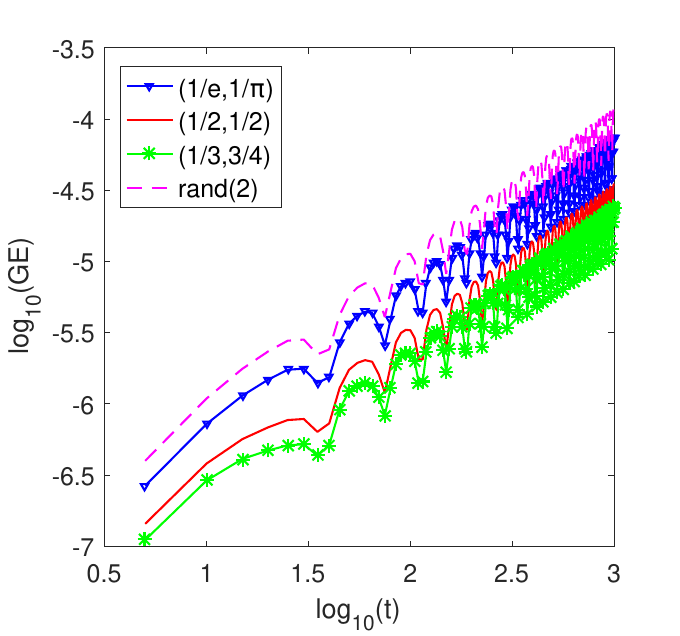}
\includegraphics[scale=0.5]{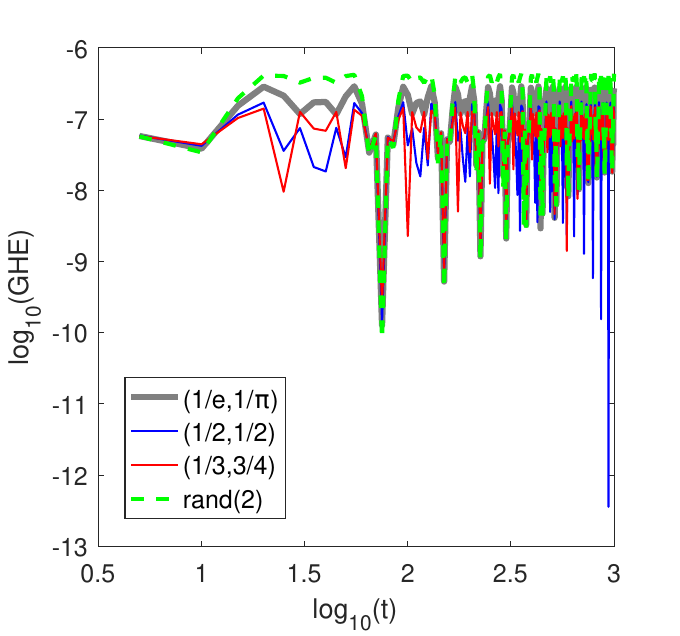}
}
\caption{\label{Fig4}Dependence of the fourth-order integrator
ExpSymp4 on the factors $(\lambda_0,\mu_0)$.}
\end{figure*}

\begin{figure*}[htb]
\centering{
\includegraphics[scale=0.5]{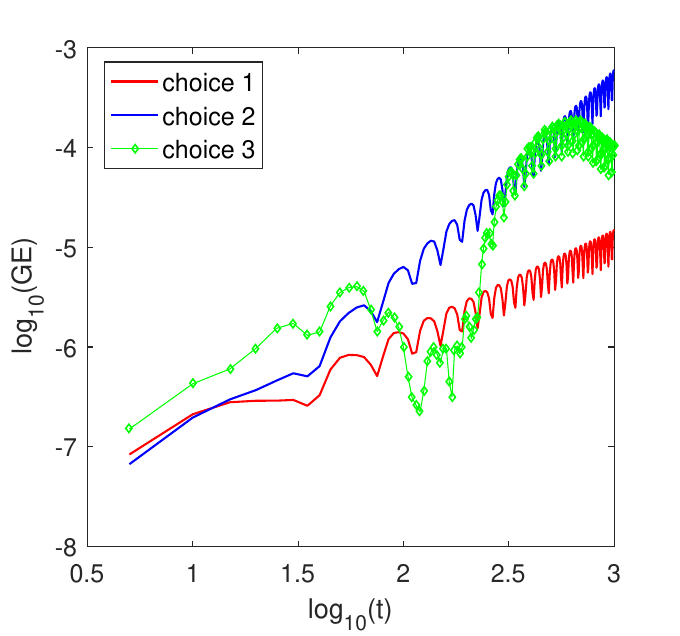}
\includegraphics[scale=0.5]{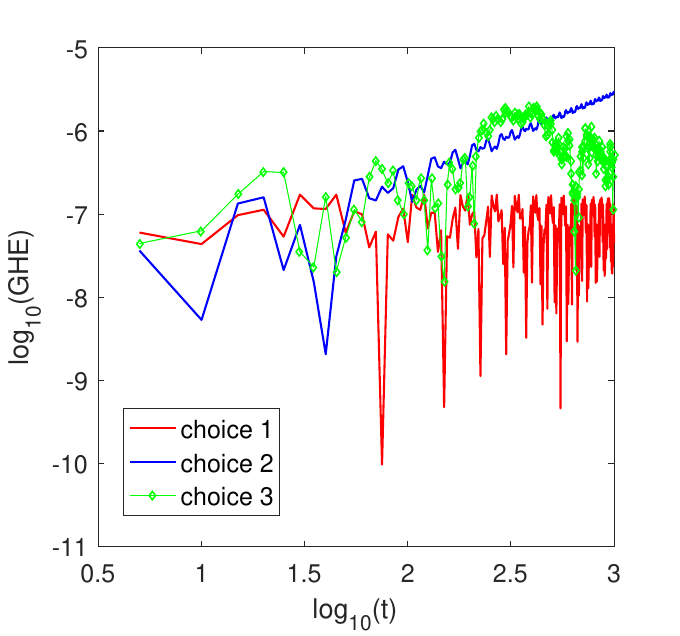}
}
\caption{\label{Fig5}The performance of different strategies
for the symplectic integrator.}
\end{figure*}

\begin{problem}\label{problem-2}
We now consider the post-Newtonian (PN) Hamiltonian system of the
spinning compact binaries \cite{Mei2018,Luo2023}, which is accurate up to 2PN order as follows:
\begin{equation}\label{formulation}
H(\bm{Q},\bm{P},\bm{S}_1,\bm{S}_2) =
H_{N} + \frac{1}{c^2}H_{1PN} + \frac{1}{c^4}H_{2PN}
+ \frac{1}{c^3}H_{1.5PN}^{SO}
+ \frac{1}{c^4}H_{2PN}^{SS},
\end{equation}
where $H_{N}$, $H_{1PN}$, and $H_{2PN}$ are respectively the Newtonian
term (i.e., the Kepler flow), the 1PN-, and 2PN-order orbital contributions to
the two compact bodies. The spin-orbit coupling, i.e., $H_{1.5PN}^{SO}$, is
accurate up to 1.5PN order by adopting the assumption that the rotational
speed of the compact object is in the same magnitude as $c$.
$H_{SS}$ is the spin-spin couplings accurate up to 2PN order.

In the formulation, $\bm{P}$ is the momenta of body 1 relative to
the center, $\bm{Q}$ is the position coordinates of body 1 relative
to body 2, $\bm{N}=\bm{Q}/r$ is the unit vector,
$r=|\bm{Q}|=\sqrt{q_1^2+q_2^2+q_3^2}$ is the distance of body 1
relative to body 2, and $\bm{S}_i~(i=1,2)$ are the spins of the two
compact bodies. The mass of the two compact bodies are respectively
denoted by $m_1$ and $m_2~(m_1\leq m_2)$. To describe the
formulation in detail, we additionally introduce the total mass
$M=m_1+m_2$, the reduced mass $\mu=m_{1}m_{2}/M$, the mass ratio
$\beta=m_1/m_2$, and the frequently used parameter
$\eta=\mu/M=\beta/(1+\beta)^{2}$. It is also noted that $\bm{S}_i$
for $i=1,2$ are usually expressed by
$\bm{S}_i=\Lambda_i\hat{\bm{S}}_i$, where $\hat{\bm{S}}_i$ are the
unit vectors,  and $\Lambda_i=\chi_{i}m_{i}^{2}/M^{2}$ ($\chi_i\in
[0, 1]$) are spin magnitudes. The time $t$, space $\bm{Q}$, momentum
$\bm{P}$, and spin $\bm{S}_i$ are respectively rescaled and measured
in $GM$, $GM$, $\mu$, and $M^2$ \cite{Huang2019}.

With the above notations, the detailed expressions are as follows:
\begin{eqnarray*}
&H_{N}&=\frac{\bm{P}^2}{2}-\frac{1}{r},
\\
  &H_{1PN}&=\frac{1}{8}(3\eta-1)\bm{P}^4-\frac{1}{2}\big[(3+\eta)\bm{P}^2
+\eta(\bm{N}\cdot\bm{P})^2\big]\frac{1}{r}
+\frac{1}{2r^{2}},
\\
&H_{2PN}&=
\frac{1}{16}(1-5\eta+5\eta^2)\bm{P}^6
+\frac{1}{8}\big[(5-20\eta-3\eta^2)\bm{P}^4
 -2\eta^2{(\bm{N}\cdot\bm{P})^2}\bm{P}^2
\\
&&-3\eta^2{(\bm{N}\cdot\bm{P})^4}\big]\frac{1}{r}
+\frac{1}{2}\big[(5+8\eta)\bm{P}^2
+3\eta (\bm{N}\cdot\bm{P})^2\big]\frac{1}{r^2}
-\frac{1}{4}(1+3\eta)\frac{1}{r^3},
 \\
& H_{1.5PN}^{SO} &= \frac{1}{r^{3}} \big(2\bm{S}
+\frac{3}{2}\bm{S}^{*}\big)\cdot\bm{L},
\end{eqnarray*}
and
\begin{equation*}
H_{2PN}^{SS}(\bm{Q},\bm{S}_1,\bm{S}_2)=\frac{1}{2r^3}
\big[3(\bm{S}_0\cdot\bm{N})^2-\bm{S}_0^2\big],
\end{equation*}
where $\bm{S}= \bm{S}_1 + \bm{S}_2$,
$\bm{S}^{*}=\frac{1}{\beta}\bm{S}_1+\beta\bm{S}_2$,
$\bm{L}$ is the orbital angular momentum vector
$\bm{L}=\bm{Q} \times \bm{P}$, and
\begin{equation*}
\bm{S}_{0} = \bm{S} + \bm{S}^{*}
= (1+\frac{1}{\beta})\bm{S}_1 +(1+\beta)\bm{S}_2.
\end{equation*}

With the canonical conjugate spin variables
$\bm{\theta}=(\theta_1,\theta_2)$ and $\bm{\xi}=(\xi_1,\xi_2)$ and
the transformation
\begin{equation*}\label{spin-new-v}
\bm{S}_i=\left(
\begin{aligned}
&\rho_{i}\cos(\theta_{i})
\\
&\rho_{i}\sin(\theta_{i})
\\
&\quad \xi_{i}
\end{aligned}
\right),\quad i=1,2,
\end{equation*}
where $\rho_{i}^2+\xi_{i}^2=\Lambda_i^2$ and $\rho_{i}>0$,
the original 12-dimensional system
$H(\bm{Q},\bm{P},\bm{S}_1,\bm{S}_2)$ of \eqref{formulation}
reduces to a 10-dimensional canonical Hamiltonian system
$H(\bm{Q},\bm{P},\bm{\theta},\bm{\xi})$ \cite{Wu2010},
whose evolution equations read
\begin{equation*}\label{eq:evolution-canonical}
\begin{aligned}
&\frac{\mathrm{d}\bm{Q}}{\mathrm{d}t}=\frac{\partial {H}}{\partial \bm{P}},
\qquad \frac{\mathrm{d}\bm{P}}{\mathrm{d}t}=-\frac{\partial {H}}{\partial \bm{Q}},
\\
& \frac{\mathrm{d}\bm{\theta}}{\mathrm{d}t}=\frac{\partial {H}}{\partial\bm{\xi}},
\qquad \frac{\mathrm{d}\bm{\xi}}{\mathrm{d}t}=-\frac{\partial {H}}{\partial\bm{\theta}}.
\end{aligned}
\end{equation*}

For the 10-dimensional canonical Hamiltonian system
$H(\bm{Q},\bm{P},\bm{\theta},\bm{\xi})$, there exist only four
independent first integrals, i.e., the total angular momentum vector
\begin{equation*}\label{angular}
\bm{J}=\bm{L}+\bm{S}_1+\bm{S}_2,
\end{equation*}
and the total energy
\begin{equation*}\label{energy}
E={H}(\bm{Q},\bm{P},\bm{\theta},\bm{\xi}).
\end{equation*}
This means that the spinning Hamiltonian \eqref{formulation}
is nonintegrable, and hence chaos may occur in this system.
In addition, the Hamiltonian ${H}(\bm{Q},\bm{P},\bm{\theta},\bm{\xi})$
is certainly nonseparable.
\end{problem}

For this problem, the following two trajectories with
different initial values and parameters are tested: trajectory 1
with {$\bm{Q}(0)=(25.34,0,0)$, $\bm{P}(0)=(0,0.18,0)$},
$(\Lambda_1,\Lambda_2)=(0.0479,0.6104)$,
$\bm{\theta}(0)=(1.2490,0.6202)$, $\bm{\xi}(0)=(0.0445,0.6104)$,
$\beta=0.28$, $c=10^{1/2}$ and trajectory 2 with
{$\bm{Q}(0)=(8.31,0,0)$, $\bm{P}(0)=(0,0.50,0)$},
$(\Lambda_1,\Lambda_2)=(0.25,0.25)$,
$\bm{\theta}(0)=(0.7587,0.8469)$, $\bm{\xi}(0)=(-0.2459,-0.2459)$,
$\beta=1$, $c=1$.

We first numerically integrate the two trajectories using the
8th-order Legendre--Gauss symplectic method with $h=1$ and
$T_{end}=10^5$. The maximum Lyapunov exponents shown in the left
panel of Figure~\ref{Fig6} indicate the regularity of trajectory 1
and the chaoticity of trajectory 2, as the former displays a linear
decrease while the latter seems to tend to a stable value of about
$10^{-3}$. The right panel of Figure~\ref{Fig6} displays the global
energy errors of the two different trajectories, from which it can
be seen that there only exists an accumulation of roundoff error for
the regular trajectory, while the energy error of the chaotic
trajectory suddenly increases at some moments. This means that the
regularity or chaoticity of the trajectory probably introduces a
significant influence on the performance of symplectic integrators.

\begin{figure}[hptb]
\center{
\subfigure[Lyapunov exponents]
{\includegraphics[scale=0.50]{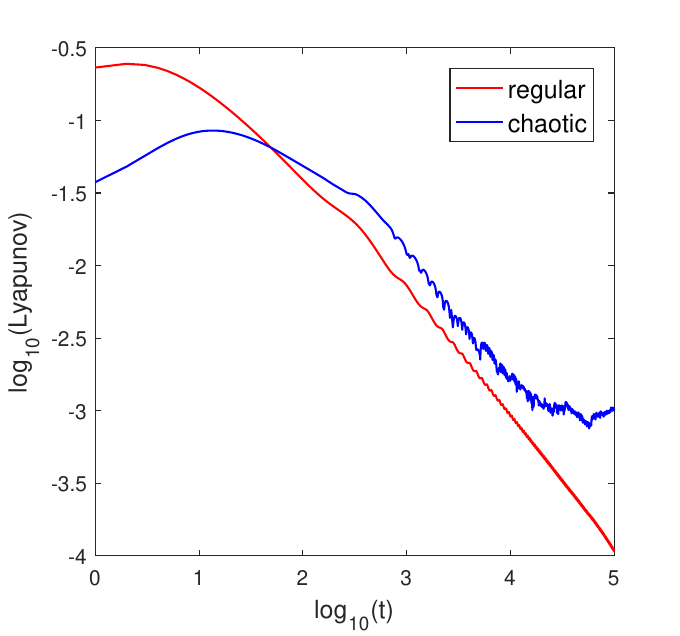}}
\subfigure[Global energy errors]
{\includegraphics[scale=0.50]{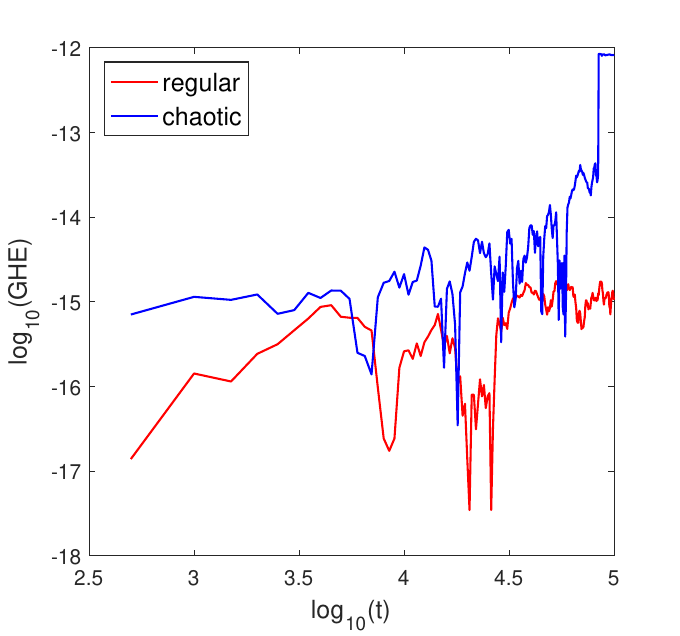}}
\caption{Lyapunov exponents of two different  trajectories and
the global energy errors of the reference solutions for Problem~\ref{problem-2}.}
\label{Fig6}}
\end{figure}

Although the reliability of numerical solutions for the chaotic trajectory
2 is a subtle issue, we still regard the numerical solutions obtained from the
8th-order Legendre--Gauss symplectic method as the reference solutions due to the
high accuracy. Then, the global errors and energy errors of the
six tested symplectic integrators with $h=1$ are shown in Figure~\ref{Fig7}
and Figure~\ref{Fig8} respectively for the regular trajectory and the chaotic
trajectory. It can be seen from Figure~\ref{Fig7} that all the symplectic
integrators display a linear growth of global errors and a uniform bound of
energy errors, which support the theoretical analysis made in Theorem~\ref{theorem8}.

However, the case of the chaotic trajectory is complicated. As the
maximum Lyapunov is about $10^{-3}$, the predictable time (or
Lyapunov time) of the chaotic trajectory is inversely proportional to $10^{-3}$,
i.e., about $10^3$. It is observed from Figure~\ref{Fig8} that the
global errors of all the six tested symplectic integrators give rise
to an exponential growth and the energy errors seem to
have a quadratic growth once $t<10^{2.3}$. However, for the case of
$t>10^{2.3}$, both the global errors and the energy errors do not
exhibit a clear growth rule but increase in general. This result
strongly supports our analysis in Section~\ref{sec:ergodic} as the
inflection time $10^{2.3}$ in Figure~\ref{Fig8} is close to the
Lyapunov time $10^{3}$. Moreover, the reference function displayed
in the left panel of Figure~\ref{Fig8} gives a strong support for
the global error estimation \eqref{shadow-est3}.

The CPU time (in seconds) consumed by each method is presented in
Table~\ref{Tab:one}. We first claim that the error tolerance
of iterations for the
chaotic trajectory is set to $\varepsilon^*_{tol}=4\times10^{-13}$,
which is three times larger than the originally used setting
$\varepsilon_{tol}$ because the semiexplicit methods do not converge
under the latter setting. It is very clear from
Table~\ref{Tab:one} that the explicit symplectic methods have much
higher efficiency than the same order implicit symplectic methods.
In addition, the semiexplict methods consume more time for the
chaotic trajectory than for the regular trajectory, while the
explicit methods and the fully implicit methods are hardly affected
by the regularity or the chaoticity of the trajectory.

\begin{figure*}[htb]
\centering{
\includegraphics[scale=0.5]{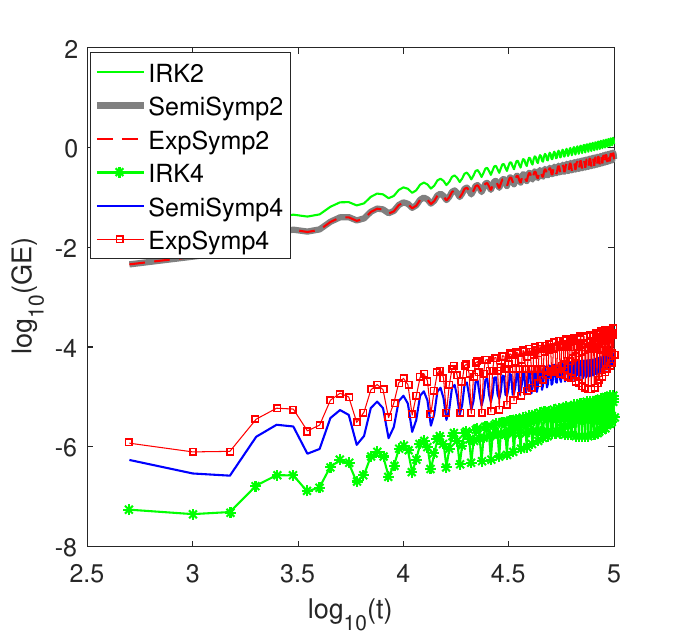}
\includegraphics[scale=0.5]{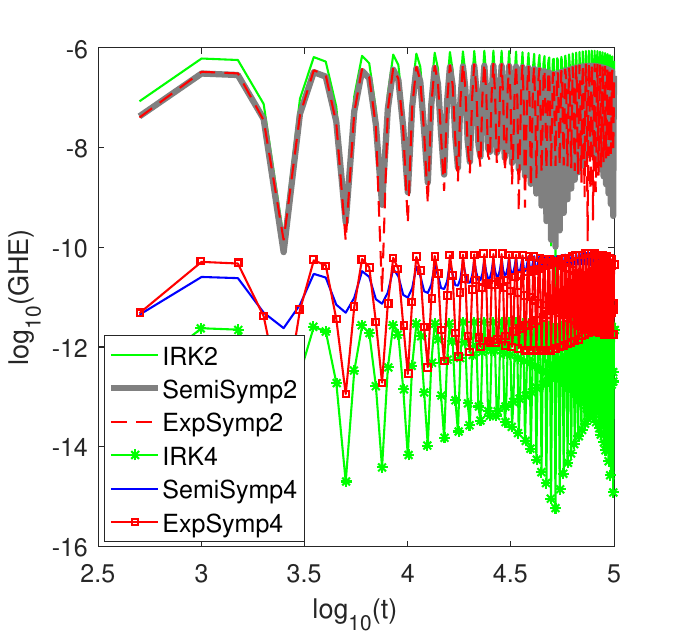}
}
\caption{\label{Fig7}The global errors and
global energy errors of the regular trajectory.}
\end{figure*}

\begin{figure*}[htb]
\centering{
\includegraphics[scale=0.5]{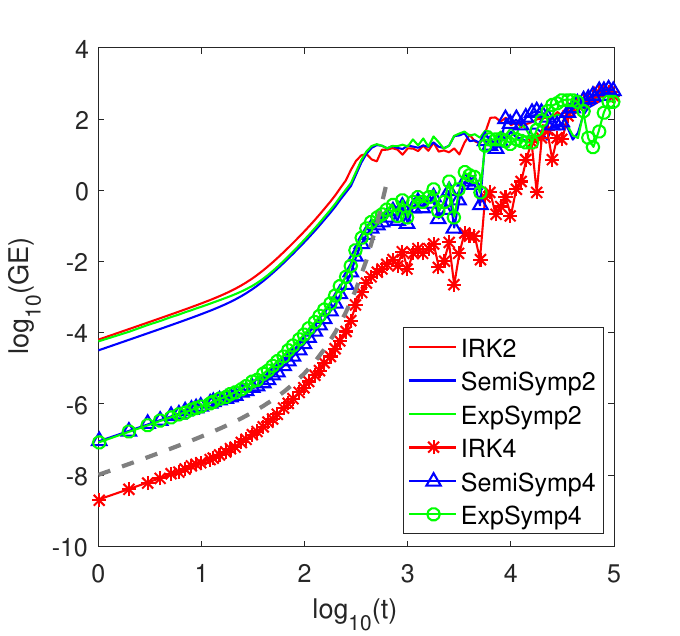}
\includegraphics[scale=0.5]{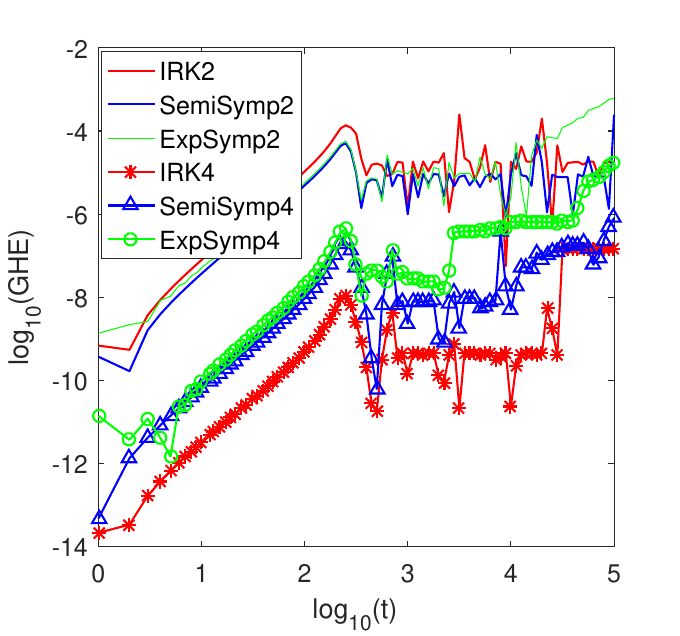}
} \caption{\label{Fig8}The global errors and global energy errors of
the chaotic trajectory. The gray dashed line in the left panel
denotes a reference function of $t$ as $t\times
e^{3.98\times10^{-2.3}t}\times 10^{-8}$.}
\end{figure*}

\begin{table}[htb]
\renewcommand{\arraystretch}{1.5}
\centering
\caption{CPU time (in seconds) of each method for the two different
trajectories with $T _{end}= 10^5$.}
\begin{tabular}{c|c|c|c|c|c|c}
  \hline
   & IRK2 & SemiSymp2  & ExpSymp2 & IRK4 & SemiSymp4  & ExpSymp4
   \\
   \hline
   regular & 45.4903  & 557.1124 & 19.2260 & 111.2681  & 365.5951 & 42.4972
   \\
    \hline
   chaotic & 48.5698 & 684.1673 & 18.9419 & 113.7709 & 406.3585 & 41.5335
   \\
  \hline
\end{tabular}
 \label{Tab:one}
\end{table}

\section{Conclusions}
The idea that discretization systems should preserve as many as
possible  the mathematical or physical structures of the original
continuous systems has been widely acknowledged and the research on
structure-preserving integrators becomes very active in the field of
numerical analysis and scientific and engineering computing.
Symplectic methods have been extensively used in many
areas of applied mathematics and natural sciences.  During the
development of structure-preserving integrators, the existence of
explicit symplectic integrators for general nonseparable Hamiltonian
systems is an open problem that had remained unsolved for decades.

This paper hinges upon three developments. Firstly, we
have proved the existence of explicit symplectic integrators for
general nonseparable Hamiltonian systems in Section
\ref{sec:existence}. Secondly, we have used the nonlinear symplectic
translation mapping to investigate explicit standard projection
symplectic integrators and  generalized projection symplectic
integrators in Section \ref{sec:symp-projection}. This has motivated
to show the existence of the global modified Hamiltonian
and the linear error growth for some
special explicit symplectic integrators  in Section~\ref{sec:error}.

In summary, we analyzed and derived several types of explicit
symplectic integrators by constructing linear and nonlinear explicit
symplectic transformations from the extended phase space to its
special submanifold. Moreover, by showing the existence of the
global modified Hamiltonian, we proved the linear growth of global
errors and the uniformly bounded error of first integrals for the
single-factor or double-factor explicit symplectic integrator when
applied to (near-) integrable Hamiltonian systems. We also presented
an estimation to the global error for the proposed explicit
integrators when applied respectively to regular and chaotic
trajectories of the nonintegrable Hamiltonian system. Finally, we
conducted the numerical experiments with a completely integrable
nonseparable Hamiltonian and a nonintegrable nonseparable
Hamiltonian, and the numerical results soundly demonstrated the
theoretical analysis made in this paper and supported the high
efficiency of the proposed explicit symplectic integrators
in comparison with the existing implicit symplectic
integrators in the literature.

%\newpage

\end{document}